\documentclass[a4paper,reqno,12pt]{amsart}
\usepackage[margin=1.2in]{geometry}
\usepackage{indentfirst}
\usepackage{amsthm}
\usepackage{amssymb}
\usepackage{amsmath}
\usepackage{thmtools}
\usepackage{mathtools}
\usepackage[disable]{todonotes}
\usepackage{enumitem}
\usepackage{setspace}
\usepackage{hyperref}
\usepackage{tikz}
\usepackage{tikz-cd}
\usetikzlibrary{matrix,arrows,positioning}
\usetikzlibrary{calc,arrows}
\usetikzlibrary{decorations.markings,shapes.geometric}
\usepackage{xypic}
\usepackage{upgreek}
\usepackage{extarrows}

\makeatletter
\providecommand\@dotsep{5}
\renewcommand{\listoftodos}[1][\@todonotes@todolistname]{%
  \@starttoc{tdo}{#1}}
\makeatother

\newtheorem{Lem}{Lemma}[section]
\newtheorem{Prop}[Lem]{Proposition}
\newtheorem*{Def}{Definition}

\theoremstyle{plain}
\newtheorem{Thm}[Lem]{Theorem}
\newtheorem{Cor}[Lem]{Corollary}

\theoremstyle{definition}
\declaretheorem[numbered=no,name=Example,qed={\lower-0.3ex\hbox{$\triangleleft$}}]{Ex}
\newtheorem*{Rem}{Remark}

\newcommand{\Hom}{\text{\textnormal{Hom}}}
\newcommand{\End}{\text{\textnormal{End}}}
\newcommand{\Aut}{\text{\textnormal{Aut}}}

\newcommand{\Mor}{\text{\textnormal{Mor}}}

\newcommand{\git}{/\!\!/}
\newcommand{\Tr}{\mathbb{T}\text{\textnormal{r}}}
\newcommand{\tr}{\text{\textnormal{tr}}}

\mathchardef\mhyphen="2D
\newcommand{\Obj}{\text{\textnormal{Obj}}}
\newcommand{\pt}{\text{\textnormal{pt}}}
\newcommand{\Ind}{\text{\textnormal{Ind}}}
\newcommand{\HInd}{\text{\textnormal{HInd}}}
\newcommand{\RInd}{\text{\textnormal{RInd}}}
\newcommand{\Res}{\text{\textnormal{Res}}}
\newcommand{\Span}{\text{\textnormal{Span}}}
\newcommand{\refl}{\text{\textnormal{ref}}}
\newcommand{\ev}{\text{\textnormal{ev}}}
\newcommand{\op}{\text{\textnormal{op}}}
\newcommand{\co}{\text{\textnormal{co}}}
\newcommand{\Ad}{\text{\textnormal{Ad}}}

\def\itemNum$#1${\item $\displaystyle#1$
   \hfill\refstepcounter{equation}(\theequation)}
\def\pser#1{[\![#1]\!]} 

\tikzset{
    partial ellipse/.style args={#1:#2:#3}{
        insert path={+ (#1:#3) arc (#1:#2:#3)}
    }
}

\makeatletter
\newcommand{\xRrightarrow}[2][]{\ext@arrow 0359\Rrightarrowfill@{#1}{#2}}
\newcommand{\Rrightarrowfill@}{\arrowfill@\equiv\equiv\Rrightarrow}
\makeatother

\begin{document}

\title[Real $2$-representation theory]{Real representation theory of finite categorical groups}

\author[M.\,B. Young]{Matthew B. Young}
\address{Max Planck Institute for Mathematics\\
Vivatsgasse 7\\
53111 Bonn, Germany}
\email{myoung@mpim-bonn.mpg.de \\ young.matthew.b@gmail.com}

\date{\today}


\begin{abstract}
We introduce and develop a categorification of the theory of Real representations of finite groups. In particular, we generalize the categorical character theory of Ganter--Kapranov and Bartlett to the Real setting. Given a Real representation of a group $\mathsf{G}$, or more generally a finite categorical group, on a linear category, we associate a number, the modified secondary trace, to each graded commuting pair $(g, \omega) \in \mathsf{G} \times \hat{\mathsf{G}}$, where $\hat{\mathsf{G}}$ is the background Real structure on $\mathsf{G}$. This collection of numbers defines the Real $2$-character of the Real representation. We also define various forms of induction for Real representations of finite categorical groups and compute the result at the level of Real $2$-characters. We interpret results in Real categorical character theory in terms of geometric structures, namely gerbes, vector bundles and functions on iterated unoriented loop groupoids. This perspective naturally leads to connections with the representation theory of unoriented versions of the twisted Drinfeld double of $\mathsf{G}$ and with discrete torsion in $M$-theory with orientifold. We speculate on an interpretation of our results as a generalized  Hopkins--Kuhn--Ravenel-type character theory in Real equivariant homotopy theory.
\end{abstract}

\maketitle


\setcounter{footnote}{0}

\section*{Introduction}
\addtocontents{toc}{\protect\setcounter{tocdepth}{1}}

Let $\mathsf{G}$ be a finite group. The complex representation theory of $\mathsf{G}$ is a classical and well-understood subject. In this paper we are interested in two variations of this theory. The first variation, also classical and well-understood, is the real representation theory of $\mathsf{G}$. More generally, following Atiyah--Segal \cite{atiyah1969} and Karoubi \cite{karoubi1970}, after fixing a short exact sequence of finite groups
\[
1 \rightarrow \mathsf{G} \rightarrow \hat{\mathsf{G}} \xrightarrow[]{\pi} \mathbb{Z}_2 \rightarrow 1
\]
we can consider the Real representation theory of $\mathsf{G}$. By a Real representation we mean a complex vector space together with an action of $\hat{\mathsf{G}}$ in which elements of $\mathsf{G}$ act complex linearly and elements of $\hat{\mathsf{G}} \backslash \mathsf{G}$ act complex antilinearly. The second, more recently studied variation is the $2$-representation theory of $\mathsf{G}$, in which $\mathsf{G}$ acts by autoequivalences of an object of a bicategory. More generally, the group $\mathsf{G}$ itself can be categorified, for example by twisting by a cohomology class $\alpha \in H^3(B \mathsf{G}, \mathbb{C}^{\times})$, leading to the representation theory of finite categorical (or weak $2$-) groups. The representation theory of categorical groups has been studied by many authors; most relevant to the present paper are the works of Elgueta \cite{elgueta2007}, Ganter--Kapranov \cite{ganter2008} and Bartlett \cite{bartlett2011}. The goal of this paper is to introduce and develop the Real representation theory of finite categorical groups, the theory obtained by considering both of the above variations simultaneously.

Apart from its intrinsic importance, the representation theory of categorical (and higher) groups has been studied because of its connections to other areas of mathematics and physics. For example, the works of Ganter--Kapranov \cite{ganter2008} and Bartlett \cite{bartlett2011} were motivated by connections to equivariant homotopy theory and oriented topological field theory, respectively. In both of these examples, the connections are strengthened by considering the categorified character theory of $2$-representations. Related appearances of higher categorical traces in geometry and representation theory can be found in the works of Ben-Zvi--Nadler \cite{benzvi2013} and Hoyois--Scherotzke--Sibilla \cite{hoyois2017}, for example. Analogously, we expect Real $2$-representations, and the resulting categorified Real character theory which we develop in this paper, to be related to certain $\mathbb{Z}_2$-equivariant refinements of the connections appearing in $2$-representation theory. In the above two examples, we have in mind Real equivariant homotopy theory and unoriented topological field theory.

In the remainder of this introduction we will explain our main results. For simplicity, we restrict attention to Real $2$-representations of finite groups on categories, leaving the general case to the body of the paper. A Real $2$-representation of $\mathsf{G}$ on a category $\mathcal{C}$ is the data of autoequivalences
\[
\rho(g): \mathcal{C} \rightarrow \mathcal{C}, \qquad g \in \mathsf{G}
\]
and antiautoequivalences
\[
\rho(\omega) : \mathcal{C}^{\op} \rightarrow \mathcal{C}, \qquad \omega \in \hat{\mathsf{G}} \backslash \mathsf{G}
\]
together with coherently associative natural isomorphisms encoding their compositions. Here $\mathcal{C}^{\op}$ is the category opposite to $\mathcal{C}$. Retaining only the information attached to $\mathsf{G}$ recovers the notion of a $2$-representation of $\mathsf{G}$ on $\mathcal{C}$. The above definition, which categorifies the Grothendieck--Witt approach to Real representation theory, admits a variation in which $\mathcal{C}$ is assumed to be complex linear, elements of $\mathsf{G}$ act by complex linear autoequivalences and elements of $\hat{\mathsf{G}} \backslash \mathsf{G}$ act by complex antilinear autoequivalences. This variation categorifies the standard approach to Real representation theory. In terms of the general theory, it is a matter of preference which categorification one uses, as all results of the paper hold in either approach.

Associated to an ordinary $2$-representation $\rho$ of $\mathsf{G}$ on $\mathcal{C}$ is a collection of sets of natural transformations,
\[
\Tr_{\rho}(g) = 2\Hom_{\mathsf{Cat}}(1_{\mathcal{C}}, \rho(g)), \qquad g \in \mathsf{G}.
\]
Ganter--Kapranov \cite{ganter2008} and Bartlett \cite{bartlett2011} categorified the conjugation invariance of the character of a representation by constructing a system of compatible bijections
\[
\beta_{g, h} : \Tr_{\rho}(g) \xrightarrow[]{\sim} \Tr_{\rho}(h g h^{-1}), \qquad g, h \in \mathsf{G}.
\]
The collection $\{\Tr_{\rho}(g)\}_{g \in \mathsf{G}}$, together with the bijections $\{\beta_{g, h}\}_{(g,h) \in \mathsf{G}^2}$, is called the categorical character of $\rho$. Suppose now that $\rho$ is a Real $2$-representation of $\mathsf{G}$. In this setting, we prove that the categorical character is enhanced to a Real categorical character, by which we mean a compatible system of bijections
\[
\beta_{g, \omega} : \Tr_{\rho}(g) \xrightarrow[]{\sim} \Tr_{\rho}(\omega g^{\pi(\omega)} \omega^{-1}), \qquad g \in \mathsf{G}, \, \omega \in \hat{\mathsf{G}}.
\]
Call a pair $(g,\omega) \in \mathsf{G} \times \hat{\mathsf{G}}$ graded commuting if the equality
\[
\omega g^{\pi(\omega)} = g \omega
\]
holds. When the Real $2$-representation is linear, Real conjugation invariance allows us to associate to each graded commuting pair $(g, \omega)$ an element of the ground field,
\[
\chi_{\rho}(g,\omega) = \tr_{\Tr_{\rho}(g)} (\beta_{g, \omega} ).
\]
We define the Real $2$-character of $\rho$ to be this collection of joint traces. The Real categorical character and Real $2$-character therefore contain strictly more information than the categorical character and $2$-character of the underlying $2$-representation. This is in contrast to the fact that, one category level down, the character of a Real representation is an ordinary character subject to additional constraints.

In geometric terms, the Real categorical character defines a flat vector bundle over the unoriented loop groupoid of $B \mathsf{G}$; see Theorem \ref{thm:catCharLoopGrpd}. Here $B \mathsf{G}$ is regarded as a double cover $B \mathsf{G} \rightarrow B \hat{\mathsf{G}}$ of groupoids. The unoriented loop groupoid of $B \mathsf{G}$ is then the quotient of the loop groupoid of $B \mathsf{G}$ by the simultaneous action of deck transformations and loop reflection. It is worth emphasizing that the Real categorical character is an ordinary, as opposed to Real, vector bundle; the Real information is encoded entirely in the base of the vector bundle. This allows us to apply techniques from the ordinary representation theory of groupoids to study Real $2$-representations. From this geometric point of view, the Real $2$-character of $\rho$ is simply the holonomy of the Real categorical character. Using this observation, we prove in Theorem \ref{thm:conjInv2Char} that Real $2$-characters are invariant under Real conjugation by $\hat{\mathsf{G}}$:
\[
\chi(g, \omega) = \chi(\sigma g^{\pi(\sigma)} \sigma^{-1}, \sigma \omega \sigma^{-1}),  \qquad \sigma \in \hat{\mathsf{G}}.
\]
Such functions on graded commuting pairs, which we call Real $2$-class functions on $\mathsf{G}$, and their twisted generalizations first appeared in \cite{mbyoung2018a}, where they were realized as characters of unoriented versions of twisted Drinfeld doubles of $\mathsf{G}$. Corollary \ref{cor:thickDDAction} explains the precise relationship between Real $2$-representations and unoriented twisted Drinfeld doubles.

We also study induction of Real $2$-representations, of which there are two forms. The first categorifies the Realification (or hyperbolic) functor from complex representation theory to Real representation theory; the second categorifies induction internal to Real representation theory. In both cases we compute the result of $2$-induction at the level of Real categorical characters and $2$-characters; see Theorems \ref{thm:Real2RealIndFunct}, \ref{thm:RInd2Char} and \ref{thm:RealIndFunct}, \ref{thm:HInd2Char}. For example, given a subgroup $\hat{\mathsf{H}} \leq \hat{\mathsf{G}}$ compatible with the structure maps to $\mathbb{Z}_2$ and a Real $2$-representation $\rho$ of $\mathsf{H}$, the Real $2$-character of the induced Real $2$-representation is
\[
\chi_{\RInd_{\hat{\mathsf{H}}}^{\hat{\mathsf{G}}}(\rho)}(g, \omega) = \frac{1}{2 \vert \mathsf{H} \vert} \sum_{\substack{\sigma \in \hat{\mathsf{G}} \\ \sigma(g, \omega) \sigma^{-1} \in \hat{\mathsf{H}}^2}} \chi_{\rho}(\sigma g^{\pi(\sigma)} \sigma^{-1}, \sigma \omega \sigma^{-1}).
\]
Our results on $2$-induction are formulated so as to allow for immediate conjectural generalizations to Real representations of higher groups.

As mentioned above, in the body of the paper we work with finite categorical groups. In order to systematically deal with the cohomological twisting data used to define the categorical group we make use of the twisted loop transgression map introduced in \cite{mbyoung2018a}. Twisted loop transgression realizes a sort of dimensional reduction from Real $2$-representation theory to the twisted representation theory of unoriented loop groupoids and facilitates the geometric interpretation of our results.

As mentioned above, one motivation for Ganter and Kapranov \cite{ganter2008} to develop their $2$-character theory was to relate $2$-representation theory to higher chromatic phenomena in equivariant homotopy theory. To explain this, denote by $\mathbf{B} \mathsf{G}$ a classifying space of $\mathsf{G}$. Hopkins, Kuhn and Ravenel \cite{hopkins2000} showed that the Borel equivariant Morava $E$-theory group $E^{\bullet}_n(\mathbf{B} \mathsf{G})$, $n \geq 1$, at a prime $p$ admits a generalized character theoretic description. In this context, generalized characters are conjugation invariant functions on the set of commuting $n$-tuples in $\mathsf{G}$; the values of these functions and the $p$\textsuperscript{th} order condition on the commuting elements will be ignored in this introduction. When $n=1$ this recovers a $p$-completed version of the character theoretic description of $K^{\bullet}(\mathbf{B} \mathsf{G})$ given by Atiyah and Segal \cite{atiyah1969}. When $n=2$ this gives a generalized character theoretic description of the Borel equivariant elliptic cohomology group $E^{\bullet}_2(\mathbf{B} \mathsf{G})$. Ganter and Kapranov showed that the $2$-character theory of $\mathsf{G}$ also produces such generalized characters, although with values in the ground field and without the $p$\textsuperscript{th} order condition. To strengthen the analogy between $2$-representation theory and equivariant $E_2$-theory, Ganter and Kapranov showed that $2$-induction of $2$-representations is given at the level of $2$-characters by the same formula as transfer for Hopkins--Kuhn--Ravenel characters. This analogy persists in the twisted case. Indeed, twisted elliptic characters appear in both Devoto's approach to the twisted elliptic cohomology of $\mathbf{B} \mathsf{G}$ \cite{devoto1996} and, by combining the works of Ganter--Usher \cite{ganter2016} and Willerton \cite{willerton2008}, in the $2$-character theory of finite categorical groups. Moreover, we expect that $2$-induction at the level of $2$-characters of finite categorical groups, as computed in Corollary \ref{cor:ind2Char}, also describes transfer in equivariant elliptic cohomology. The analogy between $E^{\bullet}_3(\mathbf{B} \mathsf{G})$ and $3$-representations of $\mathsf{G}$ was established by Wang \cite{wang2015}.

In view of the above analogy, it is natural to expect that Real $2$-representation theory can be used to shed light on Real versions of Morava $E$-theory at the prime $p=2$. More precisely, we conjecture that there exists a real oriented generalized cohomology theory $R_n^{\bullet}$ such that $R_n^{\bullet}(\mathbf{B} \mathsf{G})$, with $\mathbf{B} \mathsf{G}$ viewed as a double cover $\mathbf{B} \mathsf{G} \rightarrow \mathbf{B} \hat{\mathsf{G}}$, can be described in terms of Real $n$-class functions of $\mathsf{G}$ which satisfy a $p$\textsuperscript{th} order condition. At height one, the group $R^{\bullet}_1(\mathbf{B} \mathsf{G})$ should reduce to Atiyah's Real $K$-theory $KR^{\bullet}
(\mathbf{B} \mathsf{G})$ localized at $p=2$ (see \cite{atiyah1966}), which, by the work of Atiyah--Segal \cite{atiyah1969}, is known to admit a character theoretic description. At height two, the group $R_2^{\bullet}(\mathbf{B} \mathsf{G})$, which we envisage as a Real equivariant elliptic cohomology theory, should be described as a space of Real $2$-class functions on $\mathsf{G}$. Moreover, at the level of Real $2$-class functions, our results on hyperbolic and Real $2$-induction should agree with the transfer maps from equivariant $E_2^{\bullet}$-theory to equivariant $R_2^{\bullet}$-theory and internal to equivariant $R_2^{\bullet}$-theory, respectively. We expect that $R_n^{\bullet}$ is closely related to the real oriented theory $E\mathbb{R}^{\bullet}_n$ constructed by Hu--Kriz \cite{hu2001}, but we do not establish a direct link in this paper.

Finally, we describe some applications of Real $2$-representation theory to mathematical physics, which we plan to return to in future work. Let $\mathcal{G}$ be the categorical group determined by a finite group $\mathsf{G}$ and a cocycle $\alpha \in Z^3(B \mathsf{G}, \mathbb{C}^{\times})$. A Real structure $\hat{\mathsf{G}}$ on $\mathsf{G}$ and a lift of $\alpha$ to $\hat{\alpha} \in Z^3(B \hat{\mathsf{G}}, \mathbb{C}^{\times}_{\pi})$ determine a Real structure on $\mathcal{G}$. On the other hand, the twisted cocycle group $Z^3(B \hat{\mathsf{G}}, \mathbb{C}^{\times}_{\pi})$ is known to be related to unoriented topological gauge theory. For example, a pair $(\hat{\mathsf{G}}, \hat{\alpha})$ is expected to define an unoriented lift of fully extended three dimensional $\alpha$-twisted oriented Dijkgraaf--Witten theory. The bicategory of Real $2$-representations of $\mathcal{G}$ on Kapranov--Voevodsky $2$-vector spaces can be obtained from the value of this theory on a point by taking homotopy fixed points for the action of $\pi_0(\mathsf{O}(3)) \simeq \mathbb{Z}_2$. In a related direction, the twisted $\hat{\mathsf{G}}$-equivariance of Real $2$-characters of $\mathcal{G}$ recovers Sharpe's expressions \cite{sharpe2011} for discrete torsion phase factors in $M$-theory with orientifolds. See \cite{mbyoung2018a} for a related perspective. We also mention that the relevance of twisted Real elliptic cohomology, and  $E\mathbb{R}_n^{\bullet}$-theory more generally, to string and $M$-theory with orientifolds has been conjectured by H. Sati. Real structures on categorical groups also appear in the theory of symmetry protected topological phases with symmetry categorical groups involving time reversal symmetry and in unoriented field theory with higher gauge field symmetries. See the papers of Kapustin--Thorngren \cite{kapustin2013} and Sharpe \cite{sharpe2015} for the appearance of categorical groups in the corresponding oriented settings. Finally, the prominence of twisted loop transgression in this paper is particularly natural from the point of view of field theory, where it becomes an instance of the `quantization via cohomological push-pull' procedure. Examples of this procedure in oriented settings can be found in the works of Freed \cite{freed1994} and Freed--Hopkins--Teleman \cite{freed2010b}.

A brief overview of the paper is as follows. Section \ref{sec:backgroundMat} collects background material. Section \ref{sec:RealRepFinGroup} contains relevant results from the twisted Real representation theory of finite groups in a form which is convenient for categorification. In Section \ref{sec:Real2Reps}, after defining Real structures on finite categorical groups, we introduce the notion of a Real representation of a finite categorical group. In Section \ref{sec:Real2Rep} we develop the basics of the Real categorical character theory of Real $2$-representations. Section \ref{sec:RealProjReps} then considers the general case of (linear) Real representations of finite categorical groups. We interpret the corresponding character theory in terms of vector bundles over gerbes on the unoriented loop groupoid of $B \mathsf{G}$. Section \ref{sec:indRealRep}, which serves as preparation for the following section, contains basic, but perhaps not widely available, material about Real and hyperbolic induction of twisted Real representations of finite groups. In Section \ref{sec:indReal2Rep} we introduce two forms of $2$-induction of Real $2$-representations. We compute the effect of $2$-induction at the level of Real categorical and $2$-characters. In Section \ref{sec:RealHKRTransfer} we describe our conjectural applications to Real equivariant homotopy theory.

\subsection*{Acknowledgements}
The author would like to thank Qingyuan Jiang, Hisham Sati and Nat Stapleton for discussions related to the content of this paper. Parts of this work were completed while the author was visiting National Tsing-Hua University. The author would like to thank Nan-Kuo Ho and Siye Wu for the invitation. The author was supported by a Direct Grant from the Chinese University of Hong Kong (Project No. CUHK4053289).

\section{Background material}
\label{sec:backgroundMat}
\addtocontents{toc}{\protect\setcounter{tocdepth}{2}}

\subsection{Bicategories}

We establish our notation for bicategories. For a detailed introduction to bicategories the reader is referred to \cite{benabou1967}.

A bicategory $\mathcal{V}$ consists of the following data:
\begin{enumerate}[label=(\roman*)]
\item A class $\Obj(\mathcal{V})$ of objects.
\item For each pair $x,y \in \Obj(\mathcal{V})$, a small category $1\Hom_{\mathcal{V}}(x,y)$, objects and morphisms of which are called $1$-morphisms and $2$-morphisms, respectively.
\item For each triple $x,y, z \in \Obj(\mathcal{V})$, a composition bifunctor
\[
- \circ_0 - : 
1\Hom_{\mathcal{V}}(y,z) \times 1\Hom_{\mathcal{V}}(x,y) \rightarrow 1\Hom_{\mathcal{V}}(x,z).
\]
\item For each $x \in \Obj(\mathcal{V})$, an identity $1$-morphism $1_x: x \rightarrow x$.
\item For each triple of composable $1$-morphisms $f,g,h$, an associator $2$-isomorphism
\[
\alpha_{f,g,h}: (f \circ_0 g) \circ_0 h \Longrightarrow f \circ_0 (g \circ_0 h).
\]
\item For each $1$-morphism $f: x\rightarrow y$, a pair of unitor $2$-isomorphisms
\[
\lambda_f: 
1_y \circ_0 f \Longrightarrow f, \qquad \rho_f: f \circ_0 1_x \Longrightarrow f.
\]
\end{enumerate}
The data (i)-(vi) is subject to a number of coherence conditions which we do not recall here.

Composition of $2$-morphisms within the same $1$-morphism category will be denoted by $-\circ_1-$. When it will not lead to confusion we will write $- \circ -$ in place of $- \circ_0 -$ or $-\circ_1 -$. The set of $2$-morphisms $\Hom_{1\Hom_{\mathcal{V}}(x,y)}(f, g)$ will be denoted by $2\Hom_{\mathcal{V}}(f, g)$. Given $1$-morphisms $f_1, f_2: x \rightarrow y$ and $g: y \rightarrow z$ and a $2$-morphism $u: f_1 \Rightarrow f_2$, the left whiskering of $u$ by $g$, namely $1_g \circ_0 u : g \circ_0 f_1 \Rightarrow g \circ_0 f_2$, will be denoted by $g \circ_0 u$. We adopt analogous notation for right whiskering. 

A (strict) $2$-category is a bicategory in which all associator $2$-isomorphisms $\alpha_{f,g,h}$ and all unitor $2$-isomorphisms $\lambda_f$, $\rho_f$ are identity maps. Coherence for bicategories asserts that any bicategory is biequivalent to a $2$-category.

\begin{Ex}
Small categories, functors and natural transformations form a $2$-category $\mathsf{Cat}$. For each field $k$, there is a sub-$2$-category $\mathsf{Cat}_k$ of $\mathsf{Cat}$ consisting of $k$-linear categories, $k$-linear functors and natural transformations.
\end{Ex}

\begin{Ex}
Let $k$ be a field. Kapranov and Voevodsky defined the bicategory $2\mathsf{Vect}_k$ of finite dimensional $2$-vector spaces over $k$ \cite{kapranov1994}. This is a $2$-categorical analogue of the category $\mathsf{Vect}_k$ of finite dimensional vector spaces over $k$. There are a number of standard variants of this bicategory. One definition takes $2\mathsf{Vect}_k$ to be the bicategory of $k$-linear additive finitely semisimple categories, $k$-linear functors and natural transformations. We will use the following slightly different model. Objects of $2\mathsf{Vect}_k$ are non-negative integers $[n]$, $n \in \mathbb{Z}_{\geq 0}$. A $1$-morphism $[n] \rightarrow [m]$ is an $m \times n$ matrix $A=(A_{ij})$ whose entries are finite dimensional vector spaces over $k$. The composition of the $1$-morphisms $A: [m] \rightarrow [n]$ and $B: [n] \rightarrow [p]$ is defined by
\[
(B \circ_0 A)_{ik} = \bigoplus_{j=1}^n B_{ij} \otimes_k A_{jk}.
\]
A $2$-morphism $u: A \Rightarrow B$ is a collection of $k$-linear maps $(u_{ij}: A_{ij} \rightarrow B_{ij})$. The compositions of $2$-morphisms are defined by
\[
(v \circ_0 u )_{ik} = \bigoplus_{j=1}^n v_{ij} \otimes u_{jk}, \qquad
(u^{\prime} \circ_1 u)_{ij} = u^{\prime}_{ij} \circ u_{ij}.
\]
The composition $-\circ_0 -$ is not strictly associative.
\end{Ex}

\subsection{Duality involutions on bicategories}
\label{sec:bicatDual}

We introduce the categorical background required for our formulation of Real $2$-representation theory. Our main approach uses contravariant involutions on categories and bicategories. An alternative approach, described in Sections \ref{sec:RealFGRep} and \ref{sec:antiLinearRealFGRep}, uses antilinear covariant involutions. These approaches are parallel and one can easily translate between the two.

We begin by recalling some basic categorical notions of duality. For further details the reader is referred to \cite[\S 3]{schlichting2010}. Given a category $\mathcal{C}$, we denote by $\mathcal{C}^{\op}$ its opposite category. A category with duality is then a category $\mathcal{C}$ together with a functor $(-)^*: \mathcal{C}^{\op} \rightarrow \mathcal{C}$ and a natural isomorphism $\Theta: 1_{\mathcal{C}} \Rightarrow (-)^* \circ ((-)^*)^{\op}$
which satisfy
\begin{equation}
\label{eq:catDualCompat}
\Theta^*_x \circ \Theta_{x^*} = 1_{x^*}
\end{equation}
for each $x \in \Obj(\mathcal{C})$. A morphism $(\mathcal{C},(-)^*,\Theta) \rightarrow (\mathcal{D}, (-)^*, \Xi)$ of categories with duality, sometimes called a form functor, consists of a functor $F: \mathcal{C} \rightarrow \mathcal{D}$ and a natural transformation $\varphi: F \circ (-)^* \Rightarrow (-)^* \circ F$
which satisfy
\[
\varphi^*_x \circ \Xi_{F(x)} = \varphi_{x^*} \circ F(\Theta_x)
\]
for each $x \in \Obj(\mathcal{C})$. A symmetric form in $(\mathcal{C}, (-)^*, \Theta)$ is an object $x \in \Obj(\mathcal{C})$ together with an isomorphism $\psi: x \rightarrow x^*$ which satisfies
\[
\psi^* \circ \Theta_x = \psi.
\]
Symmetric forms and their partial isometries, that is, morphisms $\phi: x \rightarrow y$ which satisfy $\phi^* \circ \psi_y \circ \phi = \psi_x$, define the homotopy fixed point\footnote{The term `homotopy fixed points' is usually reserved for covariant group actions, but we will use it more generally.} category $\mathcal{C}^{h \mathbb{Z}_2}$. Form functors induce functors between homotopy fixed point categories.

We now turn to the categorification of the above notions. We will use the $2$-cell dual $\mathcal{V}^{\co}$ of a bicategory $\mathcal{V}$, that is, the bicategory obtained from $\mathcal{V}$ by reversing its $2$-cells. Hence, if 
\[
\begin{tikzpicture}[baseline= (a).base]
\node[scale=0.9] (a) at (0,0){
\begin{tikzcd}[column sep=5em]
 x
  \arrow[bend left=35]{r}[name=U,below]{}{f} 
  \arrow[bend right=35]{r}[name=D]{}[swap]{g}
& 
y
     \arrow[shorten <=3pt,shorten >=3pt,Rightarrow,to path={(U) -- node[label=right:\footnotesize \scriptsize $u$] {} (D)}]{}
\end{tikzcd}
};
\end{tikzpicture}
\]
is a $2$-morphism in $\mathcal{V}$, then
\[
\begin{tikzpicture}[baseline= (a).base]
\node[scale=0.9] (a) at (0,0){
\begin{tikzcd}[column sep=5em]
 x
  \arrow[bend left=35]{r}[name=U,below]{}{f} 
  \arrow[bend right=35]{r}[name=D]{}[swap]{g}
& 
y
     \arrow[shorten <=3pt,shorten >=3pt,Rightarrow,to path={(D) -- node[label=right:\footnotesize \scriptsize $u$] {} (U)}]{}
\end{tikzcd}
};
\end{tikzpicture}
\]
is a $2$-morphism in $\mathcal{V}^{\co}$.

\begin{Def}[{\cite[Definition 2.1]{shulman2018}}]
A bicategory with weak duality involution is a bicategory $\mathcal{V}$ together with
\begin{enumerate}[label=(\roman*)]
\item a pseudofunctor $(-)^{\circ} : \mathcal{V}^{\co} \rightarrow \mathcal{V}$,
\item a pseudonatural adjoint equivalence $\eta: 1_{\mathcal{V}} \Rightarrow (-)^{\circ} \circ_0 ((-)^{\circ})^{\co}$, and
\item an invertible modification $\zeta: \eta \circ_0 (-)^{\circ} \xRrightarrow{\;\; \;\;} (-)^{\circ} \circ_0 \eta^{\co}$
\end{enumerate}
such that, for each $x \in \Obj(\mathcal{V})$, the equality
\begin{equation}
\label{eq:bicatDualCompat}
\zeta_{x^{\circ}} \circ_0 \eta_x = (\zeta_x^{\circ} \circ_0 \eta_x) \circ_1 \eta(x)
\end{equation}
of $2$-morphisms holds. Here $\eta(x) : \eta_{x^{\circ \circ}} \circ \eta_x \Rightarrow \eta_x^{\circ \circ} \circ \eta_x$ is a pseudonaturality constraint for $\eta$.
\end{Def}

If $\mathcal{V}$ is a $2$-category, $(-)^{\circ}$ is a strict $2$-functor and $\eta$ and $\zeta$ are the identities, then the above data is said to define a strict duality involution on $\mathcal{V}$.

\begin{Def}[{\cite[Definition 2.2]{shulman2018}}]
A duality pseudofunctor $(\mathcal{V},(-)^{\circ}, \eta) \rightarrow (\mathcal{W},(-)^{\circ}, \lambda)$ between bicategories with weak duality involution is a pseudofunctor $F: \mathcal{V} \rightarrow \mathcal{W}$ together with
\begin{enumerate}[label=(\roman*)]
\item a pseudonatural adjoint equivalence $\mathfrak{i} : (-)^{\circ} \circ_0 F^{\co} \Rightarrow F \circ_0 (-)^{\circ}$, and
\item an invertible modification
\[
\begin{gathered}
\begin{tikzcd}[column sep=1.5cm,row sep=1.5cm]
\mathcal{V} \arrow[r,"F" above] \arrow[d,"(-)^{\circ}" left] & \mathcal{W} \arrow[d,"(-)^{\circ}" right] \arrow[dd, bend left=70,"1_{\mathcal{W}}" {name=A, right}] \arrow[dl,Rightarrow,"\mathfrak{i}^{\co}"] \\
\mathcal{V}^{\co} \arrow[r,"F^{\co}" above] \arrow[d,"((-)^{\circ})^{\co}" left] & |[alias=B]| \mathcal{W}^{\co} \arrow[d,"((-)^{\circ})^{\co}" right] \arrow[dl,Rightarrow,"\mathfrak{i}"]\\
\mathcal{V} \arrow[r,"F" below] & \mathcal{W}
\arrow[Rightarrow, shorten <= 0.45em, shorten >= 0.15em, from=A, to=B, "\lambda" above]
\end{tikzcd}
\end{gathered}
\xRrightarrow{\;\; \theta \;\;}
\begin{gathered}
\begin{tikzcd}
\mathcal{V} \arrow[d,"((-)^{\circ})^{\co}" left] \arrow[dd,bend left=90,"1_{\mathcal{V}}" {name=A, right}]\\
|[alias=B]| \mathcal{V}^{\co} \arrow[d,"(-)^{\circ}" left]\\
\mathcal{V} \arrow[d,"F" left]\\
\mathcal{W}
\arrow[Rightarrow, shorten <= 0.55em, shorten >= 0.05em, from=A, to=B, "\eta" above]
\end{tikzcd}
\end{gathered}
\]
\end{enumerate}
such that, for each $x \in \Obj(\mathcal{V})$, a coherence constraint, which we omit, is satisfied.
\end{Def}

Coherence for bicategories generalizes to the present setting. Indeed, any bicategory with weak duality involution is biequivalent via a duality pseudofunctor to a $2$-category with strict duality involution \cite[Theorem 2.3]{shulman2018}.

\begin{Ex}
The strict $2$-functor $(-)^{\op} : \mathsf{Cat}^{\co} \rightarrow \mathsf{Cat}$ which sends categories, functors and natural transformations to their opposites is a strict duality involution. The restriction of $(-)^{\op}$ to $\mathsf{Cat}_k$ is a $k$-linear strict duality involution.
\end{Ex}

\begin{Ex}
The bicategory $2\mathsf{Vect}_k$ has a weak duality involution $(-)^{\vee}$ which is a $2$-categorical analogue of the $k$-linear duality functor on $\mathsf{Vect}_k$. On objects let $[n]^{\vee} = [n]$ and on $1$- and $2$-morphisms let $(-)^{\vee}$ be given by $k$-linear duality. Explicitly, we have $(A_{ij})^{\vee} = (A_{ij}^{\vee})$ and $(u_{ij})^{\vee} = (u_{ij}^{\vee})$. The adjoint equivalence $\eta$ is induced by $\ev$, the canonical evaluation isomorphism from a finite dimensional vector space to its double dual. The modification $\zeta$ is the identity.
\end{Ex}

Next, we define homotopy fixed point objects.

\begin{Def}
A symmetric form in a bicategory with weak duality involution $(\mathcal{V}, (-)^{\circ}, \eta, \zeta)$ is an object $x \in \Obj(\mathcal{V})$ together with 
\begin{enumerate}[label=(\roman*)]
\item an equivalence $\psi: x^{\circ} \rightarrow x$, and
\item a $2$-isomorphism $\mu: 1_x \Rightarrow \psi \circ_0 \psi^{\circ} \circ_0 \eta_x$
\end{enumerate}
such that the $2$-morphism given by the diagram
\[
\begin{gathered}
\begin{tikzcd}[column sep=1.5cm,row sep=1.25cm]
x^{\circ} \arrow{rrr}[name=A,below]{}{1_{x^{\circ}}} \arrow[bend left=30]{rd}[name=C,below]{}{\eta_x^{\circ}} \arrow[bend right=30]{rd}[name=D,below]{\eta_{x^{\circ}}}{} \arrow[rdd,bend right=30, "\psi"{name=G, below}]{}& & &x^{\circ} \arrow{r}[name=B,below]{}{\psi} & x \\
& x^{\circ \circ \circ} \arrow{r}[name=E,below]{}{\psi^{\circ \circ}} & x^{\circ \circ} \arrow{ru}[name=F,below]{}{\psi^{\circ}} & & \\
& x \arrow[ru,"\eta_x" below] \arrow[bend right=30]{rrruu}[name=I,below]{1_x}{}& & &
\arrow[Rightarrow, shorten <= 1.25em, shorten >= 0.35em, from=E, to=A, "\mu^{\circ}"]
\arrow[Rightarrow, from=D, to=C, "\zeta_x"]
\arrow[Rightarrow, shorten <= 0.95em, shorten >= 0.65em, from=I, to=F, "\mu"]
\arrow[Rightarrow, bend right=20, shorten <= 0.95em, shorten >= 0.65em, from=G, to=E, "\eta_{\psi}" below]
\end{tikzcd}
\end{gathered}
\]
is equal to $1_{\psi}$.
\end{Def}

Symmetric forms in $(\mathcal{V}, (-)^{\circ}, \eta, \zeta)$ are the objects of a homotopy fixed point bicategory $\mathcal{V}^{h \mathbb{Z}_2}$. As we will not have the occasion to use $1$- and $2$-morphisms of $\mathcal{V}^{h \mathbb{Z}_2}$, we omit their explicit definitions.

\begin{Ex}
The homotopy fixed point bicategory $\mathsf{Cat}^{h \mathbb{Z}_2}$, with respect to the duality involution $(-)^{\op}$, is the bicategory of categories with duality, form functors and natural transformations of form functors.
\end{Ex}

Given an object $x$ of a bicategory with weak duality involution, put
\begin{equation}
\label{eq:leftSupNot}
\prescript{\epsilon}{}{x}=
\begin{cases}
x & \mbox{if } \epsilon=1, \\
x^{\circ} & \mbox{if } \epsilon=-1.
\end{cases}
\end{equation}
Similar notation will be used for the action of $(-)^{\circ}$ on $1$- and $2$-morphisms.

Closely related to bicategories with duality involutions are bicategories with contravariance \cite[\S 4]{shulman2018}. Roughly speaking, these are bicategories which have both covariant and contravariant $1$-morphisms. More precisely, a bicategory with contravariance consists of the following data:
\begin{enumerate}[label=(\roman*)]
\item A class $\Obj(\mathcal{V})$ of objects.
\item For each pair $x,y \in \Obj(\mathcal{V})$ and each $\epsilon \in \{\pm 1\}$, a small category $1\Hom^{\epsilon}_{\mathcal{V}}(x,y)$.
\item For each triple $x,y, z \in \Obj(\mathcal{V})$ and each pair $\epsilon_1, \epsilon_2 \in \{\pm 1\}$, a composition bifunctor
\[
- \circ_0 - : 
1\Hom^{\epsilon_2}_{\mathcal{V}}(y,z) \times \prescript{\epsilon_2}{}{1\Hom^{\epsilon_1}_{\mathcal{V}}(x,y)} \rightarrow 1\Hom^{\epsilon_2 \epsilon_1}_{\mathcal{V}}(x,z).
\]
Here we apply equation \eqref{eq:leftSupNot} to $(\mathsf{Cat}, (-)^{\op})$, so that the left superscript $\epsilon_2$ indicates whether or not we consider opposite categories.
\item For each $x \in \Obj(\mathcal{V})$, an identity $1$-morphism $1_x \in \Hom_{\mathcal{V}}^1(x,x)$.
\item Associator and unitor $2$-isomorphisms.
\end{enumerate}
This data is required to satisfy coherence constraints similar to those of a bicategory.

By keeping only the data associated to $1 \in \{\pm 1\}$, each bicategory $\mathcal{V}$ with contravariance defines a bicategory $\mathcal{V}_1$.

A pseudofunctor preserving contravariance between bicategories with contravariance is defined as in the case of a pseudofunctor, but with the additional requirement that the sign $\epsilon \in \{\pm 1\}$ of $1$-morphisms be preserved. Similarly, one defines pseudonatural transformations respecting contravariance, the components of which are required to be $1$-morphisms of degree $+1$.

There is an obvious strictification of the above definition which leads to the notion of a $2$-category with contravariance. Any bicategory with contravariance is biequivalent via a pseudofunctor preserving contravariance to a $2$-category with contravariance \cite[Theorems 8.1, 8.2]{shulman2018}.

\begin{Ex}
\begin{enumerate}[label=(\roman*)]
\item A bicategory $\mathcal{V}$ with weak duality involution defines a bicategory $\underline{\mathcal{V}}$ with contravariance by setting
\[
\Obj(\underline{\mathcal{V}}) = \Obj(\mathcal{V}), \qquad 1\Hom^{\epsilon}_{\underline{\mathcal{V}}}(x,y) = 1\Hom_{\mathcal{V}}(\prescript{\epsilon}{}{x},y).
\]
Composition of $1$- and $2$-morphisms in $\underline{\mathcal{V}}$ is induced by the corresponding compositions in $\mathcal{V}$. See \cite[Theorem 7.2]{shulman2018} for details.

\item Applying the construction from part (i) to $(\mathsf{Cat},(-)^{\op})$ yields a $2$-category with contravariance whose objects are small categories and whose $1$-morphisms are covariant ($\epsilon=1$) and contravariant ($\epsilon=-1$) functors.
\end{enumerate}
\end{Ex}

\subsection{String diagrams}
\label{sec:striDiag}

For a detailed introduction to string diagrams the reader is referred to \cite[\S 4]{baez2004}.

String diagrams, which will be used to perform calculations in $2$-categories, are Poincar\'{e} dual to globular diagrams for $2$-categories. Two dimensional regions of a string diagram are therefore labelled by objects of the $2$-category while strings and nodes are labelled by $1$- and $2$-morphisms, respectively. Our conventions are such that string diagrams are read from right to left and bottom to top. Below is a globular diagram (left) together with its corresponding string diagram (right):
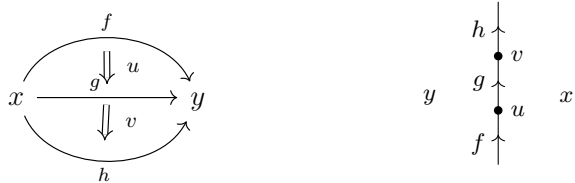
\[
\begin{tikzpicture}[baseline= (a).base]
\node[scale=0.9] (a) at (0,0){
\begin{tikzcd}[column sep=5em]
 x
  \arrow[bend left=65]{r}[name=U,below]{}{f} 
  \arrow[]{r}[name=M]{}[above left]{g}
  \arrow[bend right=65]{r}[name=D,]{}[swap]{h}
& 
y
     \arrow[Rightarrow,to path={(U) -- node[label=right:\footnotesize \scriptsize $u$] {} (M)}]{}
     \arrow[shorten <=3pt,shorten >=3pt,Rightarrow,to path={(M) -- node[label=right:\footnotesize \scriptsize $v$] {} (D)}]{}
\end{tikzcd}
};
\end{tikzpicture}
\qquad \qquad \qquad
\begin{gathered}
\begin{tikzpicture}[scale=0.18,inner sep=0.35mm, place/.style={circle,draw=black,fill=black,thick}]
\draw [decoration={markings, mark=at position 0.55 with {\arrow{>}}},        postaction={decorate}] (0,0) -- (0,4);
\draw [decoration={markings, mark=at position 0.55 with {\arrow{>}}},        postaction={decorate}] (0,4) -- (0,8);
\draw [decoration={markings, mark=at position 0.55 with {\arrow{>}}},        postaction={decorate}] (0,8) -- (0,12);
\draw (0,4) node [shape=circle,draw,fill] {};
\draw (0,8) node [shape=circle,draw,fill] {};
\node [right,label={[label distance=0.4mm]0: \scriptsize$u$}] at (0,4) {};
\node [right,label={[label distance=0.4mm]0:\scriptsize$v$}] at (0,8) {};
\node [left,label={[label distance=0.4mm]180:\scriptsize$h$}] at (0,10) {};
\node [left,label={[label distance=0.4mm]180:\scriptsize$g$}] at (0,6) {};
\node [left,label={[label distance=0.4mm]180:\scriptsize$f$}] at (0,1.5) {};
\node at (-5,5) {\scriptsize$y$};
\node at (5,5) {\scriptsize$x$};
\end{tikzpicture}
\end{gathered}
\]
Compositions of $1$- and $2$-morphisms are represented by the appropriate concatenations of string diagrams. Although arrows drawn on strings are redundant, we will often include them if they clarify diagrams. We sometimes omit labels of two dimensional regions and we do not draw identity $1$-morphisms. For example, a $2$-morphism $u: 1_x \Rightarrow f$ may be depicted as
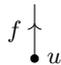
\[
\begin{tikzpicture}[scale=0.18,inner sep=0.35mm, place/.style={circle,draw=black,fill=black,thick}]
\draw [decoration={markings, mark=at position 0.65 with {\arrow{>}}},        postaction={decorate}] (0,0) -- (0,4);
\draw (0,0) node [shape=circle,draw,fill] {};
\node [right,label={[label distance=0.4mm]0: \scriptsize$u$}] at (0,0) {};
\node [left,label={[label distance=0.4mm]180:\scriptsize$f$}] at (0,2.0) {};
\end{tikzpicture}
\]

String diagrams can also be used for calculations in bicategories. If the bicategory is skeletal, as will be the case in all relevant examples below, then the only additional complication is that we must keep track of associators.

\subsection{\texorpdfstring{$\mathbb{Z}_2$}{}-graded groups}
\label{sec:Z2GrGrp}

Denote by $\mathbb{Z}_2$ the multiplicative group $\{\pm 1\}$. A group homomorphism $\pi: \hat{\mathsf{G}} \rightarrow \mathbb{Z}_2$ is called a $\mathbb{Z}_2$-graded group. Morphisms of $\mathbb{Z}_2$-graded groups are group homomorphisms which respect the structure maps to $\mathbb{Z}_2$. We will always assume that a given $\mathbb{Z}_2$-graded group is non-trivially graded in the sense that the structure map is surjective. A non-trivially graded $\mathbb{Z}_2$-graded group $\hat{\mathsf{G}}$ is necessarily an extension
\begin{equation}
\label{eq:sesGroup}
1 \rightarrow \mathsf{G} \rightarrow \hat{\mathsf{G}} \xrightarrow[]{\pi} \mathbb{Z}_2 \rightarrow 1.
\end{equation}
The subgroup $\mathsf{G} = \ker(\pi)$ is called the ungraded subgroup of $\hat{\mathsf{G}}$.

Alternatively, if we are given a group $\mathsf{G}$, then an extension of the form \eqref{eq:sesGroup} is called a Real structure on $\mathsf{G}$.

\begin{Ex}
An involutive group homomorphism $\varsigma: \mathsf{G} \rightarrow \mathsf{G}$ defines a split Real structure $\hat{\mathsf{G}} = \mathsf{G} \rtimes_{\varsigma} \mathbb{Z}_2$ on $\mathsf{G}$. Atiyah and Segal restrict attention to such Real structures in their study of equivariant $KR$-theory \cite{atiyah1969}.
\end{Ex}

\begin{Rem}
All constructions of this paper apply to general, as opposed to just split, Real structures. We give here two motivations for this level of generality.
\begin{enumerate}[label=(\roman*)]
\item Even if one is ultimately concerned only with split Real structures, inductive arguments in $KR$-theory often involve general Real structures.

\item The assumption that the Real structure is split is often too restrictive for applications to mathematical physics. For example, non-split Real structure are central to the orientifolding of string theory on orbifold backgrounds. See, for example, \cite{distler2011b}.
\end{enumerate}
\end{Rem}

Let $\hat{\mathsf{G}}$ be a $\mathbb{Z}_2$-graded group. Denote by $\Aut^{\textnormal{gen}}_{\textsf{Grp}}(\mathsf{G})$ the $\mathbb{Z}_2$-graded group of automorphisms and antiautomorphisms of $\mathsf{G}$. Define a map of $\mathbb{Z}_2$-graded groups by
\[
\varphi: \hat{\mathsf{G}} \rightarrow \Aut^{\textnormal{gen}}_{\textsf{Grp}}(\mathsf{G}), \qquad  \omega \mapsto (g \mapsto \omega g^{\pi(\omega)} \omega^{-1}).
\]
The induced $\hat{\mathsf{G}}$-action on $\mathsf{G}$ is called Real conjugation.

\begin{Ex}
Let $(\mathcal{C}, (-)^*, \Theta)$ be a category with duality. Given $x \in \Obj(\mathcal{C})$, let $\Aut_{\mathcal{C}}^{\textnormal{gen}}(x)$ be the set of all automorphisms and antiautomorphisms of $x$, the latter being by definition an isomorphism $x^* \rightarrow x$. For $f \in \Aut_{\mathcal{C}}^{\textnormal{gen}}(x)$, define $\pi(f) \in \mathbb{Z}_2$ so that $f: {^{\pi(f)}}x \rightarrow x$, the notation as in equation \eqref{eq:leftSupNot}. Then $\Aut_{\mathcal{C}}^{\textnormal{gen}}(x)$ becomes a $\mathbb{Z}_2$-graded group with ungraded subgroup $\Aut_{\mathcal{C}}(x)$ when given the product and inverse
\[
f_2 f_1 = f_2 \circ {^{\pi(f_2)}}f^{\pi(f_2)}_1 \circ \Theta_x^{\delta_{\pi(f_2),\pi(f_1),-1}}
\]
and
\[
I(f)= \begin{cases}
f^{-1} & \mbox{if } \pi(f)=1, \\
\Theta_x^{-1} \circ f^* & \mbox{if } \pi(f)=-1,
\end{cases}
\]
respectively, where we have introduced the notation
\[
\delta_{\epsilon_2, \epsilon_1 ,-1} =
\begin{cases}
1& \mbox{if } \epsilon_1 = \epsilon_2 =-1, \\
0 & \mbox{otherwise.}
\end{cases}
\]
For example, the associativity of the composition of three antiautomorphisms follows from equation \eqref{eq:catDualCompat}.
\end{Ex}

\subsection{Loop groupoids}
\label{sec:loopGrp}

Recall that a groupoid is a category in which all morphisms are isomorphisms. A groupoid is called finite if it has only finitely many objects and morphisms.

If a group $\mathsf{G}$ acts on a set $X$, then we denote by $X \git \mathsf{G}$ the groupoid with
\[
\Obj(X \git \mathsf{G}) = X, \qquad \Hom_{X \git \mathsf{G}}(x,y) = \{g \in \mathsf{G} \mid g x = y\}.
\]
We will write $B \mathsf{G}$ in place of $\pt \git \mathsf{G}$.

\begin{Def}[{\cite[\S 1.3]{willerton2008}}]
The loop groupoid of a finite groupoid $\mathfrak{G}$ is the functor category
\[
\Lambda \mathfrak{G} = 1\Hom_{\mathsf{Cat}}(B \mathbb{Z}, \mathfrak{G}).
\]
\end{Def}

Concretely, an object $(x,\gamma)$ of $\Lambda \mathfrak{G}$ is a loop $\gamma: x \rightarrow x$ in $\mathfrak{G}$ while a morphism $(x_1, \gamma_1) \rightarrow (x_2,\gamma_2)$ is a morphism $g: x_1 \rightarrow x_2$ which satisfies $\gamma_2 = g \gamma_1 g^{-1}$.

A finite groupoid over $B\mathbb{Z}_2$ is a morphism $\pi: \hat{\mathfrak{G}} \rightarrow B \mathbb{Z}_2$ of finite groupoids. The functor $\pi$ classifies an equivalence class of double covers $\pi:\mathfrak{G} \rightarrow \hat{\mathfrak{G}}$; we fix a choice of such a double cover in what follows. The relevant analogue of the loop groupoid in the setting of finite groupoids over $B \mathbb{Z}_2$ is the following.

\begin{Def}[\cite{mbyoung2018a}]
The unoriented loop groupoid $\Lambda^{\refl}_{\pi} \hat{\mathfrak{G}}$ of a finite groupoid $\hat{\mathfrak{G}}$ over $B \mathbb{Z}_2$ has objects the degree one loops in $\hat{\mathfrak{G}}$ and morphisms $(x_1, \gamma_1) \rightarrow (x_2,\gamma_2)$ the morphisms $\omega: x_1 \rightarrow x_2$ which satisfy $\gamma_2=\omega \gamma^{\pi(\omega)}_1 \omega^{-1}$.
\end{Def}

The superscript `$\refl$' stands for reflection, since $\Lambda_{\pi}^{\refl} \hat{\mathfrak{G}}$ is equivalent to the quotient of $\Lambda \mathfrak{G}$ by the diagonal $\mathbb{Z}_2$-action coming from deck transformations of $\mathfrak{G}$ and reflection of the circle $B \mathbb{Z}$.

\begin{Ex}
\begin{enumerate}[label=(\roman*)]
\item Let $\mathsf{G}$ be a finite group. The loop groupoid $\Lambda B \mathsf{G}$ is equivalent to the conjugation action groupoid $\mathsf{G} \git \mathsf{G}$.
\item Let $\hat{\mathsf{G}}$ be a finite $\mathbb{Z}_2$-graded group. The functor $B \pi: B \hat{\mathsf{G}} \rightarrow B\mathbb{Z}_2$ classifies the double cover $B \mathsf{G} \rightarrow B \hat{\mathsf{G}}$. The unoriented loop groupoid $\Lambda_{\pi}^{\refl} B \hat{\mathsf{G}}$ is equivalent to the Real conjugation action groupoid $\mathsf{G} \git_{\varphi} \hat{\mathsf{G}}$.
\end{enumerate}
\end{Ex}

\subsection{Twisted loop transgression}
\label{sec:twistLoopTran}

Loop transgression for finite groupoids was studied by Willerton \cite{willerton2008}. We recall a version of loop transgression for finite groupoids over $B \mathbb{Z}_2$  \cite{mbyoung2018a}.

Let $\hat{\mathfrak{G}}$ be a finite groupoid over $B \mathbb{Z}_2$. The double cover $\pi: \mathfrak{G} \rightarrow \hat{\mathfrak{G}}$ can be used to twist local systems on (the simplicial complex associated to) $\hat{\mathfrak{G}}$. Given an abelian group $\mathsf{A}$, viewed as a $\mathbb{Z}_2$-module via inversion, denote by $C^{\bullet}(\hat{\mathfrak{G}}, \mathsf{A}_{\pi})$ the complex of $\pi$-twisted $\mathsf{A}$-valued simplicial cochains on $\hat{\mathfrak{G}}$. Write $[\omega_n \vert \cdots \vert \omega_1]$ for the $n$-simplex of $\hat{\mathfrak{G}}$ determined by the diagram $x_1 \xrightarrow[]{\omega_1} \cdots \xrightarrow[]{\omega_n} x_{n+1}$ in $\hat{\mathfrak{G}}$.
In this notation, the differential of $\hat{\beta} \in C^n(\hat{\mathfrak{G}}, \mathsf{A}_{\pi})$ is defined by
\begin{multline*}
d \hat{\beta}([\omega_{n+1} \vert \cdots \vert \omega_1]) = \hat{\beta}([\omega_n \vert \cdots \vert \omega_1])^{\pi(\omega_{n+1})} \hat{\beta}([\omega_{n+1} \vert \cdots \vert \omega_2])^{(-1)^{n+1}} \times \\
\prod_{i=1}^n \hat{\beta}([\omega_{n+1} \vert \cdots \vert \omega_{i+2} \vert \omega_{i+1} \omega_i \vert \omega_{i-1} \vert \cdots \vert \omega_1])^{(-1)^i}.
\end{multline*}
The notation $Z^{\bullet} \subset C^{\bullet}$ indicates the subgroup of cocycles. Without loss of generality, we will assume that all cochains are normalized in the sense that they evaluate to the identity on chains in which one of the morphisms $\omega_i$ is an identity map.

We also write $[\omega_n \vert \cdots \vert \omega_1]\gamma$ for the $n$-simplex of $\Lambda_{\pi}^{\refl}  \hat{\mathfrak{G}}$ determined by the diagram
\[
\gamma \xrightarrow[]{\omega_1} \omega_1 \gamma^{\pi(\omega_1)} \omega_1^{-1} \xrightarrow[]{\omega_2} \cdots \xrightarrow[]{\omega_n} (\omega_n \cdots \omega_1) \gamma^{\pi(\omega_n \cdots \omega_1)} (\omega_n \cdots \omega_1)^{-1}.
\]

Let $k$ be a field. Reflection twisted loop transgression is a cochain map
\[
\uptau_{\pi}^{\refl} : C^{\bullet}(\hat{\mathfrak{G}} , k^{\times}_{\pi}) \rightarrow C^{\bullet-1}(\Lambda_{\pi}^{\refl} \hat{\mathfrak{G}}, k^{\times}).
\]
The map $\uptau_{\pi}^{\refl}$ is defined by a push-pull procedure, the main point being that the pushforward is along an unoriented map. This leads to the change in coefficient systems. We do not require a full description of $\uptau_{\pi}^{\refl}$. Instead, we record that for a $2$-cochain $\hat{\theta} \in C^2(\hat{\mathfrak{G}}, k^{\times}_{\pi})$ we have
\[
\uptau^{\refl}_{\pi}(\hat{\theta})([\omega]\gamma) = \hat{\theta}([\gamma^{-1} \vert \gamma])^{\frac{\pi(\omega)-1}{2}} \frac{\hat{\theta}([\omega \gamma^{\pi(\omega)} \omega^{-1} \vert \omega])}{\hat{\theta}([\omega \vert \gamma^{\pi(\omega)}])}
\]
and for a $3$-cochain $\hat{\alpha} \in C^3(\hat{\mathfrak{G}}, k^{\times}_{\pi})$ we have
\begin{multline*}
\uptau_{\pi}^{\refl}(\hat{\alpha})([\omega_2 \vert \omega_1 ]\gamma)=
\hat{\alpha}([\gamma \vert \gamma^{-1} \vert \gamma])^{\delta_{\pi(\omega_2), \pi(\omega_1),-1}} \times \\
\left(
\frac{\hat{\alpha}([\omega_1 \gamma^{-\pi(\omega_1)} \omega_1^{-1} \vert \omega_1 \gamma^{\pi(\omega_1)} \omega_1^{-1} \vert \omega_1]) \hat{\alpha}([\omega_1 \vert \gamma^{-\pi(\omega_1)} \vert \gamma^{\pi(\omega_1)}])}{\hat{\alpha}([\omega_1 \gamma^{-\pi(\omega_1)} \omega_1^{-1} \vert \omega_1 \vert \gamma^{\pi(\omega_1)}])}
\right)^{\frac{\pi(\omega_2) -1}{2}} \times \\
\frac{\hat{\alpha}([\omega_2 \vert \omega_1 \vert \gamma^{\pi(\omega_2 \omega_1)}]) \hat{\alpha}([\omega_2 \omega_1 \gamma^{\pi(\omega_2 \omega_1)} (\omega_2 \omega_1)^{-1} \vert \omega_2 \vert \omega_1])}{\hat{\alpha}([\omega_2 \vert \omega_1 \gamma^{\pi(\omega_2 \omega_1)} \omega_1^{-1} \vert \omega_1])}.
\end{multline*}
If $\hat{\theta}$ is in fact a $2$-cocycle, then $\uptau_{\pi}^{\refl}(\hat{\theta})$ is a $1$-cocycle, meaning the equality
\begin{equation}
\label{eq:twistTransPartialCocycle}
\uptau_{\pi}^{\refl}(\hat{\theta})([\omega_2] \omega_1 \gamma^{\pi(\omega_1)} \omega_1^{-1}) \uptau_{\pi}^{\refl}(\hat{\theta})([\omega_1] \gamma) = \uptau_{\pi}^{\refl}(\hat{\theta})([\omega_2 \omega_1] \gamma)
\end{equation}
holds for each $2$-chain $[\omega_2 \vert \omega_1]\gamma$. This follows from \cite{mbyoung2018a}, but can also be verified directly. This and the corresponding statement for $3$-cocycles are the only facts about $\uptau_{\pi}^{\refl}$ that we will assume. In particular, the expressions for $\uptau_{\pi}^{\refl}(\hat{\theta})$ and $\uptau_{\pi}^{\refl}(\hat{\alpha})$ will be derived from the point of view of Real ($2$-)representation theory.

Finally, we note that when $\hat{\mathfrak{G}} = B \hat{\mathsf{G}}$ and $\mathfrak{G} = B \mathsf{G}$ the twisted transgression map $\uptau_{\pi}^{\refl}$ restricts to Willerton's transgression map $\uptau : C^{\bullet}(B \mathsf{G} , k^{\times}) \rightarrow C^{\bullet-1}(\Lambda B \mathsf{G}, k^{\times})$.

\section{Twisted Real representation theory of finite groups}
\label{sec:RealRepFinGroup}

As motivation for the remainder of the paper, we recall the basics of the theory of twisted, or projective, Real representations of finite groups, in both its linear and antilinear formulation. For twisted representations of finite group(oid)s, see \cite{karpilovsky1985}, \cite{willerton2008}.

\subsection{The antilinear theory}
\label{sec:RealFGRep}

In the case of untwisted real representations, the material in this section is standard \cite{fulton1991}. Aspects of the untwisted Real case are treated in \cite{atiyah1969}, \cite{karoubi1970}. A general reference is \cite{mbyoung2018a}.

Let $k$ be a field which is a quadratic extension of a field $k_0$. We regard $k_0$ as the fixed subfield of a $k_0$-linear Galois involution $k \rightarrow k$. A standard example is $k_0 = \mathbb{R} \subset k=\mathbb{C}$. A map $V \rightarrow W$ of vector spaces over $k$ is called $+1$-linear (resp. $-1$-linear) if it is $k$-linear (resp. $k$-antilinear).

Let $\mathsf{G}$ be a finite group with Real structure $\hat{\mathsf{G}}$. Let $\hat{\theta} \in Z^2(B \hat{\mathsf{G}}, k^{\times}_{\pi})$, where $\mathsf{G}$ acts trivially on $k^{\times}$ and $\hat{\mathsf{G}} \backslash \mathsf{G}$ acts by the Galois involution. Write $\theta \in Z^2(B \mathsf{G}, k^{\times})$ for the restriction of $\hat{\theta}$ to $B\mathsf{G}$.

\begin{Def}
A $\hat{\theta}$-twisted Real representation of $\mathsf{G}$ is a finite dimensional vector space $V$ over $k$ together with $\pi(\omega)$-linear maps $\rho(\omega) : N \rightarrow N$, $\omega \in \hat{\mathsf{G}}$, which satisfy $\rho(e)=1_N$ and
\[
\rho(\omega_2) \circ \rho(\omega_1) = \hat{\theta}([\omega_2 \vert \omega_1]) \rho(\omega_2 \omega_1).
\]
\end{Def}

Twisted Real representations of $\mathsf{G}$ and their $\hat{\mathsf{G}}$-equivariant $k$-linear maps form a $k_0$-linear additive category $\mathsf{RRep}_k^{\hat{\theta}}(\mathsf{G})$. Despite the notation, $\mathsf{RRep}_k^{\hat{\theta}}(\mathsf{G})$ depends on the Real structure $\hat{\mathsf{G}}$. Let $KR^{0+\hat{\theta}}(B \mathsf{G})$ be the Grothendieck group of $\mathsf{RRep}^{\hat{\theta}}_k(\mathsf{G})$.

The Real character of a $\hat{\theta}$-twisted Real representation $\rho$ is the function
\[
\chi_{\rho}: \mathsf{G} \rightarrow  k, \qquad  g \mapsto \tr_N(\rho(g)).
\]
In other words, $\chi_{\rho}$ is the character of the underlying $\theta$-twisted representation of $\mathsf{G}$. The new feature of Real characters is their Real conjugation equivariance,
\begin{equation}
\label{eq:real1Char}
\chi_{\rho}(\omega g^{\pi(\omega)} \omega^{-1}) = \uptau_{\pi}^{\refl}(\hat{\theta})([\omega] g) \cdot \chi_{\rho}(g), \qquad \omega \in \hat{\mathsf{G}}
\end{equation}
which refines the conjugation equivariance of characters of $\theta$-twisted representations. The Real character map extends to a $k$-linear map
\begin{equation}
\label{eq:RealRepChar}
\chi: KR^{0+\hat{\theta}}(B \mathsf{G}) \otimes_{\mathbb{Z}} k \rightarrow \Gamma_{\Lambda_{\pi}^{\refl} B \hat{\mathsf{G}}}(\uptau_{\pi}^{\refl}(\hat{\theta})_k).
\end{equation}
Adopting the notation of \cite[\S 2.2]{willerton2008}, the right hand side denotes the space of flat sections of the transgressed line bundle $\uptau_{\pi}^{\refl}(\hat{\theta})_k \rightarrow \Lambda_{\pi}^{\refl} B \hat{\mathsf{G}}$. Explicitly, $\Gamma_{\Lambda_{\pi}^{\refl} B \hat{\mathsf{G}}}(\uptau_{\pi}^{\refl}(\hat{\theta})_k)$ is the space of all functions $\chi: \mathsf{G} \rightarrow k$ which satisfy equation \eqref{eq:real1Char}. When $k = \mathbb{C}$ with complex conjugation as the involution, the map \eqref{eq:RealRepChar} is an isomorphism \cite[Theorem 3.7]{mbyoung2018a}.

\begin{Ex}
The real setting is $k_0 = \mathbb{R} \subset k=\mathbb{C}$ with $\pi: \hat{\mathsf{G}} = \mathsf{G} \times \mathbb{Z}_2 \rightarrow \mathbb{Z}_2$ the projection and $\hat{\theta}=1$. Then $\mathsf{RRep}_{\mathbb{C}}(\mathsf{G}) \simeq \mathsf{Rep}_{\mathbb{R}}(\mathsf{G})$ and $KR^0(B \mathsf{G}) \simeq RO(\mathsf{G})$. Equation \eqref{eq:real1Char} becomes the statement that characters of real representations are real valued class functions and the isomorphism \eqref{eq:RealRepChar} identifies $RO(\mathsf{G}) \otimes_{\mathbb{Z}} \mathbb{C}$ with the space of functions on $\mathsf{G}$ which are constant on conjugacy classes and their inverses.
\end{Ex}

\subsection{The linear theory}
\label{sec:linearRealFGRep}

We describe a linear approach to the twisted Real representation theory of a finite group. Aspects of the untwisted real case are discussed in \cite{zibrowius2015}. This section will be the basis for our categorification in later sections.

We keep the notation from Section \ref{sec:RealFGRep}, although $k$ is now an arbitrary field and $\hat{\mathsf{G}} \backslash \mathsf{G}$ acts on $k^{\times}_{\pi}$ by inversion. We give two linear versions of the notion of a Real representation of $\mathsf{G}$. The first is less natural, requiring the choice of an element $\varsigma \in \hat{\mathsf{G}} \backslash \mathsf{G}$, but has the benefit that it fits into the framework of Grothendieck--Witt theory.

\begin{Lem}
\label{lem:dualProjRep}
Let $(V,\rho)$ be a $\theta$-twisted representation of $\mathsf{G}$. For each $\varsigma \in \hat{\mathsf{G}} \backslash \mathsf{G}$, the pair  $(V^{\vee},\rho^{\varsigma})$, where $V^{\vee}$ is the $k$-linear dual of $V$ and 
\[
\rho^{\varsigma}(g) = \uptau_{\pi}^{\refl}(\hat{\theta})([\varsigma^{-1}] g)^{-1} \rho(\varsigma^{-1} g^{-1} \varsigma)^{\vee}, \qquad g \in \mathsf{G}
\]
is a $\theta$-twisted representation of $\mathsf{G}$.
\end{Lem}

\begin{proof}
The key point is the following easily-verified identity, valid for $g_1, g_2 \in \mathsf{G}$ and $\omega \in \hat{\mathsf{G}}$:
\begin{equation}
\label{eqn:keyIdentity}
\frac{\hat{\theta}([\omega g_2 \omega^{-1} \vert \omega g_1 \omega^{-1}])}{\hat{\theta}([g_2 \vert g_1])^{\pi(\omega)}}
=
\frac{\hat{\theta}([\omega \vert \omega^{-1}]) \hat{\theta}([\omega g_2 g_1 \vert \omega^{-1}]) \hat{\theta}([\omega \vert g_2 g_1])}{\hat{\theta}([\omega \vert g_2]) \hat{\theta}([\omega g_2 \vert \omega^{-1}]) \hat{\theta}([\omega \vert g_1]) \hat{\theta}([\omega g_1 \vert \omega^{-1}])}.
\end{equation}
\end{proof}

Each element $\varsigma \in \hat{\mathsf{G}} \backslash \mathsf{G}$ determines a $k$-linear exact duality structure $(P^{\varsigma}, \Theta^{\varsigma})$ on $\mathsf{Rep}_k^{\theta}(\mathsf{G})$. The functor $P^{\varsigma} : \mathsf{Rep}_k^{\theta}(\mathsf{G})^{\op} \rightarrow  \mathsf{Rep}_k^{\theta}(\mathsf{G})$ is given on objects by $P^{\varsigma}(\rho) = \rho^{\varsigma}$. The natural isomorphism $\Theta^{\varsigma}: 1_{\mathsf{Rep}_k^{\theta}(\mathsf{G})} \Rightarrow P^{\varsigma} \circ (P^{\varsigma}) ^{\op}$ has components
\[
\Theta^{\varsigma}_{\rho} = \hat{\theta}([\varsigma^{-1} \vert \varsigma^{-1}]) \ev_{\rho} \circ \rho(\varsigma^{-2}).
\]
Given two elements $\varsigma_1, \varsigma_2 \in \hat{\mathsf{G}} \backslash \mathsf{G}$, the natural transformation $\nu^{\varsigma_1, \varsigma_2}: P^{\varsigma_1} \Rightarrow P^{\varsigma_2}$ with components $\nu^{\varsigma_1, \varsigma_2}_{\rho} = \rho(\varsigma_1^{-1} \varsigma_2)^{\vee}$ lifts to a non-singular form functor
\[
(\mathsf{Rep}_k^{\theta}(\mathsf{G}), P^{\varsigma_1}, \Theta^{\varsigma_1}) \xrightarrow[]{\sim} (\mathsf{Rep}_k^{\theta}(\mathsf{G}), P^{\varsigma_2}, \Theta^{\varsigma_2}).
\]
In this way, the pair $(\hat{\mathsf{G}}, \hat{\theta})$ determines a $\mathsf{G}$-torsor of duality structures on $\mathsf{Rep}_k^{\theta}(\mathsf{G})$.

\begin{Def}
A $\hat{\theta}$-twisted symmetric representation of $\mathsf{G}$ is a symmetric form in $(\mathsf{Rep}_k^{\theta}(\mathsf{G}), P^{\varsigma}, \Theta^{\varsigma})$.
\end{Def}

Twisted symmetric representations are objects of a homotopy fixed point category, which we denote by $\mathsf{Rep}_k^{\hat{\theta},\varsigma}(\mathsf{G})$.

\begin{Ex}
Let $\hat{\mathsf{G}} =\mathsf{G} \times \mathbb{Z}_2$ with $\varsigma$ the generator of $\mathbb{Z}_2$. An untwisted symmetric representation is a representation together with a $\mathsf{G}$-invariant nondegenerate symmetric bilinear form. If instead the twisting is $\hat{\theta}([\omega_2 \vert \omega_1])= (-1)^{\delta_{\pi(\omega_2), \pi(\omega_1),-1}}$, then the bilinear form is skew-symmetric
\end{Ex}

We now give a more invariant definition.

\begin{Def}
A $\hat{\theta}$-twisted generalized symmetric representation of $\mathsf{G}$ is a vector space $N$ together with linear maps $\rho(\omega): \prescript{\pi(\omega)}{}{N} \rightarrow N$, $\omega \in \hat{\mathsf{G}}$, which satisfy $\rho(e) = 1_N$ and
\[
\rho(\omega_2) \circ \prescript{\pi(\omega_2)}{}{\rho(\omega_1)}^{\pi(\omega_2)} \circ \ev_N^{\delta_{\pi(\omega_1), \pi(\omega_2), -1}} = \hat{\theta}([\omega_2 \vert \omega_1]) \rho(\omega_2 \omega_1).
\]
\end{Def}

Finite dimensional twisted generalized symmetric representations form a category $\mathsf{SRep}_k^{\hat{\theta}}(\mathsf{G})$, morphisms $\phi: N \rightarrow M$ being morphisms of twisted representations which satisfy $\phi \circ \rho_N(\omega) \circ \phi^{\vee} = \rho_M(\omega)$ for each $\omega \in \hat{\mathsf{G}} \backslash \mathsf{G}$.

\begin{Rem}
More generally, a Real representation of $\mathsf{G}$ on an object $x$ of a category with duality $(\mathcal{C}, (-)^*, \Theta)$ is defined to be a $\mathbb{Z}_2$-graded group homomorphism $\rho: \hat{\mathsf{G}} \rightarrow \Aut_{\mathcal{C}}^{\textnormal{gen}}(x)$. To relate this to the definition above, let ${^{\hat{\theta}}}\hat{\mathsf{G}}$ be the  $\mathbb{Z}_2$-graded group which is equal to $k^{\times} \times \hat{\mathsf{G}}$ as a set and has product
\[
(z_2,\omega_2) \cdot (z_1, \omega_1) = (\hat{\theta}([\omega_2 \vert \omega_1]) z_2 z_1^{\pi(\omega_2)}, \omega_2 \omega_1).
\]
A $\hat{\theta}$-twisted generalized symmetric representation of $\mathsf{G}$ is then a Real representation of the ungraded subgroup of ${^{\hat{\theta}}}\hat{\mathsf{G}}$ on an object of $(\mathsf{Vect}_k, (-)^{\vee}, \ev)$ which has the additional property that $k^{\times} \leq {^{\hat{\theta}}}\hat{\mathsf{G}}$ acts by scalar multiplication.
\end{Rem}

\begin{Prop}
\label{prop:symmGenSymmEquiv}
The categories $\mathsf{Rep}_k^{\hat{\theta},\varsigma}(\mathsf{G})$ and $\mathsf{SRep}_k^{\hat{\theta}}(\mathsf{G})$ are equivalent.
\end{Prop}

\begin{proof}
An equivalence $F^{\varsigma}:\mathsf{Rep}_k^{\hat{\theta},\varsigma}(\mathsf{G}) \rightarrow \mathsf{SRep}_k^{\hat{\theta}}(\mathsf{G})$ is defined on objects by assigning to a twisted symmetric representation $(N, \psi_N)$ the twisted generalized symmetric representation which is equal to $N$ as a twisted representation and has
\[
\rho(\omega) = \hat{\theta}([\omega \vert \varsigma^{-1}]) \rho(\omega \varsigma^{-1}) \circ \psi^{-1}_N, \qquad \omega \in \hat{\mathsf{G}} \backslash \mathsf{G}.
\]
On morphisms $F^{\varsigma}$ acts as the identity.
\end{proof}

Let $GW^{\hat{\theta}}_0(\mathsf{G})$ be the Grothendieck--Witt group of $(\mathsf{Rep}_k^{\theta}(\mathsf{G}), P^{\varsigma}, \Theta^{\varsigma})$. Since non-singular form functors induce isomorphisms of Grothendieck--Witt groups, up to isomorphism, $GW^{\hat{\theta}}_0(\mathsf{G})$ is independent of the choice of $\varsigma \in \hat{\mathsf{G}} \backslash \mathsf{G}$.

Characters of twisted (generalized) symmetric representations of $\mathsf{G}$ are defined in the same way as Section \ref{sec:RealFGRep}. Real conjugation equivariance \eqref{eq:real1Char} continues to hold. When $k=\mathbb{C}$, the isomorphism \eqref{eq:RealRepChar} is replaced by the isomorphism
\[
\chi: GW^{\hat{\theta}}_0(\mathsf{G}) \otimes_{\mathbb{Z}} \mathbb{C} \rightarrow \Gamma_{\Lambda_{\pi}^{\refl} B \hat{\mathsf{G}}}(\uptau_{\pi}^{\refl}(\hat{\theta})_{\mathbb{C}}).
\]
In fact, by picking a $\hat{\mathsf{G}}$-invariant Hermitian metric on each twisted symmetric representation, we obtain an isomorphism of abelian groups
\[
GW^{\hat{\theta}}_0(\mathsf{G}) \rightarrow KR^{0+\hat{\theta}}(B \mathsf{G}).
\]
So while the linear and antilinear Real representation categories are not equivalent, the relevant Grothendieck(--Witt) groups are isomorphic.

\section{Real representations of finite categorical groups}
\label{sec:Real2Reps}

\subsection{Categorical groups}
\label{sec:catGroupBasics}

The concept of a group can be categorified in a number of ways. A detailed discussion of these categorifications, and the relations between them, can be found in \cite{baez2004}.

A categorical group, called a weak $2$-group in \cite{baez2004}, is a monoidal groupoid $(\mathcal{G}, \otimes, \mathbf{1})$ in which every object admits a weak inverse. Explicitly, this means that for each object $x$ of $\mathcal{G}$ there exists a second object $y$ such that both $x \otimes y$ and $y \otimes x$ are equivalent to the monoidal unit $\mathbf{1}$. A morphism of categorical groups is a monoidal functor. By considering also monoidal natural transformations between monoidal functors, categorical groups assemble to a $2$-category.

The monoidal structure $\otimes$ gives the set of connected components $\pi_0(\mathcal{G})$ the structure of a group. The group $\pi_1(\mathcal{G})$ of autoequivalences of the monoidal unit $\mathbf{1}$ is, by an Eckmann--Hilton argument, abelian. As described in Section \ref{sec:sinh} below, the groups $\pi_0(\mathcal{G})$ and $\pi_1(\mathcal{G})$, together with some additional data, determine $\mathcal{G}$ up to equivalence.

\begin{Ex}
Any group $\mathsf{G}$, considered as a discrete category with object set $\mathsf{G}$ and monoidal structure determined (on objects) by its group law, defines a categorical group. By a slight abuse of notation, we will denote this categorical group by $\mathsf{G}$.
\end{Ex}

\begin{Ex}
Let $\mathsf{A}$ be an abelian group. The action groupoid $B \mathsf{A}$ is a categorical group, the monoidal structure determined (on morphisms) by the group law of $\mathsf{A}$.
\end{Ex}

\begin{Ex}
Let $x$ be an object of a bicategory $\mathcal{V}$. Then $1\Aut_{\mathcal{V}}(x)$, the groupoid of autoequivalences of $x$ and the $2$-isomorphisms between them, is a categorical group, called the weak automorphism $2$-group of $x$ \cite[\S 8.1]{baez2004}. If $\mathcal{V}$ is $k$-linear and we restrict attention to $k$-linear autoequivalences and their $k$-linear $2$-isomorphisms, then we obtain a categorical group $\mathsf{GL}_k(x)$ \cite[\S 3.3.2]{frenkel2012}.
\end{Ex}

\begin{Def}
A categorical group $\mathcal{G}$ is called finite if $\pi_0(\mathcal{G})$ is finite. 
\end{Def}

\subsection{Sinh's theorem}
\label{sec:sinh}

The following classification indicates that categorical groups can be viewed as twisted extended versions of groups.

\begin{Thm}[{\cite{sinh1975}; see also \cite[\S 8.3]{baez2004}}]
\label{thm:catGroupClass}
Categorical groups are classified up to equivalence by the following data:
\begin{enumerate}[label=(\roman*)]
\item A group $\mathsf{G}$.
\item An abelian group $\mathsf{A}$.
\item A group homomorphism $\Pi: \mathsf{G} \rightarrow \Aut_{\textnormal{\textsf{Grp}}}(\mathsf{A})$.
\item A cohomology class $[\alpha] \in H^3(B\mathsf{G}, \mathsf{A}_{\Pi})$.
\end{enumerate}
\end{Thm}

In a similar way, equivalence classes of finite categorical groups are classified by the data (i)-(iv), with the additional condition that $\mathsf{G}$ be finite.

An explicit model for the categorical group  determined by Theorem \ref{thm:catGroupClass} is as follows. Fix a normalized representative $\alpha \in Z^3(B \mathsf{G}, \mathsf{A}_{\Pi})$ of $[\alpha]$. Let $\mathcal{G}(\mathsf{G}, \mathsf{A}, \Pi, \alpha)$ be the skeletal groupoid with set of objects $\mathsf{G}$, a morphism $g \xrightarrow[]{a} g$ for each pair $(g, a) \in \mathsf{G} \times \mathsf{A}$ and composition law
\[
(g \xrightarrow[]{a_2} g) \circ (g \xrightarrow[]{a_1} g) = (g \xrightarrow[]{a_1 \cdot a_2} g).
\]
The monoidal structure $\otimes$ is determined on objects by the group law of $\mathsf{G}$ and on morphisms by
\[
(g \xrightarrow[]{a} g) \otimes (g^{\prime} \xrightarrow[]{a^{\prime}} g^{\prime}) = (g g^{\prime} \xrightarrow[]{a \cdot \Pi(g)(a^{\prime})} g g^{\prime}).
\]
The associator is given by the maps $g_3 g_2 g_1 \xrightarrow[]{\alpha([g_3 \vert g_2 \vert g_1])} g_3 g_2 g_1$. Since $\alpha$ is normalized, the unitors can be taken to be identity maps.

\begin{Ex}
If $\mathsf{A}$ is trivial, then $\mathcal{G}(\mathsf{G}, \mathsf{A}, \Pi, \alpha)$ is simply the group $\mathsf{G}$, viewed as a categorical group. If $\mathsf{A}$ is non-trivial but $\alpha$ is trivial, then $\mathcal{G}(\mathsf{G}, \mathsf{A}, \Pi, \alpha)$ is the categorical group extension of $\mathsf{G}$ by $B \mathsf{A}$ determined by $\Pi$.
\end{Ex}

The next example describes a particularly important and well-studied class of finite categorical groups.

\begin{Ex}
Let $k$ be a field. Let $\mathsf{G}$ be a group and let $\Pi : \mathsf{G} \rightarrow \Aut_{\mathsf{Grp}}(k^{\times})$ be the trivial map. The associated categorical group, denoted simply by $\mathcal{G}(\mathsf{G}, \alpha)$, is a twisted categorical group extension of $\mathsf{G}$ by $B k^{\times}$.
\end{Ex}

\subsection{\texorpdfstring{$\mathbb{Z}_2$}{}-graded categorical groups}
\label{sec:z2CatGrp}

Before introducing Real representations of categorical groups, we categorify the notion of a Real structure on a group.

\begin{Def}
A morphism of categorical groups $\pi: \hat{\mathcal{G}} \rightarrow \mathbb{Z}_2$ is called a $\mathbb{Z}_2$-graded categorical group.
\end{Def}

A morphism of $\mathbb{Z}_2$-graded categorical groups is a morphism of categorical groups which commutes with the structure maps to $\mathbb{Z}_2$.

The ungraded categorical group of (a non-trivially graded) $\hat{\mathcal{G}}$ is the full subcategory $\mathcal{G} \subset \hat{\mathcal{G}}$ on objects which map via $\pi$ to $1 \in \mathbb{Z}_2$. There are morphisms of categorical groups
\begin{equation}
\label{eq:ses2Group}
1 \rightarrow \mathcal{G} \xrightarrow[]{i} \hat{\mathcal{G}} \xrightarrow[]{\pi} \mathbb{Z}_2
\end{equation}
with $i$ an isomorphism onto its image and $\pi$ surjective on objects and full. Alternatively, given a categorical group $\mathcal{G}$, a diagram of the form \eqref{eq:ses2Group} having the above properties is called a Real structure on $\mathcal{G}$.

Since the category $\mathbb{Z}_2$ is skeletal, a $\mathbb{Z}_2$-grading of $\hat{\mathcal{G}}$ is simply a $\mathbb{Z}_2$-grading of $\Obj(\hat{\mathcal{G}})$ which is compatible with the monoidal structure. In particular, a $\mathbb{Z}_2$-grading of $\mathcal{G}(\mathsf{G}, \mathsf{A}, \Pi, \alpha)$ determines, and is determined by, a $\mathbb{Z}_2$-grading of $\mathsf{G}$. The ungraded categorical group is obtained by restriction to the ungraded subgroup of $\mathsf{G}$.

The following example plays an important role in the remainder of the paper.

\begin{Ex}
Let $k$ be a field and let $\pi: \hat{\mathsf{G}} \rightarrow \mathbb{Z}_2$ be a $\mathbb{Z}_2$-graded group. Let $\Pi: \hat{\mathsf{G}} \rightarrow \Aut_{\mathsf{Grp}}(k^{\times})$ be the map $\Pi(\omega)(a) = a^{\pi(\omega)}$, so that $k^{\times}_{\Pi}$ is $k^{\times}_{\pi}$ in the notation of Section \ref{sec:linearRealFGRep}. Let $\hat{\alpha} \in Z^3(B \hat{\mathsf{G}}, k^{\times}_{\pi})$. The categorical group $\mathcal{G}(\hat{\mathsf{G}}, k^{\times}, \Pi, \hat{\alpha})$ defined by Theorem \ref{thm:catGroupClass}, henceforth denoted by $\mathcal{G}(\hat{\mathsf{G}}, \hat{\alpha})$, is $\mathbb{Z}_2$-graded with ungraded categorical group $\mathcal{G}(\mathsf{G}, \alpha)$, where $\alpha \in Z^3(B \mathsf{G}, k^{\times})$ is the restriction of $\hat{\alpha}$ to $B\mathsf{G}$.
\end{Ex}

The following example categorifies the $\mathbb{Z}_2$-graded group $\Aut_{\mathcal{C}}^{\textnormal{gen}}(x)$ of Section \ref{sec:Z2GrGrp}.

\begin{Ex}
Let $x$ be an object of a bicategory $\mathcal{V}$ with weak duality involution. Then $1\Aut^{\textnormal{gen}}_{\mathcal{V}}(x)$, the collection of all equivalences $x \rightarrow x$ and $x^{\circ} \rightarrow x$, together with the $2$-isomorphisms between them, is a $\mathbb{Z}_2$-graded categorical group. The monoidal structure $\otimes$ is defined on objects by
\[
f_2 \otimes f_1 = f_2 \circ_0 ({^{\pi(f_2)}}f_1 \circ_0 \eta_x^{\delta_{\pi(f_2), \pi(f_1), -1}}),
\]
where $\pi(f) \in \mathbb{Z}_2$ is such that $f: {^{\pi(f)}} x \rightarrow x$. The definition of $\otimes$ on morphisms is similar. The associator for three antiautoequivalences is
\begin{multline*}
(f_3 \otimes f_2) \otimes f_1 = (f_3 \circ (f_2^{\circ} \circ \eta_x) ) \circ f_1 \xrightarrow[]{\alpha} f_3 \circ (f_2^{\circ} \circ ( \eta_x \circ f_1)) \xrightarrow[]{\eta}
\\
f_3 \circ (f_2^{\circ} \circ (f_1^{\circ \circ} \circ \eta_{x^{\circ}})) \xrightarrow[]{\zeta_x} f_3 \circ (f_2^{\circ} \circ (f_1^{\circ \circ} \circ \eta_x^{\circ})) = f_3 \otimes (f_2 \otimes f_1),
\end{multline*}
where $\alpha$ is a composition of associators for $\mathcal{V}$ and the arrow labelled by $\eta$ is a pseudo-naturality constraint for $\eta$. The remaining associators are constructed in a similar way, but do not use the modification $\zeta$. The pentagon identity is verified using the constraint \eqref{eq:bicatDualCompat}. If $x$ has at least one antiautoequivalence, then the morphism $\pi: 1\Aut^{\textnormal{gen}}_{\mathcal{V}}(x) \rightarrow \mathbb{Z}_2$ fits into an exact sequence of categorical groups:
\[
1 \rightarrow 1 \Aut_{\mathcal{V}}(x) \rightarrow 1\Aut^{\textnormal{gen}}_{\mathcal{V}}(x) \xrightarrow[]{\pi} \mathbb{Z}_2  \rightarrow 1.
\]
If $\mathcal{V}$ is $k$-linear and we restrict attention to $k$-linear (anti)autoequivalences and $2$-isomorphisms, then we obtain a $\mathbb{Z}_2$-graded categorical group $\mathsf{GL}^{\textnormal{gen}}_k(x)$ whose ungraded categorical group is $\mathsf{GL}_k(x)$.
\end{Ex}

\begin{Ex}
The previous example has a variation in which the bicategory with duality involution is replaced by a bicategory $\mathcal{V}$ with contravariance. In this way, for each $x \in \Obj(\mathcal{V})$ we obtain a $\mathbb{Z}_2$-graded categorical group $1\Aut^{\textnormal{gen}}_{\mathcal{V}}(x)$ whose ungraded categorical group is $1\Aut_{\mathcal{V}_1}(x)$.
\end{Ex}

A $\mathbb{Z}_2$-graded categorical group $\pi: \hat{\mathcal{G}} \rightarrow \mathbb{Z}_2$ defines a bicategory $\underline{\hat{\mathcal{G}}}$ with contravariance as follows. Let $\Obj(\underline{\hat{\mathcal{G}}}) = \{ \pt\}$. For each $\epsilon \in \mathbb{Z}_2$, let $1\Hom_{\underline{\hat{\mathcal{G}}}}^{\epsilon}(\pt,\pt)$ be the full subcategory of $\hat{\mathcal{G}}$ on objects which map via $\pi$ to $\epsilon$. The horizontal compositions and associators in $\underline{\hat{\mathcal{G}}}$ are induced by the monoidal structure of $\hat{\mathcal{G}}$.

\begin{Rem}
It may be interesting to consider gradings of categorical groups by non-trivial categorifications of $\mathbb{Z}_2$. In the $k$-linear setting, one possibility is to use the symmetric monoidal category $\mathsf{Pic}^{\mathbb{Z}_2}(k^{\times})$ of Ganter--Kapranov \cite[Example 3.1.2(d)]{ganter2014}.
\end{Rem}

\subsection{Real representations of finite categorical groups}
\label{sec:Real2RepBasic}

We introduce Real representations of finite categorical groups, categorifying the linear approach of Section \ref{sec:linearRealFGRep}. An antilinear approach can be found in Section \ref{sec:antiLinearRealFGRep}.

Let $\mathcal{G}$ be a finite categorical group with Real structure $\hat{\mathcal{G}}$.

\begin{Def}
A Real representation of $\mathcal{G}$ on a bicategory $\mathcal{V}$ with contravariance is a contravariance preserving pseudofunctor $\rho: \underline{\hat{\mathcal{G}}} \rightarrow \mathcal{V}$, where $\underline{\hat{\mathcal{G}}}$ is the bicategory with contravariance defined is as in Section \ref{sec:z2CatGrp}.
\end{Def}

Real representations of $\mathcal{G}$ on $\mathcal{V}$ assemble to a bicategory $\mathsf{RRep}_{\mathcal{V}}(\mathcal{G})$ whose $1$- and $2$-morphisms are pseudonatural transformations and modifications which respect contravariance, respectively. More compactly, we can define
\[
\mathsf{RRep}_{\mathcal{V}}(\mathcal{G}) = 1\Hom_{\mathsf{Bicat}_{\mathsf{con}}}(\underline{\hat{\mathcal{G}}}, \mathcal{V}),
\]
where $\mathsf{Bicat}_{\mathsf{con}}$ is the tricategory of bicategories with contravariance described in \cite{shulman2018}. If $\mathcal{V}$ is in fact a $2$-category, then so too is $\mathsf{RRep}_{\mathcal{V}}(\mathcal{G})$.

Let $\mathsf{Bicat}^{\leq 1}_{\mathsf{con}}$ be the category of bicategories with contravariance and their pseudofunctors preserving contravariance. Taking $1$-morphism bicategories defines a functor
\[
1\Hom_{\mathsf{Bicat}_{\mathsf{con}}}(-,-) : (\mathsf{Bicat}^{\leq 1}_{\mathsf{con}})^{\op} \times \mathsf{Bicat}^{\leq 1}_{\mathsf{con}} \rightarrow \mathsf{Bicat}^{\leq 1}.
\]
Using this functor, it can be verified that if $\mathcal{V}$ and $\mathcal{V}^{\prime}$ are biequivalent bicategories with contravariance and $\hat{\mathcal{G}}$ and $\hat{\mathcal{G}}^{\prime}$ are equivalent $\mathbb{Z}_2$-graded categorical groups, then $\mathsf{RRep}_{\mathcal{V}}(\mathcal{G})$ and $\mathsf{RRep}_{\mathcal{V}^{\prime}}(\mathcal{G}^{\prime})$ are biequivalent. Compare with \cite[\S 3.5]{elgueta2007}. In view of coherence theorem for bicategories with contravariance, it follows that there is no loss of generality in restricting attention to Real representations on $2$-categories with contravariance.

We will use the following interpretation of Real representations.

\begin{Lem}
\label{lem:RealRepHomCatGrp}
A Real representation of $\mathcal{G}$ on a bicategory $\mathcal{V}$ with contravariance is the data of an object $V \in \Obj(\mathcal{V})$ together with a morphism of $\mathbb{Z}_2$-graded categorical groups $\rho: \hat{\mathcal{G}} \rightarrow 1\Aut^{\textnormal{gen}}_{\mathcal{V}}(V)$.
\end{Lem}

\begin{proof}
This is straightforward.
\end{proof}

Motivated by Lemma \ref{lem:RealRepHomCatGrp}, define a Real representation of $\mathcal{G}$ on a bicategory $\mathcal{V}$ with weak duality involution to be an object $V \in \Obj(\mathcal{V})$ together with a morphism $\rho: \hat{\mathcal{G}} \rightarrow 1\Aut^{\textnormal{gen}}_{\mathcal{V}}(V)$ of $\mathbb{Z}_2$-graded categorical groups.

Finally, we state a $k$-linear version of the above definitions. We restrict to categorical groups of the form $\mathcal{G}(\mathsf{G}, \alpha)$ with Real structure $\mathcal{G}(\hat{\mathsf{G}}, \hat{\alpha})$.

\begin{Def}
A linear Real representation of $\mathcal{G}(\mathsf{G}, \alpha)$ on a $k$-linear bicategory $\mathcal{V}$ with contravariance is a contravariance preserving pseudofunctor $\rho: \underline{\mathcal{G}(\hat{\mathsf{G}}, \hat{\alpha})} \rightarrow \mathcal{V}$ with the additional property that $\Aut_{\hat{\mathcal{G}}}(\mathbf{1}) \simeq k^{\times}$ acts by scalar multiplication.
\end{Def}

Linear Real representations of $\mathcal{G}(\mathsf{G}, \alpha)$ form a bicategory $\mathsf{RRep}_{\mathcal{V},k}(\mathcal{G})$. The obvious analogue of Lemma \ref{lem:RealRepHomCatGrp}, with $1\Aut^{\textnormal{gen}}_{\mathcal{V}}(V)$ replaced by $\mathsf{GL}^{\textnormal{gen}}_k(V)$, holds.

To close this section, we describe an interpretation of Real representations of finite categorical groups which categorifies the homotopy fixed point perspective of Section \ref{sec:linearRealFGRep}. Fix an element $\varsigma \in \Obj(\hat{\mathcal{G}})$ such that $\pi(\varsigma)=-1$ together with a weak inverse $\overline{\varsigma}$. Define a biequivalence $F^{\varsigma}: 
\mathcal{G}^{\co} \rightarrow \mathcal{G}$ by assigning to $x: \pt \rightarrow \pt$ and $f: x \Rightarrow y$ in $\mathcal{G}^{\co}$ the $1$- and $2$-morphisms $(\varsigma \otimes x) \otimes \overline{\varsigma}$ and $(\varsigma \otimes f^{-1})  \otimes \overline{\varsigma}$ in $\mathcal{G}$, respectively. Noting that $F^{\varsigma} \circ_0 (F^{\varsigma})^{\co}$ is the adjoint action $\Ad_{\varsigma^2}=(\varsigma^2 \otimes -)\otimes \overline{\varsigma}^2$, the element $\varsigma^2$ and the associator for $\mathcal{G}$ induce a pseudonatural isomorphism $\varsigma^2: 1_{\mathcal{G}} \Rightarrow F^{\varsigma} \circ_0 (F^{\varsigma})^{\co}$. The biequivalence $F^{\varsigma}$ can be used to define a weak duality involution on $\mathsf{Rep}_{\mathcal{V}}(\mathcal{G})$ as follows, giving a $\varsigma$-twisted version of \cite[Example 2.6]{shulman2018}. The duality involution takes a pseudofunctor $\rho: \mathcal{G} \rightarrow \mathcal{V}$ to the composition
\[
\mathcal{G} \xrightarrow[]{(F^{\varsigma})^{\co}} \mathcal{G}^{\co} \xrightarrow[]{\rho^{\co}} \mathcal{V}^{\co} \xrightarrow[]{(-)^{\circ}} \mathcal{V}.
\]
The required adjoint equivalence $\tilde{\eta}$ and modification $\tilde{\zeta}$ are induced by whiskering with $\varsigma^2$ and the duality involution data of $\mathcal{V}$. For example, the component of $\tilde{\eta}$ at $\rho_V$ assigns to $\pt$ the $1$-morphism $\eta_V \circ_0 \rho_V(\varsigma^2): V \rightarrow V^{\circ \circ}$.

\begin{Prop}
\begin{enumerate}[label=(\roman*)]
\item Up to duality biequivalence, the weak duality involution on $\mathsf{Rep}_{\mathcal{V}}(\mathcal{G})$ is independent of the choice of $\varsigma \in \Obj(\hat{\mathcal{G}})$.
 
\item For any $\varsigma$ as above, there is a biequivalence $\mathsf{Rep}_{\mathcal{V}}(\mathcal{G})^{h \mathbb{Z}_2} \simeq \mathsf{RRep}_{\mathcal{V}}(\mathcal{G})$.
\end{enumerate}
\end{Prop}

\begin{proof}
Let $\varsigma_1, \varsigma_2 \in \Obj(\hat{\mathcal{G}})$ be as above with associated biequivalences $F^{\varsigma_1}, F^{\varsigma_2}: \mathcal{G}^{\co} \rightarrow \mathcal{G}$. After fixing an equivalence $\overline{\varsigma}_1 \otimes \varsigma_1 \simeq \mathbf{1}$, the $1$-morphism $\varsigma_2 \otimes \overline{\varsigma}_1$ defines a pseudonatural isomorphism $\Ad_{\varsigma_2 \otimes \overline{\varsigma}_1} : F^{\varsigma_1} \Rightarrow F^{\varsigma_2}$. The remaining components of the duality biequivalence are induced by whiskering.

The second statement is proved in the same way as Proposition \ref{prop:symmGenSymmEquiv}. We will describe a biequivalence at the level of objects, leaving the description on $1$- and $2$-morphisms to the reader. Given a symmetric form $(\rho, \psi, \mu)$ in $\mathsf{Rep}_{\mathcal{V}}(\mathcal{G})$, with $\rho(\pt) =V$, the map $\psi(\pt)$ is an equivalence $V^{\circ} \rightarrow V$. For $\omega \in \Obj(\hat{\mathcal{G}})$ with $\pi(\omega)=-1$, define $\rho(\omega)$ to be the composition $\rho(\omega \otimes \overline{\varsigma}) \circ_0 \psi(\pt)$. The monoidal coherence $2$-isomorphisms $\psi_{\bullet,\bullet}$ are induced by $\mu$ and $\tilde{\eta}$. It is straightforward to verify that this indeed defines a Real representation of $\mathcal{G}$.
\end{proof}

\begin{Rem}
\begin{enumerate}[label=(\roman*)]
\item When restricted to trivially $\mathbb{Z}_2$-graded categorical groups, the above definitions recover the representation theory of finite categorical groups, as studied in \cite{elgueta2007}, \cite{ganter2008}, \cite{bartlett2011}.

\item While the above definitions make sense for categorical groups which are not finite, in the continuous case they should be supplemented with topological coherence conditions.
\end{enumerate}
\end{Rem}

\section{Real \texorpdfstring{$2$}{}-representation theory of finite groups}
\label{sec:Real2Rep}

We study the definitions of Section \ref{sec:Real2RepBasic} when the categorical group is a finite group. The more technical case of categorical groups is the focus of Section \ref{sec:RealProjReps}.

\subsection{Real $2$-representations}
\label{sec:basicDef}

Lemma \ref{lem:RealRepHomCatGrp} leads to an explicit description of a Real representation of the categorical group determined by a finite group, which we state as a new definition.

\begin{Def}
A Real $2$-representation of a finite group $\mathsf{G}$ on a $2$-category $\mathcal{V}$ with strict duality involution consists of the following data:
\begin{enumerate}[label=(\roman*)]
\item An object $V$ of $\mathcal{V}$.
\item For each $\omega \in \hat{\mathsf{G}}$, an equivalence $\rho(\omega) : \prescript{\pi(\omega)}{}{V} \rightarrow V$.

\item For each pair $\omega_1, \omega_2 \in \hat{\mathsf{G}}$, a $2$-isomorphism
\[
\psi_{\omega_2, \omega_1} : \rho(\omega_2) \circ \prescript{\pi(\omega_2)}{}{\rho(\omega_1)} \Longrightarrow \rho(\omega_2 \omega_1).
\]
\item A $2$-isomorphism $\psi_e: \rho(e) \Rightarrow 1_V$.
\end{enumerate}
This data is required to satisfy the following conditions:
\begin{enumerate}[label=(\alph*)]
\item For each triple $\omega_1, \omega_2, \omega_3 \in \hat{\mathsf{G}}$, the equality
\begin{equation}
\label{eq:assCond}
\psi_{\omega_3 \omega_2, \omega_1} \circ_1  \left( \psi_{\omega_3, \omega_2} \circ \prescript{\pi(\omega_3 \omega_2)}{}{\rho(\omega_1)} \right) = \psi_{\omega_3, \omega_2 \omega_1} \circ_1 \left( \rho(\omega_3) \circ \prescript{\pi(\omega_3)}{}{\psi}_{\omega_2, \omega_1}^{\pi(\omega_3)} \right)
\end{equation}
of $2$-isomorphisms $\rho(\omega_3) \circ \prescript{\pi(\omega_3)}{}{\rho(\omega_2)} \circ \prescript{\pi(\omega_3 \omega_2)}{}{\rho(\omega_1)} \Longrightarrow \rho(\omega_3 \omega_2 \omega_1)$ holds.

\item For each $\omega \in \hat{\mathsf{G}}$, the equalities
\begin{equation}
\label{eq:idenCond}
\psi_{e,\omega} = \psi_e \circ \rho(\omega), \qquad \psi_{\omega,e} = \rho(\omega) \circ \prescript{\pi(\omega)}{}{\psi_e}
\end{equation}
of $2$-isomorphisms $\rho(\omega) \Rightarrow \rho(\omega)$ hold.
\end{enumerate}
\end{Def}

Denote by $\psi_{\omega_3, \omega_2, \omega_1}$ the $2$-isomorphism defined by either side of equation \eqref{eq:assCond}.

\subsection{Real conjugation invariance of categorical traces}
\label{sec:conjInv}

We study categorical traces, as introduced by Ganter--Kapranov \cite{ganter2008} and Bartlett \cite{bartlett2011}, in the presence of duality involutions.

Let $x$ be an object of a $2$-category $\mathcal{V}$. As in \cite[\S 3.1]{ganter2008} and \cite[\S 4.1]{bartlett2011}, the categorical trace $\Tr(f)$ of a $1$-morphism $f: x \rightarrow x$ is the set of all $2$-morphisms from $1_x$ to $f$:
\[
\Tr(f) = 2 \Hom_{\mathcal{V}}(1_x, f).
\]
Given a $2$-morphism $u: f_1 \Rightarrow f_2$, define $\Tr(u) : \Tr(f_1) \rightarrow \Tr(f_2)$ to be $u \circ_1 (-)$. These definitions extend the categorical trace to a functor
\[
\Tr: 1\End_{\mathcal{V}}(x) \rightarrow \mathsf{Set}.
\]
If $\mathcal{V}$ is enriched in a category $\mathcal{A}$, then $\Tr$ takes values in $\mathcal{A}$. For example, when $\mathcal{V}$ is $k$-linear the functor $\Tr$ is $\mathsf{Vect}_k$-valued.

In \cite[\S 4.3]{ganter2008} and \cite[\S 4.3]{bartlett2011} a kind of conjugation invariance of categorical traces is established. We generalize this result in what follows by showing that categorical traces in $2$-categories with duality involutions (or contravariance) are invariant under Real conjugation.

Suppose then that $\mathcal{V}$ is a $2$-category with strict duality involution. Fix a sign $\epsilon \in \mathbb{Z}_2$. Let $f : x \rightarrow x$ be an equivalence. When $\epsilon=-1$ we also fix a quasi-inverse $\tilde{f}$ of $f$ and a $2$-isomorphism $\mu: \tilde{f} \circ f \Rightarrow 1_x$. Write
\[
f^{\nu}
=
\begin{cases}
f & \mbox{if } \nu=1, \\
\tilde{f} & \mbox{if } \nu =-1.
\end{cases}
\]
Let $h: \prescript{\epsilon}{}{x} \rightarrow y$ be an equivalence with quasi-inverse $k: y \rightarrow \prescript{\epsilon}{}{x}$ and $2$-isomorphisms $u: 1_y \Rightarrow h \circ k$ and $v: 1_{\prescript{\epsilon}{}{x}} \Rightarrow k \circ h$. This data can be used to define a map
\[
\Psi(h,k,u,v; \mu): \Tr(f) \rightarrow \Tr(h \circ \prescript{\epsilon}{}{f}^{\epsilon} \circ k),
\]
henceforth denoted by $\Psi(h)$. The map $\mu$ is required only when $\epsilon=-1$. Suppose that we are given a $2$-morphism $\phi \in \Tr(f)$. Interpret $u$ as a $2$-morphism $1_y \Longrightarrow h \circ 1_{{^{\epsilon}}x} \circ k$. When $\epsilon=1$ the map $\Psi(h)$ is defined by post-composing $u$ with $\phi$:
\[
\Psi(h)(\phi) = (h \circ_0 \phi \circ_0 k) \circ_1 u.
\]
This is the definition of \cite{ganter2008}, \cite{bartlett2011}. If instead $\epsilon=-1$, then we can form the composition
\[
1_{x^{\circ}} \xLongrightarrow[]{\mu^{\circ}} \tilde{f}^{\circ} \circ f^{\circ}  \xLongrightarrow[]{\tilde{f}^{\circ} \circ_0 \phi^{\circ}} \tilde{f}^{\circ}.
\]
The map $\Psi(h)$ is defined by further pre-composing with $u$:
\[
\Psi(h)(\phi) =  \left( h \circ_0 \left( (\tilde{f}^{\circ} \circ_0 \phi^{\circ}) \circ_1
\mu^{\circ} \right) \circ_0 k \right) \circ_1 u.
\]

The following result generalizes \cite[Proposition 4.10]{ganter2008} and \cite[Proposition 4.3]{bartlett2011}. A further generalization (with a different proof) will be given in Theorem \ref{thm:catCharLoopGrpd} below.

\begin{Prop}
\label{prop:genConjInv}
For each pair of equivalences $f: x \rightarrow x$ and $h:{^{\epsilon}}x \rightarrow y$ with quasi-inverse data as above, the map
\[
\Psi(h) : \Tr(f) \rightarrow \Tr(h \circ {^{\epsilon}}f^{\epsilon} \circ k)
\]
is a bijection. Moreover, $\Psi(1_x) = 1_{\Tr(f)}$ and, given equivalences ${^{\epsilon_1}}x \xrightarrow[]{h_1} y_1$ and ${^{\epsilon_2}}y_1 \xrightarrow[]{h_2} y_2$ with quasi-inverse data, the equality $
\Psi(h_2) \circ \Psi(h_1) = \Psi(h_2 \circ {^{\epsilon_2}}h_1)$ holds.
\end{Prop}

\begin{proof}
That $\Psi(h)$ is a bijection follows from the assumption that $h$ is an equivalence. The equality $\Psi(1_x) = 1_{\Tr(f)}$ is clear from the construction.

To explain the precise meaning of the final statement, we need to describe the quasi-inverse data implicit in the definition of $\Psi(h_2 \circ {^{\epsilon_2}}h_1)$. Since $h_1$ and $h_2$ are equivalences, so too is $h_2 \circ \prescript{\epsilon_2}{}{h_1}$. We take $
\prescript{\epsilon_2}{}{k_1} \circ k_2$ as the quasi-inverse of $h_2 \circ \prescript{\epsilon_2}{}{h_1}$ with $2$-isomorphisms
\[
u: 1_{y_2} \xLongrightarrow[]{u_2} h_2 \circ k_2 \xLongrightarrow[]{h_2 \circ_0 \prescript{\epsilon_2}{}{u_1^{\epsilon_2}} \circ_0 k_2} h_2 \circ \prescript{\epsilon_2}{}{h_1} \circ \prescript{\epsilon_2}{}{k_1} \circ k_2
\]
and
\[
v: 1_{\prescript{\epsilon_2 \epsilon_1}{}{x}} \xLongrightarrow[]{\prescript{\epsilon_2}{}{v_1^{\epsilon_2}}} \prescript{\epsilon_2}{}{k_1} \circ \prescript{\epsilon_2}{}{h_1} \xLongrightarrow[]{\prescript{\epsilon_2}{}{k_1} \circ_0 v_2 \circ_0 \prescript{\epsilon_2}{}{h_1}} \prescript{\epsilon_2}{}{k_1} \circ k_2 \circ h_2 \circ \prescript{\epsilon_2}{}{h_1}.
\]
When $\epsilon_1\epsilon_2 =1$ no additional data is needed to define $\Psi(h_2 \circ \prescript{\epsilon_2}{}{h_1})$. If $\epsilon_1=-1$ and $\epsilon_2=1$, then we take for $\mu: \tilde{f} \circ f \Rightarrow 1_x$ the data used to define $\Psi(h_1)$. If instead $\epsilon_1=1$ and $\epsilon_2=-1$, then part of the data used to define $\Psi(h_2)$ is a quasi-inverse $\tilde{f}^{\prime}$ of $f^{\prime}=h_1 \circ f \circ k_1$ and a $2$-isomorphism $\mu^{\prime}: \tilde{f}^{\prime} \circ f^{\prime} \Longrightarrow 1_{y_1}$. Set $\tilde{f} = k_1 \circ \tilde{f}^{\prime} \circ h_1$ with $2$-isomorphism $\mu: \tilde{f} \circ f \Rightarrow 1_x$ given by the composition
\[
\tilde{f} \circ f \xLongrightarrow[]{\tilde{f} \circ f \circ v_1^{-1}} \tilde{f} \circ f \circ k_1 \circ h_1 = k_1 \circ \tilde{f}^{\prime} \circ h_1 \circ f \circ k_1 \circ h_1 \xLongrightarrow[]{k_1 \circ_0 \mu^{\prime} \circ_0 h_1} k_1 \circ h_1 \xLongrightarrow[]{v_1} 1_x.
\]
Then $\Psi(h_2 \circ {^{\epsilon_2}} h_1)$ is defined to be $\Psi(h_2 \circ {^{\epsilon_2}} h_1, 
{^{\epsilon_2}}k_1 \circ k_2, u, v; \mu)$. With the above definitions in place, it is now straightforward to verify the claimed equality $\Psi(h_2) \circ \Psi(h_1) = \Psi(h_2 \circ {^{\epsilon_2}} h_1)$.
\end{proof}

\begin{Rem}
While the categorical trace of an arbitrary $1$-morphism $f: x \rightarrow x$ is defined, Proposition \ref{prop:genConjInv} requires that $f$ be an equivalence.
\end{Rem}

Keeping the above notation, let us now further assume that $x=y$ and that the $1$-morphisms $f$ and $h$ graded commute in the sense that we are given a $2$-isomorphism
\[
\kappa : h \circ {^{\epsilon}}f^{\epsilon} \Longrightarrow f \circ h.
\]
We can then define a map $(h,\kappa)_* : \Tr(f) \rightarrow \Tr(f)$ by the composition
\[
\Tr(f) \xrightarrow[]{\Psi(h)} \Tr(h \circ {^{\epsilon}}f^{\epsilon} \circ k) \xrightarrow[]{\Tr(\kappa \circ_0 k)} \Tr(f \circ h \circ k) \xrightarrow[]{\Tr(f \circ_0 u^{-1})} \Tr(f).
\]
When $\epsilon=1$ this reduces to a construction of Ganter--Kapranov. Note that if $\mathcal{V}$ is enriched in $\mathcal{A}$, then $(h,\kappa)_*$ is a morphism in $\mathcal{A}$. In particular, when $\mathcal{A} = \mathsf{Vect}_k$ we can make the following definition, generalizing that of \cite[\S 3.6]{ganter2008}.

\begin{Def}
Let $\mathcal{V}$ be a $k$-linear $2$-category with strict duality involution and let $f: x \rightarrow x$ and $h : {^{\epsilon}}x \rightarrow x$ be graded commuting equivalences. Assuming that the vector space $\Tr(f)$ is finite dimensional, the joint trace of $(f,h)$ is
\[
\tr(f,h) = \tr \left( (h,\kappa)_*: \Tr(f) \rightarrow \Tr(f) \right).
\]
\end{Def}

\subsection{Real categorical characters}
\label{sec:RealCatChar}

Let $\rho$ be a Real $2$-representation of a finite group $\mathsf{G}$ on a $2$-category $\mathcal{V}$ with strict duality involution. For $g \in \mathsf{G}$, write $\Tr_{\rho}(g)$ for the set $\Tr(\rho(g))$. Fix $g \in \mathsf{G}$ and $\omega \in \hat{\mathsf{G}}$. By applying Proposition \ref{prop:genConjInv} to the equivalences
\[
f= \rho(g), \qquad \tilde{f} = \rho(g^{-1}), \qquad h = \rho(\omega), \qquad k= {^{\pi(\omega)}}\rho(\omega^{-1}),
\]
and the $2$-isomorphisms
\[
u=\psi_{\omega,\omega^{-1}}^{-1} \circ_0 \psi_e^{-1}, \qquad \mu = \psi_e \circ_0 \psi_{g^{-1},g}
\]
we obtain a map
\[
\Tr(\rho(g)) \rightarrow \Tr \left( \rho(\omega) \circ {^{\pi(\omega)}}\rho(g^{\pi(\omega)}) \circ {^{\pi(\omega)}}\rho(\omega^{-1}) \right).
\]
Post composing with $\Tr(\psi_{\omega, g^{\pi(\omega)}, \omega^{-1}})$ then gives a map $\beta_{g,\omega} : \Tr_{\rho}(g) \rightarrow \Tr_{\rho}(\omega g^{\pi(\omega)} \omega^{-1})$.

\begin{Def}
The Real categorical character of $\rho$ is the assignment
\[
g \mapsto \Tr_{\rho}(g), \qquad g \in \mathsf{G}
\]
together with the bijections
\[
\beta_{g, \omega} : \Tr_{\rho}(g) \rightarrow \Tr_{\rho}(\omega g^{\pi(\omega)} \omega^{-1}), \qquad (g, \omega) \in \mathsf{G} \times \hat{\mathsf{G}}.
\]
\end{Def}

The sets $\{\Tr_{\rho}(g)\}_{g \in \mathsf{G}}$, together with the bijections $\{\beta_{g_1,g_2}\}_{(g_1, g_2) \in \mathsf{G}^2}$, define the categorical character of the underlying $2$-representation of $\mathsf{G}$ \cite{ganter2008}, \cite{bartlett2011}. In particular, unlike the case of Real characters, Real categorical characters contain strictly more information than the categorical character of the underlying $2$-representation.

\begin{Prop}
\label{prop:catCharLoopGrpd}
The Real categorical character of a Real $2$-representation $\rho$ of $\mathsf{G}$ on $\mathcal{V}$ defines a functor
\[
\Tr(\rho) : \Lambda_{\pi}^{\refl} B \hat{\mathsf{G}} \rightarrow \mathsf{Set}.
\]
Moreover, if $\mathcal{V}$ is enriched in $\mathcal{A}$, then the functor $\Tr(\rho)$ takes values in $\mathcal{A}$.
\end{Prop}

\begin{proof}
Recall that objects of $\Lambda_{\pi}^{\refl} B\hat{\mathsf{G}}$ are labelled by elements $g \in \mathsf{G}$. Set $\Tr(\rho)(g) = \Tr_{\rho}(g)$. Given a morphism $\omega: g \rightarrow \omega g^{\pi(\omega)} \omega^{-1}$ in $\Lambda_{\pi}^{\refl} B \hat{\mathsf{G}}$, set $\Tr(\rho)(\omega) = \beta_{g, \omega}$. That these assignments define a functor $\Tr(\rho): \Lambda_{\pi}^{\refl} B \hat{\mathsf{G}} \rightarrow \mathsf{Set}$ follows from Proposition \ref{prop:genConjInv}. The final statement is clear.
\end{proof}

For example, when $\mathcal{V}$ is $k$-linear, Proposition \ref{prop:catCharLoopGrpd} states that the Real categorical character of a linear Real $2$-representation of $\mathsf{G}$ is a vector bundle over $\Lambda_{\pi}^{\refl} B \hat{\mathsf{G}}$.

\begin{Ex}
Let $\mathsf{Vect}_{\mathbb{F}_1}(X)$ be the category of $\mathbb{F}_1$-vector bundles over a finite set $X$, as in \cite{kapranovUnpub}. An action of a finite $\mathbb{Z}_2$-graded group $\hat{\mathsf{G}}$ on $X$ defines a Real $2$-representation $\rho$ of $\mathsf{G}$ on $\mathsf{Vect}_{\mathbb{F}_1}(X) \in \Obj(\mathsf{Cat})$ by setting
\[
\rho(g) = (g^{-1})^*, \qquad \rho(\omega) = (-)^{\vee} \circ (\omega^{-1})^*
\]
for $g \in \mathsf{G}$ and $\omega \in \hat{\mathsf{G}} \backslash \mathsf{G}$, where $(-)^{\vee} = \Hom_{\mathbb{F}_1}(-,\mathbb{F}_1)$. As $(-)^{\vee}$ squares to the identify functor, the $2$-isomorphisms $\psi_{\bullet,\bullet}$ are canonical. The Real categorical character $\Tr(\rho)$ is the unoriented loop groupoid of the double cover $X \git \mathsf{G} \rightarrow X \git \hat{\mathsf{G}}$, viewed as a groupoid over $\Lambda_{\pi}^{\refl} B \hat{\mathsf{G}}$. Interpreting the trace of an endomorphism of a finite set as the cardinality of its fixed point set, the joint traces $\tr(\rho(g), \rho(\omega)) = \chi_{\rho}(g,\omega)$ of Section \ref{sec:conjInv} compute the cardinality of joint fixed point sets:
\[
\chi_{\rho}(g,\omega) = \# X^{g,\omega}.
\]

In physical terminology, this example describes an orientifold of the $\mathbb{F}_1$-analogue of the $\mathsf{G}$-equivariant $B$-model on $X$. From this point of view, the Real part of $\chi_{\rho}$ computes the charge of the orientifold plane.
\end{Ex}

\section{Twisted Real \texorpdfstring{$2$}{}-representation theory of finite groups}
\label{sec:RealProjReps}

In this section we study linear Real representations of finite categorical groups. This recovers the $k$-linear version of the results of Section \ref{sec:Real2Rep} when the categorical group is a trivial extension of a finite group by $B k^{\times}$.

\subsection{Basic definitions}
\label{sec:RealProj2Rep}

Fix a field $k$. The following is a Real variant of definitions of Frenkel--Zhu \cite[Definition 2.8]{frenkel2012} and Ganter--Usher \cite[Definition 4.1]{ganter2016}.

\begin{Def}
A twisted Real $2$-representation of a finite group $\mathsf{G}$ on a $k$-linear $2$-category $\mathcal{V}$ with strict duality involution consists of data $V \in \Obj(\mathcal{V})$, $\rho(\omega)$, $\psi_{\omega_2, \omega_1}$ and $\psi_e$ as in Section \ref{sec:basicDef}, with the constraint \eqref{eq:idenCond} unchanged but with the constraint \eqref{eq:assCond} replaced by the condition that
\begin{equation}
\label{eq:projAssCond}
\hat{\alpha}([\omega_3 \vert \omega_2 \vert \omega_1]) \cdot \psi_{\omega_3 \omega_2, \omega_1} \left( \psi_{\omega_3, \omega_2} \circ \prescript{\pi(\omega_3 \omega_2)}{}{\rho(\omega_1)} \right) =\psi_{\omega_3, \omega_2 \omega_1}  \left( \rho(\omega_3) \circ \prescript{\pi(\omega_3)}{}{\psi}^{\pi(\omega_3)}_{\omega_2, \omega_1} \right)
\end{equation}
for some function $\hat{\alpha}: \hat{\mathsf{G}} \times \hat{\mathsf{G}} \times \hat{\mathsf{G}} \rightarrow k^{\times}$.
\end{Def}

We call $\hat{\alpha}$, which we regard as a $3$-cochain on $B \hat{\mathsf{G}}$, the Real $2$-Schur multiplier of the twisted Real $2$-representation $\rho$.

In terms of string diagrams, the $2$-isomorphisms $\psi_{\omega_2, \omega_1}$ and $\psi_e$ are
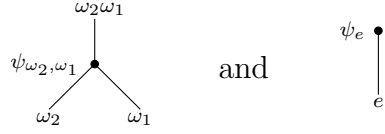
\[
\begin{gathered}
\begin{tikzpicture}[scale=0.15,inner sep=0.35mm, place/.style={circle,draw=black,fill=black,thick}]
\draw (0,0) -- (4,4);
\draw (4,4) -- (8,0);
\draw (4,4) -- (4,8);
\draw (4,4) node [shape=circle,draw,fill] {};
\node [below] at (0,0) {$\scriptstyle  \omega_2$};
\node [below] at (8,0) {$\scriptstyle \omega_1$};
\node [below] at (4.5,9.65) {$\scriptstyle \omega_2 \omega_1$};
\node [left,label={[label distance=-11.5mm]0: $\scriptstyle \psi_{\omega_2, \omega_1}$}] at (4,4) {};
\end{tikzpicture}
\end{gathered}
\qquad \textnormal{and} \qquad
\begin{gathered}
\begin{tikzpicture}[scale=0.15,inner sep=0.35mm, place/.style={circle,draw=black,fill=black,thick}]
\draw (0,0) -- (0,5.65);
\draw (0,5.65) node [shape=circle,draw,fill] {};
\node [below] at (0,0) {$\scriptstyle e$};
\node [left,label={[label distance=-5.5mm]0: $\scriptstyle \psi_e$}] at (0,5.65) {};
\end{tikzpicture}
\end{gathered}
\]
respectively. Equation \eqref{eq:projAssCond} will be written as
\begin{equation}
\label{diag:associativity}
\begin{gathered}
\begin{tikzpicture}[scale=0.2,inner sep=0.35mm, place/.style={circle,draw=black,fill=black,thick}]
\draw (0,0) -- (4,4);
\draw (4,4) -- (8,0);
\draw (4,4) -- (4,6);
\draw (2,2) -- (4,0);
\draw (2,2) node [shape=circle,draw,fill] {};
\draw (4,4) node [shape=circle,draw,fill] {};
\node [below] at (0,0) {$\scriptstyle \omega_3$};
\node [below] at (4,0) {$\scriptstyle \omega_2$};
\node [below] at (8,0) {$\scriptstyle \omega_1$};
\node [left,label={[label distance=-14.5mm]0: $\scriptstyle \psi_{\omega_3 \omega_2, \omega_1}$}] at (4,4) {};
\node [left,label={[label distance=-12.0mm]0: $\scriptstyle \psi_{\omega_3, \omega_2}$}] at (2,2) {};
\end{tikzpicture}
\end{gathered}
\xrightarrow[]{\; \hat{\alpha}([\omega_3 \vert \omega_2 \vert \omega_1]) \;}
\begin{gathered}
\begin{tikzpicture}[scale=0.2,inner sep=0.35mm, place/.style={circle,draw=black,fill=black,thick}]
\draw (0,0) -- (4,4);
\draw (4,4) -- (8,0);
\draw (4,4) -- (4,6);
\draw (6,2) -- (4,0);
\draw (4,4) node [shape=circle,draw,fill] {};
\draw (6,2) node [shape=circle,draw,fill] {};
\node [below] at (0,0) {$\scriptstyle \omega_3$};
\node [below] at (4,0) {$\scriptstyle \omega_2$};
\node [below] at (8,0) {$\scriptstyle \omega_1$};
\node [right,label={[label distance=0.4mm]0: $\scriptstyle \psi_{\omega_3,  \omega_2 \omega_1}$}] at (4,4) {};
\node [right,label={[label distance=0.4mm]0: $\scriptstyle \prescript{\pi(\omega_3)}{}{\psi}_{\omega_2, \omega_1}^{\pi(\omega_3)}$}] at (6,2) {};
\end{tikzpicture}
\end{gathered}
\end{equation}
the arrow indicating that the $2$-morphism on the right is $\hat{\alpha}([\omega_3 \vert \omega_2 \vert \omega_1])$ times that on the left. When computing with string diagrams, labels of $1$-morphisms will often be omitted when they can be reconstructed from the labelled data in a string diagram. For example, the $2$-morphism $\psi_e \circ \psi_{\omega, \omega^{-1}}$ will often be drawn as
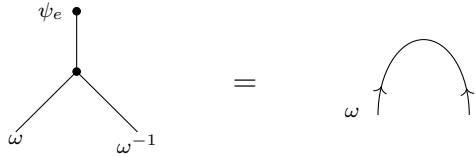
\[
\begin{gathered}
\begin{tikzpicture}[scale=0.2,inner sep=0.35mm, place/.style={circle,draw=black,fill=black,thick}]
\draw[black] (0,0) -- (4,4);
\draw[black] (4,4) -- (8,0);
\draw[black] (4,4) -- (4,8);
\draw (4,4) node [shape=circle,draw,fill] {};
\draw (4,8) node [shape=circle,draw,fill] {};
\node [below] at (0,0) {$\scriptstyle \omega$};
\node [below] at (8,0) {$\scriptstyle \omega^{-1}$};
\node [left,label={[label distance=-11.5mm]0: $\scriptstyle \psi_e$}] at (7,8) {};
\end{tikzpicture}
\end{gathered}
\qquad
=
\qquad
\begin{gathered}
\begin{tikzpicture}[scale=0.2,color=black, baseline]
\draw[black, decoration={markings, mark=at position 0.15 with {\arrow{>}}}, decoration={markings, mark=at position 0.895 with {\arrow{<}}},        postaction={decorate}](0,0) [partial ellipse=180:0:3 and 5];
\node at (-4.7,0.2) {$\scriptstyle \omega$};
\end{tikzpicture}
\end{gathered}
\]

\begin{Lem}
\label{lem:RealProjCocycle}
The Real $2$-Schur multiplier defines a $3$-cocycle $\hat{\alpha} \in Z^3(B \hat{\mathsf{G}}, k_{\pi}^{\times})$.
\end{Lem}

\begin{proof}
We need to verify that the equality
\[
\hat{\alpha}([\omega_4 \omega_3 \vert \omega_2 \vert \omega_1]) \hat{\alpha}([\omega_4 \vert \omega_3 \vert \omega_2 \omega_1]) =  \hat{\alpha}([\omega_3 \vert \omega_2 \vert \omega_1])^{\pi(\omega_4)} \hat{\alpha}([\omega_4 \vert \omega_3 \omega_2 \vert \omega_1]) \hat{\alpha}([\omega_4 \vert \omega_3 \vert \omega_2])
\]
holds for all $\omega_1, \omega_2, \omega_3, \omega_4 \in \hat{\mathsf{G}}$. This can be proved using string diagrams, similarly to \cite[Proposition 4.3]{ganter2016}, the corresponding statement for twisted $2$-representations. Repeated application of equation \eqref{diag:associativity} gives the following commutative diagram of string diagrams:
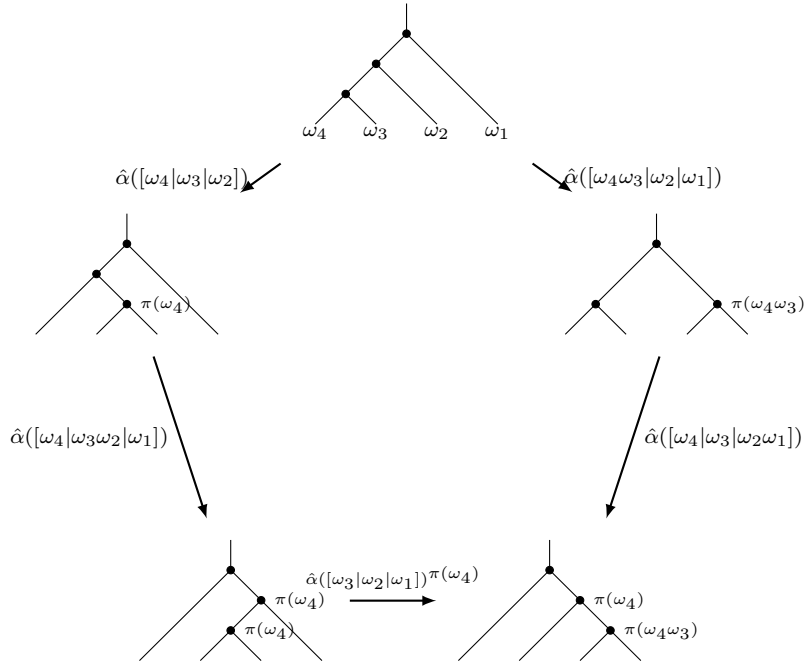
\[
\begin{tikzpicture}[scale=0.9]
\node (P2) at (1*18:4.3cm) { $
\begin{tikzpicture}
\node (n_2) {$
\begin{tikzpicture}[scale=0.20,inner sep=0.35mm, place/.style={circle,draw=black,fill=black,thick}]
\draw (0,0) -- (6,6);
\draw (6,6) -- (12,0);
\draw (6,6) -- (6,8);
\draw (4,0) -- (2,2);
\draw (10,2) -- (8,0);
\draw (2,2) node [shape=circle,draw,fill] {};
\draw (6,6) node [shape=circle,draw,fill] {};
\draw (10,2) node [shape=circle,draw,fill] {};
\node [right,label={[label distance=0.5mm]0: $\scriptscriptstyle \pi(\omega_4 \omega_3)$}] at (10,2) {};
\end{tikzpicture}
$};
\end{tikzpicture} 
$}; 

\node (P1) at (1*18+1*72:4.3cm) {$
\begin{tikzpicture}
\node (n_1) {$
\begin{tikzpicture}[scale=0.2,inner sep=0.35mm, place/.style={circle,draw=black,fill=black,thick}]
\draw (0,0) -- (6,6);
\draw (6,6) -- (12,0);
\draw (6,6) -- (6,8);
\draw (4,0) -- (2,2);
\draw (8,0) -- (4,4);
\draw (2,2) node [shape=circle,draw,fill] {};
\draw (4,4) node [shape=circle,draw,fill] {};
\draw (6,6) node [shape=circle,draw,fill] {};
\node [below] at (0,0) {$\scriptstyle \omega_4$};
\node [below] at (4,0) {$\scriptstyle \omega_3$};
\node [below] at (8,0) {$\scriptstyle \omega_2$};
\node [below] at (12,0) {$\scriptstyle \omega_1$};
\end{tikzpicture}
$};
\end{tikzpicture}
$}; 

\node (P4) at (1*18+2*72:4.3cm) {$
\begin{tikzpicture}
\node (n_4) {$
\begin{tikzpicture}[scale=0.20,inner sep=0.35mm, place/.style={circle,draw=black,fill=black,thick}]
\draw (0,0) -- (6,6);
\draw (6,6) -- (12,0);
\draw (6,6) -- (6,8);
\draw (4,4) -- (8,0);
\draw (6,2) -- (4,0);
\draw (6,2) node [shape=circle,draw,fill] {};
\draw (4,4) node [shape=circle,draw,fill] {};
\draw (6,6) node [shape=circle,draw,fill] {};
\node [right,label={[label distance=0.5mm]0: $\scriptscriptstyle \pi(\omega_4)$}] at (6,2) {};
\end{tikzpicture}
$};
\end{tikzpicture}
$}; ; 

\node (P5) at (1*18+3*72:4.3cm) {$
\begin{tikzpicture}
\node (n_5) {$
\begin{tikzpicture}[scale=0.20,inner sep=0.35mm, place/.style={circle,draw=black,fill=black,thick}]
\draw (0,0) -- (6,6);
\draw (6,6) -- (12,0);
\draw (6,6) -- (6,8);
\draw (8,4) -- (4,0);
\draw (6,2) -- (8,0);
\draw (6,2) node [shape=circle,draw,fill] {};
\draw (8,4) node [shape=circle,draw,fill] {};
\draw (6,6) node [shape=circle,draw,fill] {};
\node [right,label={[label distance=0.5mm]0: $\scriptscriptstyle \pi(\omega_4)$}] at (8,4) {};
\node [right,label={[label distance=0.5mm]0: $\scriptscriptstyle \pi(\omega_4)$}] at (6,2) {};
\end{tikzpicture}
$};
\end{tikzpicture}
$}; 

\node (P3) at (1*18+4*72:4.3cm) {$
\begin{tikzpicture}
\node (n_3) {$
\begin{tikzpicture}[scale=0.20,inner sep=0.35mm, place/.style={circle,draw=black,fill=black,thick}]
\draw (0,0) -- (6,6);
\draw (6,6) -- (12,0);
\draw (6,6) -- (6,8);
\draw (8,4) -- (4,0);
\draw (8,0) -- (10,2);
\draw (6,6) node [shape=circle,draw,fill] {};
\draw (8,4) node [shape=circle,draw,fill] {};
\draw (10,2) node [shape=circle,draw,fill] {};
\node [right,label={[label distance=0.5mm]0: $\scriptscriptstyle \pi(\omega_4)$}] at (8,4) {};
\node [right,label={[label distance=0.5mm]0: $\scriptscriptstyle \pi(\omega_4 \omega_3)$}] at (10,2) {};
\end{tikzpicture}
$};
\end{tikzpicture}
$}; 
\draw [-latex, thick] (P1) -- node[right]{$\scriptstyle \hat{\alpha}([\omega_4 \omega_3 \vert \omega_2 \vert \omega_1])$} (P2);
\draw [-latex, thick] (P2) -- node[right]{$\scriptstyle \hat{\alpha}([\omega_4 \vert \omega_3 \vert \omega_2 \omega_1])$} (P3);
\draw [-latex, thick] (P1) -- node[left]{$\scriptstyle \hat{\alpha}([\omega_4 \vert \omega_3 \vert \omega_2])$} (P4);
\draw [-latex, thick] (P4) -- node[left]{$\scriptstyle \hat{\alpha}([\omega_4 \vert \omega_3 \omega_2 \vert \omega_1])$} (P5);
\draw [-latex, thick] (P5) -- node[above]{$\scriptscriptstyle \hat{\alpha}([\omega_3 \vert \omega_2 \vert \omega_1])^{\pi(\omega_4)}$} (P3);
\end{tikzpicture}
\]
A node labelled by $-1$ indicates that it is $\psi_{\bullet, \bullet}^{- \op}$, instead of $\psi_{\bullet, \bullet}$, which is applied. For example, in the bottom right string diagram of the above diagram the node labelled by $\pi(\omega_4)$ corresponds to the $2$-isomorphism $\prescript{\pi(\omega_4)}{}{\psi}^{\pi(\omega_4)}_{\omega_3, \omega_2 \omega_1}$. The bottom arrow is multiplication by $\hat{\alpha}([\omega_3 \vert \omega_2 \vert \omega_1])^{\pi(\omega_4)}$, since it is the $\pi(\omega_4)$\textsuperscript{th} power of equation \eqref{diag:associativity} which is being applied. Commutativity of the above diagram implies the desired cocycle condition.
\end{proof}

By combining Theorem \ref{thm:catGroupClass} and Lemma \ref{lem:RealProjCocycle}, we find that an $\hat{\alpha}$-twisted Real $2$-representation of $\mathsf{G}$ determines a $\mathbb{Z}_2$-graded categorical group $\mathcal{G}(\hat{\mathsf{G}}, \hat{\alpha})$. The following proposition shows that the corresponding Real ($2$-)representation theories are equivalent.

\begin{Prop}
\label{prop:projToTwisted}
There is a canonical biequivalence between $\mathsf{RRep}_{\mathcal{V},k}(\mathcal{G}(\mathsf{G}, \alpha))$ and the bicategory of $\hat{\alpha}$-twisted Real $2$-representations of $\mathsf{G}$ on $\mathcal{V}$.
\end{Prop}

\begin{proof}
By construction, $\hat{\alpha} \in Z^3(B \hat{\mathsf{G}}, k_{\pi}^{\times})$ determines the associator of the monoidal groupoid $\mathcal{G}(\hat{\mathsf{G}}, \hat{\alpha})$. After observing that equation \eqref{diag:associativity} encodes the hexagon diagram for a monoidal functor $\mathcal{G}(\hat{\mathsf{G}}, \hat{\alpha}) \rightarrow \mathsf{GL}^{\textnormal{gen}}_k(V)$ which is compatible with the structure maps to $\mathbb{Z}_2$, the remainder of the proof is straightforward.
\end{proof}

To end this section, we record some basic string diagram identities.

\begin{Lem}
For all $g \in \mathsf{G}$
and $\omega_1, \omega_2 \in \hat{\mathsf{G}}$, the following identities hold:
\begin{enumerate}[label=(\roman*)]
\item
\begin{equation}
\label{diag:loopRemoval}
\begin{gathered}
\begin{tikzpicture}[scale=0.35,color=black, baseline]
\draw[decoration={markings, mark=at position 0.5 with {\arrow{>}}},        postaction={decorate}] (0,-0.5) -- (0,5.5);
\node at (-1.5,2.5) {$\scriptstyle \omega_2 \omega_1$};
\end{tikzpicture}
\end{gathered}
\qquad
=
\qquad
\begin{gathered}
\begin{tikzpicture}[scale=0.35,color=black, baseline]
\node at (0,0.5) [circle,draw,fill,inner sep=0.35mm] {};
\node at (0,4.5) [circle,draw,fill,inner sep=0.35mm] {};
\draw[decoration={markings, mark=at position 0.7 with {\arrow{>}}},        postaction={decorate}] (0,4.5) -- (0,5.5);
\draw[decoration={markings, mark=at position 0.5 with {\arrow{>}}},        postaction={decorate}] (0,-0.5) -- (0,0.5);\draw[decoration={markings, mark=at position 0.5 with {\arrow{>}}},        postaction={decorate}] (0,0.5) .. controls +(-2,2) and +(-2,-2) .. (0,4.5);
\draw[decoration={markings, mark=at position 0.5 with {\arrow{>}}},        postaction={decorate}] (0,0.5) .. controls +(2,2) and +(2,-2) .. (0,4.5);
\node at (0,6) {$\scriptstyle \omega_2 \omega_1$};
\node at (0,-1) {$\scriptstyle \omega_2 \omega_1$};
\node at (-2.5,2.0) {$\scriptstyle \omega_2$};
\node at (2.5,2.0) {$\scriptstyle \omega_1$};
\end{tikzpicture}
\end{gathered}
\end{equation}

\item 
\begin{equation}
\label{diag:crossRemoval}
\begin{gathered}
\begin{tikzpicture}[scale=0.2,inner sep=0.35mm, place/.style={circle,draw=black,fill=black,thick}]
\draw[decoration={markings, mark=at position 0.5 with {\arrow{>}}},        postaction={decorate}] (0,0) -- (3,3);
\draw[decoration={markings, mark=at position 0.5 with {\arrow{>}}},        postaction={decorate}] (6,0) -- (3,3);
\draw[decoration={markings, mark=at position 0.5 with {\arrow{>}}},        postaction={decorate}] (3,3) -- (3,6);
\draw[decoration={markings, mark=at position 0.5 with {\arrow{>}}},        postaction={decorate}] (3,6) -- (0,9);
\draw[decoration={markings, mark=at position 0.5 with {\arrow{>}}},        postaction={decorate}] (3,6) -- (6,9);
\draw (3,3) node [shape=circle,draw,fill] {};
\draw (3,6) node [shape=circle,draw,fill] {};
\node [below] at (-1,0) {$\scriptstyle  \omega_2$};
\node [below] at (7,0) {$\scriptstyle \omega_1$};
\node [below] at (-1,10) {$\scriptstyle \omega_2$};
\node [below] at (7,10) {$\scriptstyle \omega_1$};
\node [below] at (0,5) {$\scriptstyle \omega_2 \omega_1$};
\end{tikzpicture}
\end{gathered}
\qquad 
=
\qquad
\begin{gathered}
\begin{tikzpicture}[scale=0.2,color=black, baseline]
\draw[decoration={markings, mark=at position 0.5 with {\arrow{>}}},        postaction={decorate}] (0,0) -- (0,9);
\draw[decoration={markings, mark=at position 0.5 with {\arrow{>}}},        postaction={decorate}] (3,0) -- (3,9);
\node at (-2,4) {$\scriptstyle \omega_2$};
\node at (5,4) {$\scriptstyle \omega_1$};
\end{tikzpicture}
\end{gathered}
\end{equation}

\item
\begin{equation}
\label{diag:moveToVertex}
\begin{gathered}
\begin{tikzpicture}[scale=0.18,inner sep=0.35mm, place/.style={circle,draw=black,fill=black,thick}]
\draw[black, decoration={markings, mark=at position 0.15 with {\arrow{>}}}, decoration={markings, mark=at position 0.895 with {\arrow{<}}},        postaction={decorate}](0,0) [partial ellipse=180:0:3 and 5];
\draw[decoration={markings, mark=at position 0.6 with {\arrow{>}}},        postaction={decorate}] (-2,3.9) -- (-2,7);
\node at (-4.7,0.2) {$\scriptstyle \omega_2$};
\node at (4.9,0.2) {$\scriptstyle \omega_1$};
\node at (-1.9,8) {$\scriptstyle \omega_2 \omega_1$};
\draw (-1.9,3.9) node [shape=circle,draw,fill] {};
\end{tikzpicture}
\end{gathered}
\qquad
\xrightarrow[]{\hat{\alpha}([\omega_2 \omega_1 \vert \omega_1^{-1} \vert \omega_1])^{-1}}
\qquad
\begin{gathered}
\begin{tikzpicture}[scale=0.18,inner sep=0.35mm, place/.style={circle,draw=black,fill=black,thick}]
\draw[decoration={markings, mark=at position 0.5 with {\arrow{>}}},        postaction={decorate}] (0,0) -- (4,4);
\draw[decoration={markings, mark=at position 0.5 with {\arrow{>}}},        postaction={decorate}] (8,0) -- (4,4);
\draw[decoration={markings, mark=at position 0.5 with {\arrow{>}}},        postaction={decorate}] (4,4) -- (4,8);
\draw (4,4) node [shape=circle,draw,fill] {};
\node [below] at (0,0) {$\scriptstyle  \omega_2$};
\node [below] at (8,0) {$\scriptstyle \omega_1$};
\node [below] at (1,6) {$\scriptstyle \omega_2 \omega_1$};
\end{tikzpicture}
\end{gathered}
\qquad
\xleftarrow[]{\hat{\alpha}([\omega_2 \vert \omega_2^{-1} \vert \omega_2 \omega_1])}
\begin{gathered}
\begin{tikzpicture}[scale=0.2,inner sep=0.35mm, place/.style={circle,draw=black,fill=black,thick}]
\draw[black, decoration={markings, mark=at position 0.15 with {\arrow{>}}}, decoration={markings, mark=at position 0.895 with {\arrow{<}}},        postaction={decorate}](0,0) [partial ellipse=180:0:3 and 5];
\draw[decoration={markings, mark=at position 0.6 with {\arrow{>}}},        postaction={decorate}] (2,3.9) -- (2,7);
\node at (-4.7,0.2) {$\scriptstyle \omega_2$};
\node at (4.9,0.2) {$\scriptstyle \omega_1$};
\node at (2,8) {$\scriptstyle \omega_2 \omega_1$};
\draw (2,3.9) node [shape=circle,draw,fill] {};
\end{tikzpicture}
\end{gathered}
\end{equation}

\item
\begin{equation}
\label{diag:snakeRemoval}
\begin{gathered}
\begin{tikzpicture}[scale=0.2,inner sep=0.35mm, place/.style={circle,draw=black,fill=black,thick}]
\draw[black](0,0) [partial ellipse=180:0:4 and 4];
\draw[black](-8,0) [partial ellipse=180:360:4 and 4];
\draw[decoration={markings, mark=at position 0.5 with {\arrow{>}}},        postaction={decorate}] (-12,0) -- (-12,5);
\draw[decoration={markings, mark=at position 0.5 with {\arrow{<}}},        postaction={decorate}] (4,0) -- (4,-5);
\node [below] at (-13,4) {$\scriptstyle g$};
\end{tikzpicture}
\end{gathered}
\qquad
\xrightarrow[]{\hat{\alpha}([g \vert g^{-1} \vert g])}
\qquad
\begin{gathered}
\begin{tikzpicture}[scale=0.2,color=black, baseline]
\draw[decoration={markings, mark=at position 0.5 with {\arrow{>}}},        postaction={decorate}] (4,0) -- (4,12);
\node at (2.6,6) {$\scriptstyle g$};
\end{tikzpicture}
\end{gathered}
\end{equation}

\item
\begin{equation}
\label{diag:slide}
\begin{gathered}
\begin{tikzpicture}[scale=0.2,inner sep=0.35mm, place/.style={circle,draw=black,fill=black,thick}]
\draw[black, decoration={markings, mark=at position 0.15 with {\arrow{<}}}, decoration={markings, mark=at position 0.895 with {\arrow{>}}},        postaction={decorate}](0,0) [partial ellipse=180:360:3 and 5];
\draw[decoration={markings, mark=at position 0.8 with {\arrow{>}}},        postaction={decorate}] (0,-5) .. controls +(2,-5) and +(-2,-8) .. (10,0);
\node at (-4.7,0.2) {$\scriptstyle \omega_2$};
\node at (4.9,0.2) {$\scriptstyle \omega_1$};
\node at (10,1.5) {$\scriptstyle \omega_1^{-1} \omega_2^{-1}$};
\draw (0,-5) node [shape=circle,draw,fill] {};
\end{tikzpicture}
\end{gathered}
\qquad
\xrightarrow[]{\; \;\frac{\hat{\alpha}([\omega_1 \vert \omega_1^{-1} \vert \omega_2^{-1}])}{\hat{\alpha}([\omega_2 \vert \omega_1 \vert \omega_1^{-1} \omega_2^{-1}])}\;\;}
\qquad
\begin{gathered}
\begin{tikzpicture}[scale=0.2,inner sep=0.35mm, place/.style={circle,draw=black,fill=black,thick}]
\draw[black, decoration={markings, mark=at position 0.15 with {\arrow{<}}}, decoration={markings, mark=at position 0.895 with {\arrow{>}}},        postaction={decorate}](0,0) [partial ellipse=180:360:8 and 6];
\draw[decoration={markings, mark=at position 0.15 with {\arrow{<}}},        postaction={decorate}] (-10,0) .. controls +(2,-9) and +(-2,-8) .. (6,-4);
\node at (-11,1.0) {$\scriptstyle \omega_2$};
\node at (-7,1.0) {$\scriptstyle \omega_1$};
\node at (10,1.5) {$\scriptstyle \omega_1^{-1} \omega_2^{-1}$};
\draw (6,-4) node [shape=circle,draw,fill] {};
\end{tikzpicture}
\end{gathered}
\end{equation}
\end{enumerate}
\end{Lem}

\begin{proof}
The first two identities are obvious. For the remaining identities, see \cite[Lemma 4.8]{ganter2016}, \cite[Corollary 4.9]{ganter2016} and \cite[Corollary 4.10]{ganter2016}. See also \cite[\S\S 3.2.1, 3.4]{bartlett2011}
\end{proof}

\subsection{Real \texorpdfstring{$2$}{}-characters and \texorpdfstring{$2$}{}-class functions}
\label{sec:2Char}

We extend the theory of Real categorical characters to twisted Real $2$-representations. Instead of the direct approach of Section \ref{sec:RealCatChar}, we use string diagrams. See \cite[\S 4]{ganter2016} for the ungraded case.

We begin with some terminology.

\begin{Def}
A pair $(g, \omega) \in \mathsf{G} \times \hat{\mathsf{G}}$ is said to graded commute if $\omega g^{\pi(\omega)} = g \omega$.
\end{Def}

The group $\hat{\mathsf{G}}$ acts on $\mathsf{G} \times \hat{\mathsf{G}}$ by $\sigma \cdot (g, \omega) = (\sigma g^{\pi(\sigma)} \sigma^{-1}, \sigma \omega \sigma^{-1})$. This action preserves the subset $\hat{\mathsf{G}}^{(2)}$ of graded commuting pairs.

\begin{Def}
The Real categorical character of a twisted Real $2$-representation $\rho$ of $\mathsf{G}$ is the assignment
\[
g \mapsto \Tr_{\rho}(g) \in \Obj(\mathsf{Vect}_k), \qquad g \in \mathsf{G}
\]
together with the collection of $k$-linear isomorphisms
\[
\beta_{g, \omega} : \Tr_{\rho}(g) \rightarrow \Tr_{\rho}(\omega g^{\pi(\omega)} \omega^{-1}), \qquad (g, \omega) \in \mathsf{G} \times \hat{\mathsf{G}}
\]
defined by the string diagrams
\begin{equation}
\label{eq:charStriDiagPos}
\phi \longmapsto
\begin{tikzpicture}[scale=0.30,color=black, baseline]
\draw [decoration={markings, mark=at position 0.575 with {\arrow{<}}}, decoration={markings, mark=at position 0.945 with {\arrow{>}}},        postaction={decorate}] (0,0) ellipse (3 and 5);
\node (n1) at (0,1.0) [circle,draw,fill,inner sep=0.35mm] {};
\node[below=0.01mm of n1] {$\scriptstyle \phi$};
\node at (-2.75,2) [circle,draw,fill,inner sep=0.5mm] {};
\node at (-2.2,3.3) [circle,draw,fill,inner sep=0.5mm] {};
\draw[decoration={markings, mark=at position 0.5 with {\arrow{>}}},        postaction={decorate}] (n1) -- (-2.75,2);
\draw[decoration={markings, mark=at position 0.7 with {\arrow{>}}},        postaction={decorate}] (-2.2,3.3) -- (-3.5,5);
\node at (-4.5,-2) {$\scriptstyle \omega$};
\node at (4.7,-1.9) {$\scriptstyle \omega^{-1}$};
\node at (-0.9,2.5) {$\scriptstyle g$};
\node at (-3.5,6) {$\scriptstyle \omega g \omega^{-1}$};
\end{tikzpicture}
\qquad \omega \in \mathsf{G}
\end{equation}
and
\begin{equation}
\label{eq:charStriDiagNeg}
\phi \longmapsto
\begin{tikzpicture}[scale=0.30,color=black, baseline]
\draw [decoration={markings, mark=at position 0.575 with {\arrow{<}}}, decoration={markings, mark=at position 0.945 with {\arrow{>}}},        postaction={decorate}] (0,0) ellipse (3 and 5);
\node (n1) at (0,1.0) [circle,draw,fill,inner sep=0.35mm] {};
\node[above=0.01mm of n1] {$\scriptstyle \phi^{\circ}$};
\node at (-2.75,2) [circle,draw,fill,inner sep=0.5mm] {};
\node at (-2.2,3.3) [circle,draw,fill,inner sep=0.5mm] {};
\draw[decoration={markings, mark=at position 0.7 with {\arrow{>}}},decoration={markings, mark=at position 0.15 with {\arrow{<}}},        postaction={decorate}] (n1) .. controls +(-1,-5) and +(1,-3) .. (-2.75,2);
\draw[decoration={markings, mark=at position 0.7 with {\arrow{>}}},        postaction={decorate}] (-2.2,3.3) -- (-3.5,5);
\node at (-4.5,-2) {$\scriptstyle \omega$};
\node at (4.7,-1.9) {$\scriptstyle \omega^{-1}$};
\node at (0.7,-0.9) {$\scriptstyle g$};
\node at (-3.5,6) {$\scriptstyle \omega g^{-1} \omega^{-1}$};
\end{tikzpicture}
\qquad \omega \in \hat{\mathsf{G}} \backslash \mathsf{G}.
\end{equation}
\end{Def}

This definition recovers that of Section \ref{sec:RealCatChar} in the case of trivial Real $2$-Schur multiplier. The next definition formulates the joint trace of Section \ref{sec:conjInv} in the context of Real $2$-representation theory.

\begin{Def}
Assume that each vector space $\Tr_{\rho}(g)$, $g \in \mathsf{G}$, is finite dimensional. Then the Real $2$-character of $\rho$ is the collection of joint traces
\[
\chi_{\rho}(g,\omega) = \tr_{\Tr_{\rho}(g)}(\beta_{g, \omega}), \qquad (g, \omega) \in \hat{\mathsf{G}}^{(2)}.
\]
\end{Def}

Before proceeding, we note that the theory of $2$-characters is weaker than its non-categorical counterpart, in that inequivalent $2$-representations may have the same $2$-character. For an explicit example, see \cite[\S 5]{osorno2010}. Analogous statements apply to Real $2$-characters as can be seen, for example, by using Proposition \ref{prop:Real2VectChar} below.

\begin{Ex}
In this example we assume basic familiarity with homological matrix factorizations. For all necessary background, see \cite{polishchuk2012}, \cite{dyckerhoff2011}, \cite{carqueville2016}.

Let $k$ be a field of characteristic zero. Denote by $\mathsf{LG}_k$ the bicategory of Landau--Ginzburg models over $k$, as in \cite[\S 2.2]{carqueville2016}. Objects of $\mathsf{LG}_k$ are pairs $(R,W)$ consisting of a ring $R$ of the form $k\pser{x_1, \dots, x_n}$ for some $n \geq 0$ and a potential $W \in R$. The $1$-morphism category $1\Hom_{\mathsf{LG}_k}((R_1,W_1), (R_2,W_2))$ is the $2$-periodic triangulated category of finite rank matrix factorizations of $W_2 - W_1$:
\[
1\Hom_{\mathsf{LG}_k}((R_1,W_1), (R_2,W_2)) = \mathsf{HMF}(R_1 \hat{\otimes}_k R_2, W_2 - W_1).
\]
The composition $-\circ_0-$ is tensor product of matrix factorizations. For example, the identity $1$-morphism $1_{(R,W)}: (R,W) \rightarrow (R,W)$ is represented by the stabilized diagonal
\[
\Delta_W = {\bigwedge}_{R^e}^{\bullet} (\bigoplus_{i=1}^n R^e \cdot \theta_i), \qquad
d_{\Delta_W} = \sum_{i=1}^n (x_i - x_i^{\prime}) \cdot \theta_i^{\vee} + \sum_{i=1}^n \partial^{x,x^{\prime}}_{[i]} W \cdot \theta_i \wedge -.
\]
Here $\theta_i$ are Grassmann variables, $R^e = R \hat{\otimes}_k R \simeq k\pser{x_1, \dots, x_n, x_1^{\prime}, \dots, x_n^{\prime}}$ and
\[
\partial^{x,x^{\prime}}_{[i]} W = \frac{W(x^{\prime}_1, \dots, x^{\prime}_{i-1}, x_i, \dots, x_n, x^{\prime}) - W(x^{\prime}_1, \dots, x^{\prime}_i, x_{i+1}, \dots, x_n, x^{\prime})}{x_i - x_i^{\prime}}.
\]

Define a weak duality involution $(-)^{\vee}$ on $\mathsf{LG}_k$ as follows. On objects set $(R,W)^{\vee} = (R,-W)$. On $1$-morphism categories, $(-)^{\vee}$ acts as the linear (with respect to the ground ring) dual of matrix factorizations. Because of the Koszul sign rule, the dual of a matrix factorization of $W$ is canonically a matrix factorization of $-W$. The adjoint equivalence $\eta$ is induced by the canonical evaluation isomorphism from a finite rank free module to its double dual.

Suppose now that a finite $\mathbb{Z}_2$-graded group $\hat{\mathsf{G}}$ acts on $R=k\pser{x_1, \dots, x_n}$ by unital algebra automorphisms. Assume that $W \in R$ is a potential which satisfies
\[
\omega(W) = \pi(\omega) W, \qquad \omega \in \hat{\mathsf{G}}.
\]
That is, $\mathsf{G}$ and $\hat{\mathsf{G}} \backslash \mathsf{G}$ act by symmetries and antisymmetries of $W$, respectively. A potential together with a finite group of symmetries defines a Landau--Ginzburg orbifold model. Mathematically, such models can be studied within the framework of equivariant matrix factorizations. On the other hand, a potential with an action of $\hat{\mathsf{G}}$ as above defines a Landau--Ginzburg orientifold model \cite{hori2008}. Such models have not been studied in the mathematical literature. As an explicit example, take $a,b \geq 1$ and let $W= x^{2a+1} + x y^{2b} \in k\pser{x,y}$. Consider the exact sequence of multiplicative groups
\[
1 \rightarrow \mathbb{Z}_b \rightarrow \mathbb{Z}_{2b} \xrightarrow[]{\pi} \mathbb{Z}_2 \rightarrow 1.
\]
Let $\xi$ be a primitive $2b$\textsuperscript{th} root of unity, which we assume to lie in $k$. Then a $W$-compatible action of $\mathbb{Z}_{2b}$ on $k\pser{x,y}$ is given by $\xi \cdot (x,y) = (- x, \xi y)$.

Define a Real $2$-representation $\rho$ of $\mathsf{G}$ on $(R,W) \in \Obj(\mathsf{LG}_k)$ by letting $\omega \in \hat{\mathsf{G}}$ act by the $1$-morphism ${_{\omega}}\Delta_W \in \mathsf{HMF}(R \hat{\otimes}_k R, W - \pi(\omega) W)$ which is the pullback of $\Delta_W$ by $\omega \otimes 1$. Explicitly, ${_{\omega}}\Delta_W$ is equal to $\Delta_W$ as an $R^e$-module but has the twisted differential
\[
d_{{_{\omega}}\Delta_W} = \sum_{i=1}^n (\omega(x_i) - x_i^{\prime}) \theta_i^{\vee} + \sum_{i=1}^n \partial^{\omega(x),x^{\prime}}_{[i]} W \cdot \theta_i \wedge -.
\]
The coherence $2$-isomorphisms $\psi_{\bullet,\bullet}$ are induced by the associators in $\mathsf{LG}_k$. More generally, the maps $\psi_{\bullet,\bullet}$ can be twisted by a $2$-cocycle $\hat{\theta} \in Z^2(B \hat{\mathsf{G}}, k^{\times}_{\pi})$, thereby incorporating discrete torsion. We will not do this here; see, however, Section \ref{sec:2VectRep}.

Using \cite[Lemma 2.5.3]{polishchuk2012}, we find that the Real categorical character is the Hochschild homology,
\[
\Tr_{\rho}(g) \simeq HH_{\bullet}(\mathsf{MF}(R^g, W^g)), \qquad g \in \mathsf{G}.
\]
The maps
\[
\beta_{g,\omega} : HH_{\bullet}(\mathsf{MF}(R^g, W^g)) \rightarrow HH_{\bullet}(\mathsf{MF}(R^{\omega g^{\pi(\omega)} \omega^{-1}}, W^{\omega g^{\pi(\omega)} \omega^{-1}}))
\]
are the canonical $k$-linear isomorphisms. The pair $(R^g, W^g)$ is defined as follows. Choose coordinates $x_1, \dots, x_n$ of $R$ in which $g$ acts linearly and such that $\Span_k \{x_1, \dots, x_n\}^g = \Span_k \{x_{t+1}, \dots, x_n\}$. Then $R^g = R \slash (x_1, \dots, x_t)$ and $W^g$ is the image of $W$ in $R^g$. By \cite[\S 6.3]{dyckerhoff2011}, the Hochschild homology $HH_{\bullet}(\mathsf{MF}(R^g, W^g))$ is isomorphic to the Milnor algebra of $W^g$, supported in degree $n-t \mod 2$.

As for Real $2$-characters, we do not know a geometric interpretation of each joint trace $\chi_{\rho}(g,\omega)$. However, it follows from \cite[Theorem 2.5.4]{polishchuk2012} that we have
\[
\frac{1}{\vert \mathsf{G} \vert}
\sum_{(g, h) \in \mathsf{G}^{(2)}} \chi_{\rho}(g,h) = \dim_k HH^{\mathsf{G}}_{\bullet}(\mathsf{MF}(R,W)),
\]
where $\mathsf{G}^{(2)} \subset \mathsf{G}^2$ is the subset of commuting pairs. Similarly, the Real $2$-character computes the dimension of the $\mathsf{G}$-equivariant involutive Hochschild homology:
\[
\frac{1}{2 \vert \mathsf{G} \vert}
\sum_{(g, \omega) \in \hat{\mathsf{G}}^{(2)}} \chi_{\rho}(g,\omega) = \dim_k HH^{\mathsf{G},+}_{\bullet}(\mathsf{MF}(R,W)).
\]
The cohomology $HH^{\mathsf{G},+}_{\bullet}(\mathsf{MF}(R,W))$ controls the deformation theory of the equivariant matrix factorization category $\mathsf{MF}_{\mathsf{G}}(R,W)$, considered as a category with duality determined by $\hat{\mathsf{G}}$, and so controls the deformation theory of the corresponding Landau--Ginzburg orientifold model.
\end{Ex}

Our next goal is to give geometric interpretations of $\Tr(\rho)$ and $\chi_{\rho}$. We require some preliminary material. Let $\mathfrak{G}$ be a finite groupoid. Following Willerton \cite[\S 2.3.1]{willerton2008}, a $2$-cocycle $\theta \in Z^2(\mathfrak{G}, k^{\times})$ defines a $k^{\times}$-gerbe $\prescript{\theta}{}{\mathfrak{G}}$ over $\mathfrak{G}$. Explicitly, $\prescript{\theta}{}{\mathfrak{G}}$ is the category with
\[
\Obj(\prescript{\theta}{}{\mathfrak{G}}) =\Obj(\mathfrak{G}), \qquad
\Hom_{\prescript{\theta}{}{\mathfrak{G}}}(x_1, x_2) = k^{\times} \times \Hom_{\mathfrak{G}}(x_1, x_2)
\] 
and composition law
\[
(z_2, g_2) \circ (z_1, g_1) = \left( \theta([g_2 \vert g_1]) z_2 z_1, g_2 g_1 \right), \qquad z_1, z_2 \in k^{\times}, \; g_1, g_2 \in \Mor(\mathfrak{G}).
\]
A vector bundle over $\prescript{\theta}{}{\mathfrak{G}}$, also called a $\theta$-twisted vector bundle over $\mathfrak{G}$, is a functor $\prescript{\theta}{}{\mathfrak{G}} \rightarrow \mathsf{Vect}_k$ with the additional property that each subgroupoid $(B k^{\times})_{\vert x} \subset \prescript{\theta}{}{\mathfrak{G}}$, $x \in \Obj(\mathfrak{G})$, acts by scalar multiplication.

The following result generalizes Proposition \ref{prop:catCharLoopGrpd} to finite categorical groups. The analogous result in the ungraded setting is \cite[Theorem 4.17]{ganter2016}.

\begin{Thm}
\label{thm:catCharLoopGrpd}
The Real categorical character of an $\hat{\alpha}$-twisted Real $2$-representation $\rho$ of $\mathsf{G}$ defines a $\uptau_{\pi}^{\refl}(\hat{\alpha})$-twisted vector bundle over $\Lambda_{\pi}^{\refl} B \hat{\mathsf{G}}$,
\[
\Tr(\rho) : {^{\uptau_{\pi}^{\refl}(\hat{\alpha})}}\Lambda_{\pi}^{\refl}B \hat{\mathsf{G}} \rightarrow \mathsf{Vect}_k.
\]
\end{Thm}

\begin{proof}
The theorem is equivalent to twisted commutativity of the diagram
\[
\begin{tikzpicture}
\node (n1) at (0,0) {$\Tr_{\rho}(g)$};
\node (n2) at (9.5,0) {$\Tr_{\rho}(\omega_2 \omega_1 g^{\pi(\omega_2 \omega_1)} \omega_1^{-1} \omega_2^{-1})$};
\node (n3) at (4,-2) {$\Tr_{\rho}(\omega_1 g^{\pi(\omega_1)} \omega_1^{-1})$};
\draw[->] (n1) to node[sloped, anchor=center, above, scale=.7] (a1) [above] {$\beta_{g,\omega_2 \omega_1}$} (n2);
\draw[->] (n1) to node[sloped, anchor=center, above, scale=.7] (a2) [below] {$\beta_{g,\omega_1}$} (n3);
\draw[->] (n3) to node[sloped, anchor=center, above, scale=.7] (a3) [below] {$\beta_{\omega_1 g^{\pi(\omega_1)} \omega_1^{-1},\omega_2}$} (n2);
\draw[double,->,shorten <=6pt,shorten >=6pt] (a1) to node[scale=.7] [right] {$\uptau_{\pi}^{\refl}(\hat{\alpha})([\omega_2 \vert \omega_1]g)$} (n3);
\end{tikzpicture}
\]
for all $g \in \mathsf{G}$ and $\omega_1, \omega_2 \in \hat{\mathsf{G}}$. The vertical double arrow indicates that the top arrow is  $\uptau_{\pi}^{\refl}(\hat{\alpha})([\omega_2 \vert \omega_1]g)$ times the bottom composition. To prove twisted commutativity, suppose first that $\pi(\omega_2)=1$. In this case the expression for $\uptau_{\pi}^{\refl}(\hat{\alpha})([\omega_2 \vert \omega_1]g)$ differs from that of $\uptau(\alpha)([\omega_2 \vert \omega_1]g)$ only through the replacement of $g$ with $g^{\pi(\omega_1)}$ (see Section \ref{sec:twistLoopTran}). The desired equality can therefore be verified by a straightforward modification of the arguments used to prove \cite[Theorem 4.17]{ganter2016}.

Suppose then that $\pi(\omega_2)=-1$. Consider first the case $\pi(\omega_1)=1$. Let $\phi \in \Tr_{\rho}(g)$. Then $\beta_{\omega_1 g^{\pi(\omega_1)} \omega_1^{-1}, \omega_2} \left( \beta_{g, \omega_1} (\phi) \right)$ is computed by the string diagram
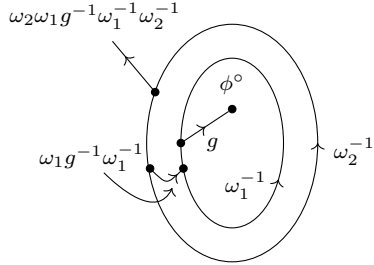
\[
\begin{tikzpicture}[scale=0.225,color=black, baseline]
\draw [decoration={markings, mark=at position 0.915 with {\arrow{>}}},        postaction={decorate}] (0,0) ellipse (3 and 5);
\draw [decoration={markings, mark=at position 0.0 with {\arrow{>}}},        postaction={decorate}] (0,0) ellipse (5 and 7);
\node (n1) at (0,2) [circle,draw,fill,inner sep=0.35mm] {};
\node[above=0.01mm of n1] {$\scriptstyle \phi^{\circ}$};
\node at (-3,0) [circle,draw,fill,inner sep=0.35mm] {};
\node at (-2.85,-1.5) [circle,draw,fill,inner sep=0.35mm] {};
\node at (-4.5,3) [circle,draw,fill,inner sep=0.35mm] {};
\node at (-4.8,-1.5) [circle,draw,fill,inner sep=0.35mm] {};
\draw[decoration={markings, mark=at position 0.8 with {\arrow{>}}},        postaction={decorate},] (-4.8,-1.5) .. controls +(1,-1) and +(-1,-1) .. (-2.85,-1.5);
\draw [decoration={markings, mark=at position 0.7 with {\arrow{>}}},        postaction={decorate}] (-4.5,3) -- (-7,6);
\draw [decoration={markings, mark=at position 0.6 with {\arrow{<}}},        postaction={decorate}] (n1) -- (-3,0);
\node at (7.15,-0.5) {$\scriptstyle \omega^{-1}_2$};
\node at (0.8,-2.5) {$\scriptstyle \omega^{-1}_1$};
\node at (-8,7.5) {$\scriptstyle \omega_2 \omega_1 g^{-1} \omega_1^{-1} \omega_2^{-1}$};
\node at (-1.15,-0.1) {$\scriptstyle g$};
\node at (-8.2,-0.85) {$\scriptstyle \omega_1 g^{-1} \omega_1^{-1}$};
\draw[thin,->] (-7.5,-1.75) .. controls +(0.9,-0.9) and +(-0.9,-1.9) .. (-3.5,-2.65);
\end{tikzpicture}
\]
In this diagram, and those which follow, the exterior region is labelled by the category $V$ while the interior regions are labelled by $V^{\op}$. Using equations \eqref{diag:crossRemoval} and \eqref{diag:slide}, the previous string diagram is seen to equal
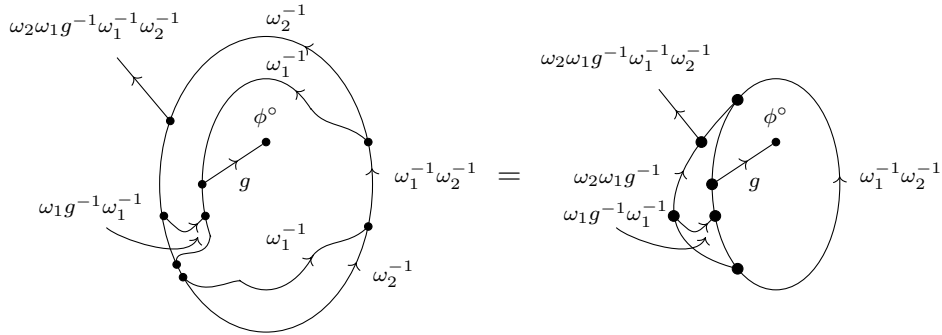
\[
\begin{tikzpicture}[scale=0.28,color=black, baseline]
\draw [decoration={markings, mark=at position 0.02 with {\arrow{>}}}, decoration={markings, mark=at position 0.2 with {\arrow{>}}},decoration={markings, mark=at position 0.9 with {\arrow{>}}}, postaction={decorate}] (0,0) ellipse (5 and 7);
\draw[black, decoration={markings, mark=at position 0.1 with {\arrow{>}}}, postaction={decorate}](0,0) [partial ellipse=45:210:3 and 5];
\draw[black] (0,0) [partial ellipse=245:318:3 and 5];
\node (n1) at (0,2) [circle,draw,fill,inner sep=0.35mm] {};
\node[above=0.01mm of n1] {$\scriptstyle \phi^{\circ}$};
\node at (-3,0) [circle,draw,fill,inner sep=0.35mm] {};
\node at (-2.85,-1.5) [circle,draw,fill,inner sep=0.35mm] {};
\node at (-4.5,3) [circle,draw,fill,inner sep=0.35mm] {};
\node at (-4.8,-1.5) [circle,draw,fill,inner sep=0.35mm] {};
\node at (4.8,2) [circle,draw,fill,inner sep=0.35mm] {};
\node at (4.8,-2) [circle,draw,fill,inner sep=0.35mm] {};
\node at (-4.2,-3.8) [circle,draw,fill,inner sep=0.35mm] {};
\node at (-3.9,-4.4) [circle,draw,fill,inner sep=0.35mm] {};
\draw[decoration={markings, mark=at position 0.8 with {\arrow{>}}},        postaction={decorate},] (-4.8,-1.5) .. controls +(1,-1) and +(-1,-1) .. (-2.85,-1.5);
\draw [decoration={markings, mark=at position 0.7 with {\arrow{>}}},        postaction={decorate}] (-4.5,3) -- (-7,6);
\draw [decoration={markings, mark=at position 0.5 with {\arrow{<}}},        postaction={decorate}] (n1) -- (-3,0);
\draw[decoration={markings, mark=at position 0.02 with {\arrow{>}}}, postaction={decorate}] (2.10,-3.5) .. controls +(0.7,1) and +(-1,-1) .. (4.8,-2);
\draw (2.1,3.6) .. controls +(0.7,-1) and +(-1,1) .. (4.8,2);
\draw (-1.18,-4.58) .. controls +(-0.5,0.1) and +(1,-1) .. (-3.9,-4.4);
\draw (-2.62,-2.4) .. controls +(-0.1,-1.4) and +(-0.2,1) .. (-4.2,-3.8);
\node at (-8.2,-1.05) {$\scriptstyle \omega_1 g^{-1} \omega_1^{-1}$};
\draw[thin,->] (-7.5,-2.05) .. controls +(0.5,-0.5) and +(-0.5,-1.25) .. (-3.2,-2.5);
\node at (8.0,0.4) {$\scriptstyle \omega_1 ^{-1} \omega_2^{-1}$};
\node at (1.0,5.7) {$\scriptstyle \omega_1^{-1}$};
\node at (1.0,7.9) {$\scriptstyle  \omega_2^{-1}$};
\node at (-8,7.5) {$\scriptstyle \omega_2 \omega_1 g^{-1} \omega_1^{-1} \omega_2^{-1}$};
\node at (-1.0,0.0) {$\scriptstyle g$};
\node at (1.0,-2.5) {$\scriptstyle \omega_1^{-1}$};
\node at (6.0,-4.25) {$\scriptstyle  \omega_2^{-1}$};
\end{tikzpicture}
=
\begin{tikzpicture}[scale=0.28,color=black, baseline]
\draw [decoration={markings, mark=at position 0.015 with {\arrow{>}}},        postaction={decorate}] (0,0) ellipse (3 and 5);
\node (n1) at (0,2) [circle,draw,fill,inner sep=0.35mm] {};
\node[above=0.01mm of n1] {$\scriptstyle \phi^{\circ}$};
\node at (-3,0) [circle,draw,fill,inner sep=0.5mm] {};
\node at (-2.85,-1.5) [circle,draw,fill,inner sep=0.5mm] {};
\node at (-1.8,-4) [circle,draw,fill,inner sep=0.5mm] {};
\node at (-1.8,4) [circle,draw,fill,inner sep=0.5mm] {};
\node at (-4.8,-1.5) [circle,draw,fill,inner sep=0.5mm] {};
\node at (-3.5,2) [circle,draw,fill,inner sep=0.5mm] {};
\draw[decoration={markings, mark=at position 0.6 with {\arrow{>}}},        postaction={decorate}] (-1.8,-4) .. controls +(-4,1) and +(-4,-4) .. (-1.8,4);
\draw[decoration={markings, mark=at position 0.8 with {\arrow{>}}},        postaction={decorate},] (-4.8,-1.5) .. controls +(1,-1) and +(-1,-1) .. (-2.85,-1.5);
\draw[decoration={markings, mark=at position 0.6 with {\arrow{>}}},        postaction={decorate}] (-3.5,2) -- (-5.5,4.5);
\draw[decoration={markings, mark=at position 0.5 with {\arrow{<}}},        postaction={decorate}] (n1) -- (-3,0);
\draw[thin,->] (-7.5,-2.05) .. controls +(0.5,-0.5) and +(-0.5,-1.25) .. (-3.3,-2.5);
\node at (5.9,0.4) {$\scriptstyle \omega_1 ^{-1} \omega_2^{-1}$};
\node at (-1.0,0.0) {$\scriptstyle g$};
\node at (-7.4,0.4) {$\scriptstyle \omega_2 \omega_1 g^{-1}$};
\node at (-7,6) {$\scriptstyle \omega_2 \omega_1 g^{-1} \omega_1^{-1} \omega_2^{-1}$};
\node at (-7.5,-1.4) {$\scriptstyle \omega_1 g^{-1} \omega_1^{-1}$};
\end{tikzpicture}
\]
which, by equation \eqref{diag:crossRemoval}, is equal to
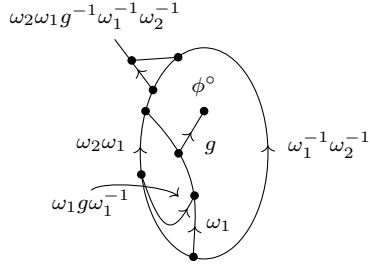
\[
\begin{tikzpicture}[scale=0.28,color=black, baseline]
\draw[decoration={markings, mark=at position 0.015 with {\arrow{>}}},decoration={markings, mark=at position 0.5 with {\arrow{<}}},        postaction={decorate}] (0,0) ellipse (3 and 5);
\node (n1) at (0,2) [circle,draw,fill,inner sep=0.35mm] {};
\node[above=0.01mm of n1] {$\scriptstyle \phi^{\circ}$};
\node at (-2.75,2) [circle,draw,fill,inner sep=0.35mm] {};
\node at (-2.4,3) [circle,draw,fill,inner sep=0.35mm] {};
\node at (-0.5,-4.9) [circle,draw,fill,inner sep=0.35mm] {};
\node at (-0.45,-2) [circle,draw,fill,inner sep=0.35mm] {};
\node at (-2.95,-1) [circle,draw,fill,inner sep=0.35mm] {};
\node at (-1.2,0) [circle,draw,fill,inner sep=0.35mm] {};
\node at (-3.4,4.4) [circle,draw,fill,inner sep=0.35mm] {};
\node at (-1.22,4.55) [circle,draw,fill,inner sep=0.35mm] {};
\draw[decoration={markings, mark=at position 0.4 with {\arrow{>}}},        postaction={decorate}] (-2.4,3) -- (-4.2,5.4);
\draw (-3.4,4.4) -- (-1.22,4.55);
\draw[decoration={markings, mark=at position 0.85 with {\arrow{<}}},        postaction={decorate}] (-2.75,2) .. controls +(3,-3) and +(0,2) .. (-0.5,-4.9);
\draw [decoration={markings, mark=at position 0.6 with {\arrow{<}}},        postaction={decorate}] (n1) -- (-1.2,0);
\draw[decoration={markings, mark=at position 0.85 with {\arrow{>}}},        postaction={decorate}] (-2.95,-1) .. controls +(1,-3) and +(-1,-2) .. (-0.45,-2);
\node at (5.9,0.4) {$\scriptstyle \omega_1^{-1} \omega_2^{-1}$};
\node at (-4.8,0.4) {$\scriptstyle \omega_2 \omega_1$};
\node at (0.7,-3.4) {$\scriptstyle \omega_1$};
\node at (0.3,0.3) {$\scriptstyle g$};
\node at (-5.3,-2.35) {$\scriptstyle \omega_1 g \omega_1^{-1}$};
\draw[thin,->] (-5.3,-1.75) .. controls +(0.5,0.5) and +(-0.5,0.5) .. (-1.2,-2.05);
\node at (-5.1,6.5) {$\scriptstyle \omega_2 \omega_1 g^{-1} \omega_1^{-1} \omega_2^{-1}$};
\end{tikzpicture}
\]
Repeatedly applying equation \eqref{diag:associativity} gives
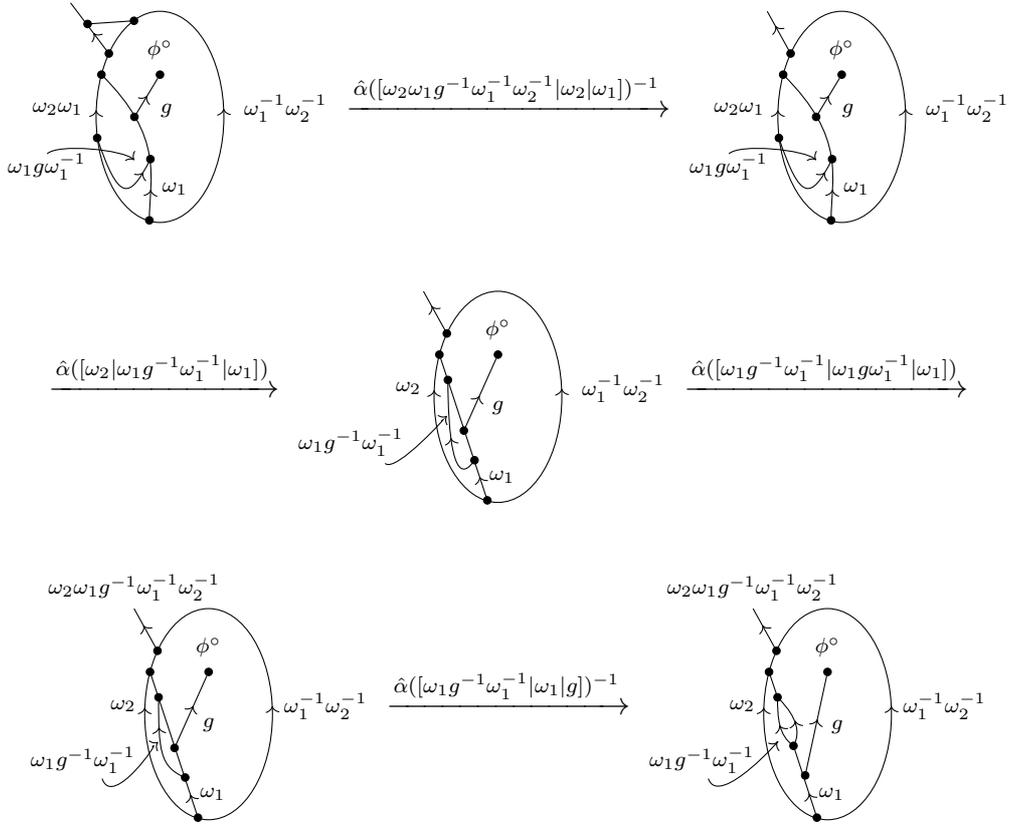
\[
\begin{tikzpicture}[scale=0.28,color=black, baseline]
\draw[decoration={markings, mark=at position 0.015 with {\arrow{>}}},decoration={markings, mark=at position 0.5 with {\arrow{<}}},        postaction={decorate}] (0,0) ellipse (3 and 5);
\node (n1) at (0,2) [circle,draw,fill,inner sep=0.35mm] {};
\node[above=0.01mm of n1] {$\scriptstyle \phi^{\circ}$};
\node at (-2.75,2) [circle,draw,fill,inner sep=0.35mm] {};
\node at (-2.4,3) [circle,draw,fill,inner sep=0.35mm] {};
\node at (-0.5,-4.9) [circle,draw,fill,inner sep=0.35mm] {};
\node at (-0.45,-2) [circle,draw,fill,inner sep=0.35mm] {};
\node at (-2.95,-1) [circle,draw,fill,inner sep=0.35mm] {};
\node at (-1.2,0) [circle,draw,fill,inner sep=0.35mm] {};
\node at (-3.4,4.4) [circle,draw,fill,inner sep=0.35mm] {};
\node at (-1.22,4.55) [circle,draw,fill,inner sep=0.35mm] {};
\draw (-3.4,4.4) -- (-1.22,4.55);
\draw[decoration={markings, mark=at position 0.4 with {\arrow{>}}},        postaction={decorate}] (-2.4,3) -- (-4.2,5.4);
\draw[decoration={markings, mark=at position 0.85 with {\arrow{<}}},        postaction={decorate}] (-2.75,2) .. controls +(3,-3) and +(0,2) .. (-0.5,-4.9);
\draw[decoration={markings, mark=at position 0.6 with {\arrow{<}}},        postaction={decorate}] (n1) -- (-1.2,0);
\draw[decoration={markings, mark=at position 0.85 with {\arrow{>}}},        postaction={decorate}] (-2.95,-1) .. controls +(1,-3) and +(-1,-2) .. (-0.45,-2);
\node at (5.9,0.4) {$\scriptstyle \omega_1^{-1} \omega_2^{-1}$};
\node at (-4.8,0.4) {$\scriptstyle \omega_2 \omega_1$};
\node at (0.7,-3.4) {$\scriptstyle \omega_1$};
\node at (0.3,0.3) {$\scriptstyle g$};
\node at (-5.3,-2.35) {$\scriptstyle \omega_1 g \omega_1^{-1}$};
\draw[thin,->] (-5.3,-1.75) .. controls +(0.5,0.5) and +(-0.5,0.5) .. (-1.2,-2.05);
\end{tikzpicture}
\xrightarrow[]{\hat{\alpha}([\omega_2 \omega_1 g^{-1} \omega_1^{-1} \omega_2^{-1} \vert \omega_2 \vert \omega_1])^{-1}}
\begin{tikzpicture}[scale=0.28,color=black, baseline]
\draw[decoration={markings, mark=at position 0.015 with {\arrow{>}}},decoration={markings, mark=at position 0.5 with {\arrow{<}}},        postaction={decorate}] (0,0) ellipse (3 and 5);
\node (n1) at (0,2) [circle,draw,fill,inner sep=0.35mm] {};
\node[above=0.01mm of n1] {$\scriptstyle \phi^{\circ}$};
\node at (-2.75,2) [circle,draw,fill,inner sep=0.35mm] {};
\node at (-2.4,3) [circle,draw,fill,inner sep=0.35mm] {};
\node at (-0.5,-4.9) [circle,draw,fill,inner sep=0.35mm] {};
\node at (-0.45,-2) [circle,draw,fill,inner sep=0.35mm] {};
\node at (-2.95,-1) [circle,draw,fill,inner sep=0.35mm] {};
\node at (-1.2,0) [circle,draw,fill,inner sep=0.35mm] {};
\draw[decoration={markings, mark=at position 0.65 with {\arrow{>}}},        postaction={decorate}] (-2.4,3) -- (-3.5,5);
\draw[decoration={markings, mark=at position 0.85 with {\arrow{<}}},        postaction={decorate}] (-2.75,2) .. controls +(3,-3) and +(0,2) .. (-0.5,-4.9);
\draw[decoration={markings, mark=at position 0.6 with {\arrow{<}}},        postaction={decorate}] (n1) -- (-1.2,0);
\draw[decoration={markings, mark=at position 0.85 with {\arrow{>}}},        postaction={decorate}] (-2.95,-1) .. controls +(1,-3) and +(-1,-2) .. (-0.45,-2);
\node at (5.9,0.4) {$\scriptstyle \omega_1^{-1} \omega_2^{-1}$};
\node at (-4.8,0.4) {$\scriptstyle \omega_2 \omega_1 $};
\node at (0.7,-3.4) {$\scriptstyle \omega_1$};
\node at (0.3,0.3) {$\scriptstyle g$};
\node at (-5.3,-2.35) {$\scriptstyle \omega_1 g \omega_1^{-1}$};
\draw[thin,->] (-5.3,-1.75) .. controls +(0.5,0.5) and +(-0.5,0.5) .. (-1.2,-2.05);
\end{tikzpicture}
\]
\[
\xrightarrow[]{\hat{\alpha}([\omega_2 \vert \omega_1 g^{-1} \omega_1^{-1} \vert \omega_1])}
\begin{tikzpicture}[scale=0.28,color=black, baseline]
\draw[decoration={markings, mark=at position 0.015 with {\arrow{>}}},decoration={markings, mark=at position 0.5 with {\arrow{<}}},        postaction={decorate}] (0,0) ellipse (3 and 5);
\node (n1) at (0,2) [circle,draw,fill,inner sep=0.35mm] {};
\node[above=0.01mm of n1] {$\scriptstyle \phi^{\circ}$};
\node at (-2.75,2) [circle,draw,fill,inner sep=0.35mm] {};
\node at (-2.4,3) [circle,draw,fill,inner sep=0.35mm] {};
\node at (-0.5,-4.9) [circle,draw,fill,inner sep=0.35mm] {};
\node at (-2.35,0.8) [circle,draw,fill,inner sep=0.35mm] {};
\node at (-1.1,-3) [circle,draw,fill,inner sep=0.35mm] {};
\node at (-1.6,-1.6) [circle,draw,fill,inner sep=0.35mm] {};
\draw[decoration={markings, mark=at position 0.65 with {\arrow{>}}},        postaction={decorate}] (-2.4,3) -- (-3.5,5);
\draw[decoration={markings, mark=at position 0.88 with {\arrow{<}}},        postaction={decorate}] (-2.75,2) -- (-0.5,-4.9);
\draw[decoration={markings, mark=at position 0.6 with {\arrow{<}}},        postaction={decorate}] (n1) -- (-1.6,-1.6);
\draw[decoration={markings, mark=at position 0.6 with {\arrow{<}}},        postaction={decorate}] (-2.35,0.8) .. controls +(0,-2.5) and +(-1,-1.5) .. (-1.1,-3);
\node at (5.9,0.4) {$\scriptstyle \omega_1^{-1} \omega_2^{-1}$};
\node at (-4.2,0.4) {$\scriptstyle \omega_2$};
\node at (-0.0,-0.5) {$\scriptstyle g$};
\node at (0.2,-3.8) {$\scriptstyle \omega_1$};
\draw[thin,decoration={markings, mark=at position 1.0 with {\arrow{<}}},        postaction={decorate}] (-5.3,-3.15) .. controls +(0.5,-0.5) and +(0.5,0.5) .. (-2.65,-1.15);
\node at (-6.9,-2.15) {$\scriptstyle \omega_1 g^{-1} \omega_1^{-1}$};
\end{tikzpicture}
\xrightarrow[]{\hat{\alpha}([\omega_1 g^{-1} \omega_1^{-1} \vert \omega_1 g \omega_1^{-1} \vert \omega_1])}
\]
\[
\begin{tikzpicture}[scale=0.28,color=black, baseline]
\draw[decoration={markings, mark=at position 0.015 with {\arrow{>}}},decoration={markings, mark=at position 0.5 with {\arrow{<}}},        postaction={decorate}] (0,0) ellipse (3 and 5);
\node (n1) at (0,2) [circle,draw,fill,inner sep=0.35mm] {};
\node[above=0.01mm of n1] {$\scriptstyle \phi^{\circ}$};
\node at (-2.75,2) [circle,draw,fill,inner sep=0.35mm] {};
\node at (-2.4,3) [circle,draw,fill,inner sep=0.35mm] {};
\node at (-0.5,-4.9) [circle,draw,fill,inner sep=0.35mm] {};
\node at (-2.35,0.8)  [circle,draw,fill,inner sep=0.35mm] {};
\node at (-1.1,-3) [circle,draw,fill,inner sep=0.35mm] {};
\node at (-1.6,-1.6) [circle,draw,fill,inner sep=0.35mm] {};
\draw[decoration={markings, mark=at position 0.6 with {\arrow{>}}},        postaction={decorate}] (-2.4,3) -- (-3.5,5);
\draw[decoration={markings, mark=at position 0.885 with {\arrow{<}}},        postaction={decorate}] (-2.75,2) -- (-0.5,-4.9);
\draw[decoration={markings, mark=at position 0.5 with {\arrow{<}}},        postaction={decorate}] (n1) -- (-1.6,-1.6);
\draw[decoration={markings, mark=at position 0.4 with {\arrow{<}}},        postaction={decorate}] (-2.35,0.8) .. controls +(0,-2.5) and +(-1.2,0.5) .. (-1.1,-3);
\node at (5.5,0.4) {$\scriptstyle \omega_1^{-1} \omega_2^{-1}$};
\node at (-4.0,0.4) {$\scriptstyle \omega_2$};
\node at (-0.0,-0.5) {$\scriptstyle g$};
\node at (0.2,-3.8) {$\scriptstyle \omega_1$};
\draw[thin,decoration={markings, mark=at position 1.0 with {\arrow{<}}},        postaction={decorate}] (-5.0,-3.05) .. controls +(0.5,-1.5) and +(0.5,0.5) .. (-2.65,-1.55);
\node at (-5.9,-2.15) {$\scriptstyle \omega_1 g^{-1} \omega_1^{-1}$};
\node at (-3.5,6) {$\scriptstyle \omega_2 \omega_1 g^{-1} \omega_1^{-1} \omega_2^{-1}$};
\end{tikzpicture}
\xrightarrow[]{\hat{\alpha}([\omega_1 g^{-1} \omega^{-1}_1 \vert \omega_1 \vert g])^{-1}}
\begin{tikzpicture}[scale=0.28,color=black, baseline]
\draw[decoration={markings, mark=at position 0.015 with {\arrow{>}}},decoration={markings, mark=at position 0.5 with {\arrow{<}}},        postaction={decorate}] (0,0) ellipse (3 and 5);
\node (n1) at (0,2) [circle,draw,fill,inner sep=0.35mm] {};
\node[above=0.01mm of n1] {$\scriptstyle \phi^{\circ}$};
\node at (-2.75,2) [circle,draw,fill,inner sep=0.35mm] {};
\node at (-2.4,3) [circle,draw,fill,inner sep=0.35mm] {};
\node at (-0.5,-4.9) [circle,draw,fill,inner sep=0.35mm] {};
\node at (-2.35,0.8)  [circle,draw,fill,inner sep=0.35mm] {};
\node at (-1.6,-1.5) [circle,draw,fill,inner sep=0.35mm] {};
\node at (-1.05,-2.9) [circle,draw,fill,inner sep=0.35mm] {};
\draw[decoration={markings, mark=at position 0.6 with {\arrow{>}}},        postaction={decorate}] (-2.4,3) -- (-3.5,5);
\draw (-2.75,2) -- (-2.35,0.8);
\draw[decoration={markings, mark=at position 0.75 with {\arrow{<}}},        postaction={decorate}] (-1.6,-1.5) -- (-0.5,-4.9);
\draw[decoration={markings, mark=at position 0.5 with {\arrow{<}}},        postaction={decorate}] (n1) -- (-1.05,-2.9);
\draw[decoration={markings, mark=at position 0.65 with {\arrow{<}}},        postaction={decorate}] (-2.35,0.8) .. controls +(0.5,-0.5) and +(0.5,0.9) .. (-1.6,-1.5);
\draw[decoration={markings, mark=at position 0.65 with {\arrow{<}}},        postaction={decorate}] (-2.35,0.8)  .. controls +(0,-1) and +(-0.7,0.5) .. (-1.6,-1.5);
\node at (5.5,0.4) {$\scriptstyle \omega_1^{-1} \omega_2^{-1}$};
\node at (-4.0,0.4) {$\scriptstyle \omega_2$};
\node at (0.45,-0.6) {$\scriptstyle g$};
\node at (0.2,-3.8) {$\scriptstyle \omega_1$};
\node at (-5.9,-2.15) {$\scriptstyle \omega_1 g^{-1} \omega_1^{-1}$};
\draw[thin,decoration={markings, mark=at position 1.0 with {\arrow{>}}},        postaction={decorate}] (-5.6,-3.05) .. controls +(0.5,-1.5) and +(0.5,0.5) .. (-2.4,-1.35);
\node at (-3.5,6) {$\scriptstyle \omega_2 \omega_1 g^{-1} \omega_1^{-1} \omega_2^{-1}$};
\end{tikzpicture}
\]
More precisely, the first two arrows arise from the inverted form of equation \eqref{diag:associativity} while the last two arrows used equation \eqref{diag:associativity} but applied in the category $V^{\op}$. This explains why we multiply by $\hat{\alpha}([\omega_1 g^{-1} \omega_1^{-1} \vert \omega_1 g \omega_1^{-1} \vert \omega_1])$ and $\hat{\alpha}([\omega_1 g^{-1} \omega^{-1}_1 \vert \omega_1 \vert g])^{-1}$, rather than their inverses. Continuing, by first removing the loop and then adding a different loop (see equation \eqref{diag:loopRemoval}), the previous diagram becomes
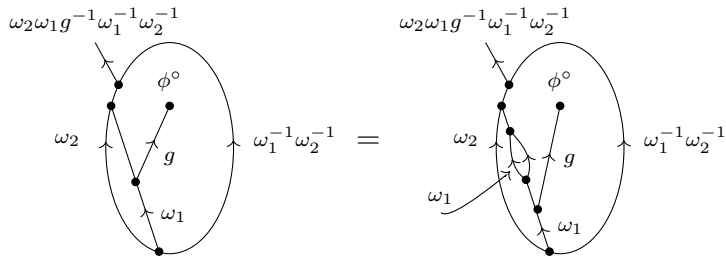
\[
\begin{tikzpicture}[scale=0.28,color=black, baseline]
\draw[decoration={markings, mark=at position 0.015 with {\arrow{>}}},decoration={markings, mark=at position 0.5 with {\arrow{<}}},        postaction={decorate}] (0,0) ellipse (3 and 5);
\node (n1) at (0,2) [circle,draw,fill,inner sep=0.35mm] {};
\node[above=0.01mm of n1] {$\scriptstyle \phi^{\circ}$};
\node at (-2.75,2) [circle,draw,fill,inner sep=0.35mm] {};
\node at (-2.4,3) [circle,draw,fill,inner sep=0.35mm] {};
\node at (-0.5,-4.9) [circle,draw,fill,inner sep=0.35mm] {};
\node at (-1.6,-1.6) [circle,draw,fill,inner sep=0.35mm] {};
\draw[decoration={markings, mark=at position 0.6 with {\arrow{>}}},        postaction={decorate}] (-2.4,3) -- (-3.5,5);
\draw[decoration={markings, mark=at position 0.75 with {\arrow{<}}},        postaction={decorate}] (-2.75,2) -- (-0.5,-4.9);
\draw[decoration={markings, mark=at position 0.5 with {\arrow{<}}},        postaction={decorate}] (n1) -- (-1.6,-1.6);
\node at (5.9,0.4) {$\scriptstyle \omega_1^{-1} \omega_2^{-1}$};
\node at (-4.8,0.4) {$\scriptstyle \omega_2$};
\node at (-0.0,-0.5) {$\scriptstyle g$};
\node at (0.2,-3.2) {$\scriptstyle \omega_1$};
\node at (-3.5,6) {$\scriptstyle \omega_2 \omega_1 g^{-1} \omega_1^{-1} \omega_2^{-1}$};
\end{tikzpicture}
=
\begin{tikzpicture}[scale=0.28,color=black, baseline]
\draw[decoration={markings, mark=at position 0.015 with {\arrow{>}}},decoration={markings, mark=at position 0.5 with {\arrow{<}}},        postaction={decorate}] (0,0) ellipse (3 and 5);
\node (n1) at (0,2) [circle,draw,fill,inner sep=0.35mm] {};
\node[above=0.01mm of n1] {$\scriptstyle \phi^{\circ}$};
\node at (-2.75,2) [circle,draw,fill,inner sep=0.35mm] {};
\node at (-2.4,3) [circle,draw,fill,inner sep=0.35mm] {};
\node at (-0.5,-4.9) [circle,draw,fill,inner sep=0.35mm] {};
\node at (-2.35,0.8)  [circle,draw,fill,inner sep=0.35mm] {};
\node at (-1.6,-1.5) [circle,draw,fill,inner sep=0.35mm] {};
\node at (-1.05,-2.9) [circle,draw,fill,inner sep=0.35mm] {};
\draw[decoration={markings, mark=at position 0.6 with {\arrow{>}}},        postaction={decorate}] (-2.4,3) -- (-3.5,5);
\draw (-2.75,2) -- (-2.35,0.8);
\draw[decoration={markings, mark=at position 0.75 with {\arrow{<}}},        postaction={decorate}] (-1.6,-1.5) -- (-0.5,-4.9);
\draw[decoration={markings, mark=at position 0.5 with {\arrow{<}}},        postaction={decorate}] (n1) -- (-1.05,-2.9);
\draw[decoration={markings, mark=at position 0.65 with {\arrow{<}}},        postaction={decorate}] (-2.35,0.8) .. controls +(0.5,-0.5) and +(0.5,0.9) .. (-1.6,-1.5);
\draw[decoration={markings, mark=at position 0.65 with {\arrow{<}}},        postaction={decorate}] (-2.35,0.8)  .. controls +(0,-1) and +(-0.7,0.5) .. (-1.6,-1.5);
\node at (5.9,0.4) {$\scriptstyle \omega_1 ^{-1} \omega_2^{-1}$};
\node at (-4.4,0.4) {$\scriptstyle \omega_2$};
\node at (0.45,-0.6) {$\scriptstyle g$};
\node at (0.5,-3.8) {$\scriptstyle \omega_1$};
\node at (-5.6,-2.35) {$\scriptstyle \omega_1$};
\draw[thin,decoration={markings, mark=at position 1.0 with {\arrow{<}}},        postaction={decorate}] (-5.6,-3.05) .. controls +(0.5,-0.5) and +(0.5,0.5) .. (-2.6,-1.35);
\node at (-3.5,6) {$\scriptstyle \omega_2 \omega_1 g^{-1} \omega_1^{-1} \omega_2^{-1}$};
\end{tikzpicture}
\]
Repeatedly applying equation \eqref{diag:associativity} then gives
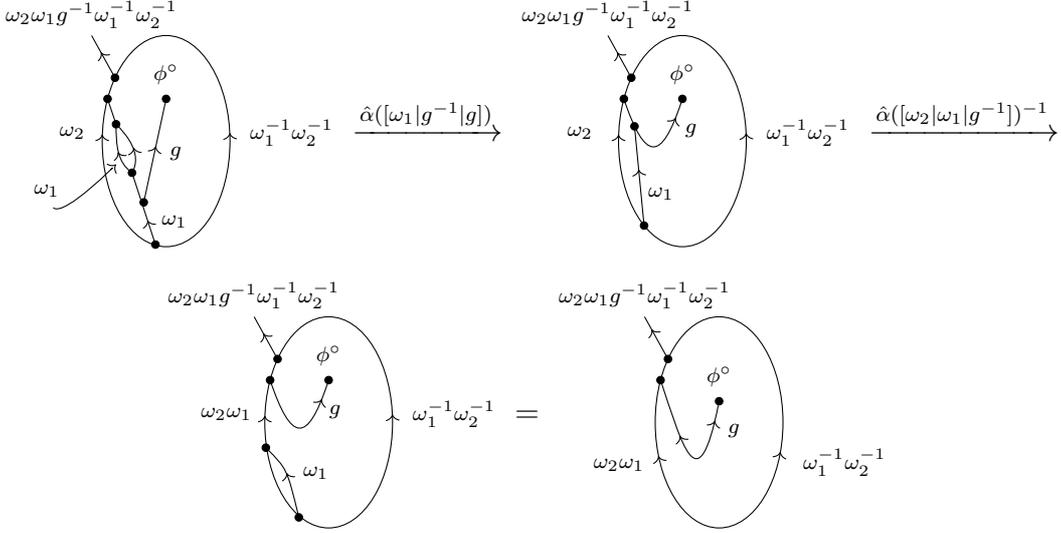
\[
\begin{tikzpicture}[scale=0.28,color=black, baseline]
\draw[decoration={markings, mark=at position 0.015 with {\arrow{>}}},decoration={markings, mark=at position 0.5 with {\arrow{<}}},        postaction={decorate}] (0,0) ellipse (3 and 5);
\node (n1) at (0,2) [circle,draw,fill,inner sep=0.35mm] {};
\node[above=0.01mm of n1] {$\scriptstyle \phi^{\circ}$};
\node at (-2.75,2) [circle,draw,fill,inner sep=0.35mm] {};
\node at (-2.4,3) [circle,draw,fill,inner sep=0.35mm] {};
\node at (-0.5,-4.9) [circle,draw,fill,inner sep=0.35mm] {};
\node at (-2.35,0.8)  [circle,draw,fill,inner sep=0.35mm] {};
\node at (-1.6,-1.5) [circle,draw,fill,inner sep=0.35mm] {};
\node at (-1.05,-2.9) [circle,draw,fill,inner sep=0.35mm] {};
\draw[decoration={markings, mark=at position 0.6 with {\arrow{>}}},        postaction={decorate}] (-2.4,3) -- (-3.5,5);
\draw (-2.75,2) -- (-2.35,0.8);
\draw[decoration={markings, mark=at position 0.75 with {\arrow{<}}},        postaction={decorate}] (-1.6,-1.5) -- (-0.5,-4.9);
\draw[decoration={markings, mark=at position 0.5 with {\arrow{<}}},        postaction={decorate}] (n1) -- (-1.05,-2.9);
\draw[decoration={markings, mark=at position 0.65 with {\arrow{<}}},        postaction={decorate}] (-2.35,0.8) .. controls +(0.5,-0.5) and +(0.5,0.9) .. (-1.6,-1.5);
\draw[decoration={markings, mark=at position 0.65 with {\arrow{<}}},        postaction={decorate}] (-2.35,0.8)  .. controls +(0,-1) and +(-0.7,0.5) .. (-1.6,-1.5);
\node at (5.9,0.4) {$\scriptstyle \omega_1 ^{-1} \omega_2^{-1}$};
\node at (-4.4,0.4) {$\scriptstyle \omega_2$};
\node at (0.45,-0.6) {$\scriptstyle g$};
\node at (0.4,-3.8) {$\scriptstyle \omega_1$};
\node at (-5.6,-2.35) {$\scriptstyle \omega_1$};
\draw[thin,decoration={markings, mark=at position 1.0 with {\arrow{<}}},        postaction={decorate}] (-5.3,-3.15) .. controls +(0.5,-0.5) and +(0.5,0.5) .. (-2.6,-1.35);
\node at (-3.5,6) {$\scriptstyle \omega_2 \omega_1 g^{-1} \omega_1^{-1} \omega_2^{-1}$};
\end{tikzpicture}
\xrightarrow[]{\hat{\alpha}([\omega_1 \vert g^{-1} \vert g])}
\begin{tikzpicture}[scale=0.28,color=black, baseline]
\node (n1) at (0,2) [circle,draw,fill,inner sep=0.35mm] {};
\node[above=0.01mm of n1] {$\scriptstyle \phi^{\circ}$};
\node at (-2.75,2) [circle,draw,fill,inner sep=0.35mm] {};
\node at (-2.4,3) [circle,draw,fill,inner sep=0.35mm] {};
\node at (-2.25,0.7) [circle,draw,fill,inner sep=0.35mm] {};
\node at (-1.8,-4) [circle,draw,fill,inner sep=0.35mm] {};
\draw[decoration={markings, mark=at position 0.015 with {\arrow{>}}},decoration={markings, mark=at position 0.5 with {\arrow{<}}},        postaction={decorate}] (0,0) ellipse (3 and 5);
\draw[decoration={markings, mark=at position 0.2 with {\arrow{<}}},        postaction={decorate}]  (n1) .. controls +(-1,-3) and +(1,-3) .. (-2.75,2);
\draw[decoration={markings, mark=at position 0.6 with {\arrow{>}}},        postaction={decorate}] (-2.4,3) -- (-3.5,5);
\draw[decoration={markings, mark=at position 0.5 with {\arrow{<}}},        postaction={decorate}] (-2.25,0.7) -- (-1.8,-4);
\node at (5.9,0.4) {$\scriptstyle \omega_1 ^{-1} \omega_2^{-1}$};
\node at (-1.0,-2.5) {$\scriptstyle \omega_1$};
\node at (0.4,0.5) {$\scriptstyle g$};
\node at (-4.8,0.4) {$\scriptstyle \omega_2$};
\node at (-3.5,6) {$\scriptstyle \omega_2 \omega_1 g^{-1} \omega_1^{-1} \omega_2^{-1}$};
\end{tikzpicture}
\xrightarrow[]{\hat{\alpha}([\omega_2 \vert \omega_1 \vert g^{-1}])^{-1}}
\]
\[
\begin{tikzpicture}[scale=0.28,color=black, baseline]
\node (n1) at (0,2) [circle,draw,fill,inner sep=0.35mm] {};
\node[above=0.01mm of n1] {$\scriptstyle \phi^{\circ}$};
\node at (-2.75,2) [circle,draw,fill,inner sep=0.35mm] {};
\node at (-2.4,3) [circle,draw,fill,inner sep=0.35mm] {};
\node at (-2.95,-1.2) [circle,draw,fill,inner sep=0.35mm] {};
\node at (-1.4,-4.5) [circle,draw,fill,inner sep=0.35mm] {};
\draw[decoration={markings, mark=at position 0.015 with {\arrow{>}}},decoration={markings, mark=at position 0.5 with {\arrow{<}}},        postaction={decorate}] (0,0) ellipse (3 and 5);
\draw[decoration={markings, mark=at position 0.2 with {\arrow{<}}},        postaction={decorate}] (n1) .. controls +(-1,-3) and +(1,-3) .. (-2.75,2);
\draw[decoration={markings, mark=at position 0.6 with {\arrow{>}}},        postaction={decorate}] (-2.4,3) -- (-3.5,5);
\draw[decoration={markings, mark=at position 0.5 with {\arrow{<}}},        postaction={decorate}] (-2.95,-1.2) .. controls +(1,-1) .. (-1.4,-4.5);
\node at (5.9,0.4) {$\scriptstyle \omega_1 ^{-1} \omega_2^{-1}$};
\node at (-4.8,0.4) {$\scriptstyle \omega_2 \omega_1$};
\node at (0.3,0.5) {$\scriptstyle g$};
\node at (-0.6,-2.4) {$\scriptstyle \omega_1$};
\node at (-3.5,6) {$\scriptstyle \omega_2 \omega_1 g^{-1} \omega_1^{-1} \omega_2^{-1}$};
\end{tikzpicture}
=
\begin{tikzpicture}[scale=0.28,color=black, baseline]
\draw [decoration={markings, mark=at position 0.575 with {\arrow{<}}}, decoration={markings, mark=at position 0.945 with {\arrow{>}}},        postaction={decorate}] (0,0) ellipse (3 and 5);
\node (n1) at (0,1.0) [circle,draw,fill,inner sep=0.35mm] {};
\node[above=0.01mm of n1] {$\scriptstyle \phi^{\circ}$};
\node at (-2.75,2) [circle,draw,fill,inner sep=0.35mm] {};
\node at (-2.4,3) [circle,draw,fill,inner sep=0.35mm] {};
\draw[decoration={markings, mark=at position 0.6 with {\arrow{>}}},decoration={markings, mark=at position 0.15 with {\arrow{<}}},        postaction={decorate}] (n1) .. controls +(-1,-5) and +(1,-3) .. (-2.75,2);
\draw[decoration={markings, mark=at position 0.7 with {\arrow{>}}},        postaction={decorate}] (-2.4,3) -- (-3.5,5);
\node at (-4.7,-2) {$\scriptstyle \omega_2 \omega_1$};
\node at (5.9,-1.9) {$\scriptstyle \omega_1^{-1} \omega_2^{-1}$};
\node at (0.7,-0.4) {$\scriptstyle g$};
\node at (-3.5,6) {$\scriptstyle \omega_2 \omega_1 g^{-1} \omega_1^{-1} \omega_2^{-1}$};
\end{tikzpicture}
\]
The first arrow used equation \eqref{diag:associativity} in $V^{\op}$. By definition, the final diagram computes $\beta_{g, \omega_2 \omega_1}(\phi)$. The scalar introduced in the entire computation is thus
\[
\frac{\hat{\alpha}([\omega_1 \vert g^{-1} \vert g]) \hat{\alpha}([\omega_1 g^{-1} \omega_1^{-1} \vert \omega_1 g \omega_1^{-1} \vert \omega_1])}{\hat{\alpha}([\omega_1 g^{-1} \omega^{-1}_1 \vert \omega_1 \vert g])} \times 
\frac{\hat{\alpha}([\omega_2 \vert \omega_1 g^{-1} \omega_1^{-1} \vert \omega_1])}{\hat{\alpha}([\omega_2 \omega_1 g^{-1} \omega_1^{-1} \omega_2^{-1} \vert \omega_2 \vert \omega_1]) \hat{\alpha}([\omega_2 \vert \omega_1 \vert g^{-1}])},
\]
which we recognize as $\uptau_{\pi}^{\refl}(\hat{\alpha})([\omega_2 \vert \omega_1]g)^{-1}$.

A similar calculation can be performed when $\pi(\omega_1)=-1$. The key difference is that, since both $\omega_1$ and $\omega_2$ are of degree $-1$, at the final stage of the calculation we produce a scalar multiple of the string diagram
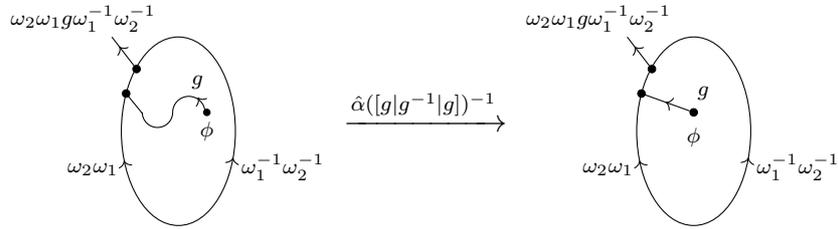
\[
\begin{tikzpicture}[scale=0.25,color=black, baseline]
\draw [decoration={markings, mark=at position 0.575 with {\arrow{<}}}, decoration={markings, mark=at position 0.945 with {\arrow{>}}},        postaction={decorate}] (0,0) ellipse (3 and 5);
\node (n1) at (-0.17,1.9) [circle] {};
\node at (1.5,0.0) {$\scriptstyle \phi$};
\node at (-2.2,3.3) [circle,draw,fill,inner sep=0.35mm] {};
\node at (-2.75,2) [circle,draw,fill,inner sep=0.35mm] {};
\node at (-4.5,-2) {$\scriptstyle \omega_2 \omega_1$};
\node at (5.5,-1.9) {$\scriptstyle \omega_1^{-1} \omega_2^{-1}$};
\node at (1.0,2.6) {$\scriptstyle g$};
\node at (-5,6) {$\scriptstyle \omega_2 \omega_1 g \omega_1^{-1} \omega_2^{-1}$};
\draw[black, decoration={markings, mark=at position 0.795 with {\arrow{<}}},        postaction={decorate}](0.55,1.0) [partial ellipse=180:0:0.85 and 0.85];
\draw[black](-1.1,1.0) [partial ellipse=180:360:0.8 and 0.8];
\draw[] (-1.9,0.95) -- (-2.75,2);
\node at (1.5,1.0) [circle,draw,fill,inner sep=0.3mm] {};
\draw[decoration={markings, mark=at position 0.7 with {\arrow{>}}},        postaction={decorate}] (-2.2,3.3) -- (-3.5,5);
\end{tikzpicture}
\xrightarrow[]{\hat{\alpha}([g\vert g^{-1} \vert g])^{-1}}
\begin{tikzpicture}[scale=0.25,color=black, baseline]
\draw [decoration={markings, mark=at position 0.575 with {\arrow{<}}}, decoration={markings, mark=at position 0.945 with {\arrow{>}}},        postaction={decorate}] (0,0) ellipse (3 and 5);
\node (n1) at (0,1.0) [circle,draw,fill,inner sep=0.35mm] {};
\node[below=0.01mm of n1] {$\scriptstyle \phi$};
\node at (-2.75,2) [circle,draw,fill,inner sep=0.35mm] {};
\node at (-2.2,3.3) [circle,draw,fill,inner sep=0.35mm] {};
\draw[decoration={markings, mark=at position 0.7 with {\arrow{>}}},        postaction={decorate}] (-2.2,3.3) -- (-3.5,5);
\draw[decoration={markings, mark=at position 0.5 with {\arrow{>}}},        postaction={decorate}] (n1) -- (-2.75,2);
\node at (-4.5,-2) {$\scriptstyle \omega_2 \omega_1$};
\node at (5.5,-1.9) {$\scriptstyle \omega_1^{-1} \omega_2^{-1}$};
\node at (0.5,2.0) {$\scriptstyle g$};
\node at (-5,6) {$\scriptstyle \omega_2 \omega_1 g \omega_1^{-1} \omega_2^{-1}$};
\end{tikzpicture}
\]
The last step used equation \eqref{diag:snakeRemoval}. This gives the additional factor of $\hat{\alpha}([g \vert g^{-1} \vert g])$ appearing in $\uptau_{\pi}^{\refl}(\hat{\alpha})([\omega_2 \vert \omega_1] g)$ when both $\omega_1$ and $\omega_2$ are of degree $-1$. This completes the proof.
\end{proof}

Let $\alpha \in Z^3(B \mathsf{G}, k^{\times})$. In \cite[\S 3.1]{willerton2008} Willerton showed that the $\alpha$-twisted Drinfeld double of $\mathsf{G}$, as introduced in \cite{dijkgraaf1992}, is isomorphic (as an algebra) to the $\uptau(\alpha)$-twisted groupoid algebra of $\Lambda B \mathsf{G}$:
\[
D^{\alpha}(\mathsf{G}) \simeq k^{\uptau(\alpha)}[\Lambda B \mathsf{G}].
\]
Motivated by this, for a twisted $3$-cocycle $\hat{\alpha} \in Z^3(B \hat{\mathsf{G}}, k^{\times}_{\pi})$ the $\hat{\alpha}$-twisted thickened Drinfeld double of $\mathsf{G}$ was defined in \cite[\S 4.2]{mbyoung2018a} to be the $\uptau_{\pi}^{\refl}(\hat{\alpha})$-twisted groupoid algebra of $\Lambda_{\pi}^{\refl} B \hat{\mathsf{G}}$:
\[
D^{\hat{\alpha}}(\mathsf{G}) = k^{\uptau_{\pi}^{\refl}(\hat{\alpha})}[ \Lambda_{\pi}^{\refl} B \hat{\mathsf{G}}].
\]
The inclusion $\mathsf{G} \hookrightarrow \hat{\mathsf{G}}$ defines a faithful functor $\Lambda B \mathsf{G} \rightarrow \Lambda_{\pi}^{\refl} B \hat{\mathsf{G}}$ under which $\uptau_{\pi}^{\refl}(\hat{\alpha})$ restricts to $\uptau(\alpha)$. It follows that there is a $k$-algebra embedding $D^{\alpha}(\mathsf{G}) \hookrightarrow D^{\hat{\alpha}}(\mathsf{G})$, hence the terminology.

\begin{Cor}
\label{cor:thickDDAction}
The Real categorical character of an $\hat{\alpha}$-twisted Real $2$-representation of $\mathsf{G}$ is a module over the $\hat{\alpha}$-twisted thickened Drinfeld double of $\mathsf{G}$.
\end{Cor}

\begin{proof}
The category of vector bundles over ${^{\uptau_{\pi}^{\refl}(\hat{\alpha})}}\Lambda_{\pi}^{\refl}B \hat{\mathsf{G}}$ is equivalent to the category of $\uptau_{\pi}^{\refl}(\hat{\alpha})$-twisted $k[\Lambda_{\pi}^{\refl}B \hat{\mathsf{G}}]$-modules, which is in turn equivalent to the category of $k^{\uptau_{\pi}^{\refl}(\hat{\alpha})}[\Lambda_{\pi}^{\refl} B \hat{\mathsf{G}}]$-modules. The statement now follows from Theorem \ref{thm:catCharLoopGrpd}.
\end{proof}

The next result describes the equivariance properties of Real $2$-characters.

\begin{Thm}
\label{thm:conjInv2Char}
The Real $2$-character of an $\hat{\alpha}$-twisted Real $2$-representation $\rho$ of $\mathsf{G}$ is a flat section of the line bundle $\uptau \uptau_{\pi}^{\refl}(\hat{\alpha})_k \rightarrow \Lambda \Lambda_{\pi}^{\refl} B \hat{\mathsf{G}}$. Equivalently, the equality
\[
\chi_{\rho}(\sigma g^{\pi(\sigma)} \sigma^{-1}, \sigma \omega \sigma^{-1}) = \frac{\uptau_{\pi}^{\refl}(\hat{\alpha})([\sigma \omega \sigma^{-1} \vert \sigma] g)}{\uptau_{\pi}^{\refl}(\hat{\alpha})([\sigma \vert \omega] g)} \cdot \chi_{\rho}(g, \omega)
\]
holds for all $(g,\omega) \in \hat{\mathsf{G}}^{(2)}$ and $\sigma \in \hat{\mathsf{G}}$.
\end{Thm}

\begin{proof}
By Theorem \ref{thm:catCharLoopGrpd}, the Real categorical character $\Tr(\rho)$ is a $\uptau_{\pi}^{\refl}(\hat{\alpha})$-twisted vector bundle over $\Lambda_{\pi}^{\refl} B \hat{\mathsf{G}}$. By the results of \cite[\S 2.3.3]{willerton2008}, the holonomy of $\Tr(\rho)$ is a flat section of the transgressed line bundle $\uptau \uptau_{\pi}^{\refl}(\hat{\alpha})_k \rightarrow \Lambda \Lambda_{\pi}^{\refl} B \hat{\mathsf{G}}$. On the other hand, the holonomy of $\Tr(\rho)$ is by construction the Real $2$-character of $\rho$. Combining these results gives the desired statement.

The explicit description of the $\hat{\mathsf{G}}$-equivariance of $\chi_{\rho}$ follows from Willerton's formula for the loop transgression of an untwisted $2$-cocycle \cite[\S 1.3.3]{willerton2008}.
\end{proof}

Flat sections of the line bundle $\uptau \uptau_{\pi}^{\refl}(\hat{\alpha})_{\mathbb{C}} \rightarrow \Lambda \Lambda_{\pi}^{\refl} B \hat{\mathsf{G}}$ were first studied in \cite[\S 4.4]{mbyoung2018a}, where they were shown to describe the complexified representation ring of $D^{\hat{\alpha}}(\mathsf{G})$. Theorem \ref{thm:conjInv2Char} gives a second interpretation of such sections, namely, as Real $2$-class functions for $\hat{\alpha}$-twisted Real $2$-representations of $\mathsf{G}$. Corollary \ref{cor:thickDDAction} explains the relationship between the two seemingly unrelated interpretations.

Upon substitution of the explicit expression for $\uptau_{\pi}^{\refl}(\hat{\alpha})$ into the equality appearing in the statement of Theorem \ref{thm:conjInv2Char}, the coefficient of $\chi_{\rho}(g, \omega)$ reproduces Sharpe's $C$-field discrete torsion phase factors for the three dimensional torus and the Klein bottle times $S^1$ in $M$-theory with orientifolds \cite[\S 6.2]{sharpe2011}. This strongly suggests a role for Real $2$-representation theory in $M$-theory. More precisely, consider an orientifold compatible $C$-field which is pulled back from the $C$-field $\hat{\alpha}$ on $B \hat{\mathsf{G}}$. We expect the $2$-Hilbert spaces resulting from the higher geometric quantization of membranes in this background to be a Real representation of $\mathcal{G}(\mathsf{G}, \alpha)$.

\subsection{Real \texorpdfstring{$2$}{}-representations on \texorpdfstring{$2\mathsf{Vect}_k$}{}}
\label{sec:2VectRep}

We study twisted Real $2$-representations on $2 \mathsf{Vect}_k$. The ungraded case is treated in \cite{elgueta2007}; see also \cite[\S\S 5.1-2]{ganter2008}, \cite{osorno2010}.

Consider $2 \mathsf{Vect}_k$ with its weak duality involution $(-)^{\vee}$ from Section \ref{sec:bicatDual}. We begin with a cohomological classification of linear Real $2$-representations on $2 \mathsf{Vect}_k$. The underlying object $[n] \in \Obj(2 \mathsf{Vect}_k)$ of such a representation is called its dimension. Denote by $\mathfrak{S}_n$ the symmetric group on $n$ letters.

\begin{Thm}
\label{thm:Real2VectClass}
Equivalence classes of linear Real $2$-representations of $\mathsf{G}$ on $2\mathsf{Vect}_k$ of dimension $n$ are in bijection with equivalence classes of data consisting of
\begin{enumerate}[label=(\roman*)]
\item a group homomorphism $\rho_0: \hat{\mathsf{G}} \rightarrow \mathfrak{S}_n$, and
\item a class $[\hat{\theta}] \in H^2(B \hat{\mathsf{G}}, (k^{\times}_{\pi})_{\rho_0}^n)$, where $(k^{\times}_{\pi})_{\rho_0}^n$ is the abelian group $(k^{\times})^n$ with $\hat{\mathsf{G}}$-action
\[
\omega \cdot (a_1, \dots, a_n) = (a_{\rho_0(\omega)^{-1}(1)}^{\pi(\omega)}, \dots, a_{\rho_0(\omega)^{-1}(n)}^{\pi(\omega)}).
\]
\end{enumerate}
Two such data are equivalent if they differ by the action of $\mathfrak{S}_n$ on $\Hom_{\mathsf{Grp}}(\hat{\mathsf{G}}, \mathfrak{S}_n)$.
\end{Thm}

\begin{proof}
The proof is a modification of the classification in the ungraded case \cite[Theorem 5.5]{elgueta2007}, \cite[Proposition 4]{osorno2010}. Let $\rho$ be a linear Real $2$-representation of dimension $n$. For each $\omega \in \hat{\mathsf{G}}$, the $1$-morphism $\rho(\omega) : [n] \rightarrow [n]$ is an equivalence and hence is a permutation $2$-matrix (see \cite[Lemma 5.3]{ganter2008}). After noting that $(-)^{\vee}$ fixes the isomorphism class of a permutation $2$-matrix, the existence of a $2$-isomorphism $\rho(\omega_2) \circ \prescript{\pi(\omega_2)}{}{\rho}(\omega_1) \simeq \rho(\omega_2 \omega_1)$ implies that $\rho$ defines a group homomorphism $\rho_0: \hat{\mathsf{G}} \rightarrow \mathfrak{S}_n$. Fix a basis of each vector space appearing in each $1$-morphism $\rho(\omega)$. Then the $2$-isomorphism $\psi_{\omega_2, \omega_1}$ is given by an $n$-tuple $(\hat{\theta}_i([\omega_2 \vert \omega_1]))^n_{i=1} \in (k^{\times})^n$, the $i$\textsuperscript{th} component being the isomorphism between the unique one dimensional vector spaces of the $i$\textsuperscript{th} rows of $\rho(\omega_2) \circ \prescript{\pi(\omega_2)}{}{\rho(\omega_1)}$ and $\rho(\omega_2 \omega_1)$. By the associativity\footnote{As $2 \mathsf{Vect}_k$ is a bicategory which is not a $2$-category, a non-trivial associator $2$-isomorphism must be incorporated in equation \eqref{eq:assCond} in the obvious way.} constraint \eqref{eq:assCond}, the $\hat{\theta}_i$ assemble to a $2$-cocycle $\hat{\theta} \in Z^2(B \hat{\mathsf{G}}, (k^{\times}_{\pi})_{\rho_0}^n)$. A different choice of basis of $\rho(\omega)$ defines a cohomologous $2$-cocycle. Similarly, a contravariance respecting pseudonatural isomorphism $u: \rho \Rightarrow \rho^{\prime}$ defines a twisted $1$-cochain $\lambda_u$ such that $\theta = \theta^{\prime} \cdot d \lambda_u$. In this way, each equivalence class of linear Real $2$-representations of dimension $n$ defines a class in $H^2(B \hat{\mathsf{G}}, (k^{\times}_{\pi})_{\rho_0}^n)$.

Reversing the above construction associates to a $2$-cocycle $\hat{\theta} \in Z^2(B \hat{\mathsf{G}}, (k^{\times}_{\pi})_{\rho_0}^n)$ an $n$-dimensional linear Real $2$-representation, the entries of each $1$-morphism $\rho(\omega)$ being either trivial or the trivialized $k$-line. This association is quasi-inverse to the construction of the previous paragraph.
\end{proof}

The next result is a Real version of \cite[Theorem 10]{osorno2010}.

\begin{Prop}
\label{prop:Real2VectChar}
The Real $2$-character of the Real $2$-representation $\rho_{[\hat{\theta}]}$ determined by $[\hat{\theta}] \in H^2(B \hat{\mathsf{G}}, (k^{\times}_{\pi})_{\rho_0}^n)$ is
\[
\chi_{\rho_{[\hat{\theta}]}}(g,\omega) = \sum_{\substack{i \in \{1, \dots, n\} \\ \rho_0(g)(i) = \rho_0(\omega)(i)}} \hat{\theta}_i([g^{-1} \vert g])^{-\frac{\pi(\omega)-1}{2}} \frac{\hat{\theta}_i([\omega \vert g^{\pi(\omega)}])}{\hat{\theta}_i([g \vert \omega])}.
\]
\end{Prop}

\begin{proof}
A short calculation shows that the right hand side of the claimed formula is independent of the choice of normalized cocycle representative of $[\hat{\theta}]$. Fix such a choice $\hat{\theta}$. A $2$-morphism $\phi: 1_{[n]} \Rightarrow \rho_{\hat{\theta}}(g)$, $g \in \mathsf{G}$, is an $n\times n$ matrix which has non-zero entries only at those diagonals for which $\rho_{\hat{\theta}}(g)$ is non-zero. Using this observation, the computation of $\chi_{\rho_{\hat{\theta}}}$ reduces to the one dimensional case. In this case, direct inspection of the string diagrams \eqref{eq:charStriDiagPos} and \eqref{eq:charStriDiagNeg} shows that $\beta_{g, \omega}: k \rightarrow k$ is multiplication by
\[
\hat{\theta}([g^{-1} \vert g])^{-\frac{\pi(\omega)-1}{2}} \frac{\hat{\theta}([\omega \vert g^{\pi(\omega)}])}{\hat{\theta}([\omega g^{\pi(\omega)} \omega^{-1} \vert \omega])}.
\]
Upon restriction to graded commuting pairs, this gives the claimed formula.
\end{proof}

A linear variation of the example from Section \ref{sec:RealCatChar} can be used to give a geometric interpretation of Theorem \ref{thm:Real2VectClass}. Instead of a matrix model of $2\mathsf{Vect}_k$, we work with the model given by $k$-linear additive finitely semisimple categories. The object $[n] \in \Obj(2 \mathsf{Vect}_k)$ is modelled by the category of vector bundles $\mathsf{Vect}_k(X)$ over any set $X$ of cardinality $n$. Suppose that $\hat{\mathsf{G}}$ acts on $X$ and fix a $2$-cocycle $\hat{\theta} \in Z^2(X \git \hat{\mathsf{G}} , k^{\times}_{\pi})$ with coefficient system $k_{\pi}^{\times}$ twisted by the double cover $\pi: X \git \mathsf{G} \rightarrow X \git \hat{\mathsf{G}}$. Write $\hat{\theta}_x([\omega_2 \vert \omega_1])$ for the value of $\hat{\theta}$ on the $2$-chain $[\omega_2 \vert \omega_1]x$. A linear Real $2$-representation $\rho_{\hat{\theta}}$ is defined on $\mathsf{Vect}_k(X)$ as in the $\mathbb{F}_1$-linear case, where now the $2$-isomorphism $\psi_{\omega_2, \omega_1}$ multiplies the fibre over $x \in X$ by $\hat{\theta}_x([\omega_2 \vert \omega_1])$, pre-composing with $\ev$ if $\omega_1, \omega_2 \in \hat{\mathsf{G}} \backslash \mathsf{G}$. In this language, Theorem \ref{thm:Real2VectClass} becomes the statement that all linear Real $2$-representations of $\mathsf{G}$ in $2$-vector spaces arise in this way while Proposition \ref{prop:Real2VectChar} becomes the joint fixed point formula
\[
\chi_{\rho_{\hat{\theta}}}(g,\omega) = \sum_{x \in X^{g, \omega}} \uptau_{\pi}^{\refl}(\hat{\theta}_x)([\omega] g)^{-1}.
\]

To end this section we mention the twisted generalization of Theorem \ref{thm:Real2VectClass}. An $\hat{\alpha}$-twisted Real $2$-representation of $\mathsf{G}$ on $2\mathsf{Vect}_k$ determines an equivalence class of data consisting of
\begin{enumerate}[label=(\roman*)]
\item a group homomorphism $\rho_0: \hat{\mathsf{G}} \rightarrow \mathfrak{S}_n$,
\item a morphism $\rho_1: k^{\times}_{\pi} \rightarrow (k^{\times}_{\pi})^n_{\rho_0}$ of $\hat{\mathsf{G}}$-modules, and
\item a normalized $2$-cochain $\hat{\theta} \in C^2 (B \hat{\mathsf{G}}, (k^{\times}_{\pi})^n_{\rho_0})$
\end{enumerate}
such that the equality $d \hat{\theta} = \rho_1(\hat{\alpha})$ holds in $Z^3(B \hat{\mathsf{G}}, (k^{\times}_{\pi})^n_{\rho_0})$. This data, up to  equivalence as in Theorem \ref{thm:Real2VectClass}, classifies equivalence classes of twisted Real $2$-representations on $2 \mathsf{Vect}_k$. We omit the proof, which is a straightforward combination of those of \cite[Theorem 5.5]{elgueta2007} and Theorem \ref{thm:Real2VectClass}. Proposition \ref{prop:Real2VectChar} is unchanged.

\subsection{The antilinear theory}
\label{sec:antiLinearRealFGRep}

We explain a categorification of the antilinear theory of Real representations of a finite group from Section \ref{sec:RealFGRep}.

Let $k$ be a field which is a quadratic extension of a subfield $k_0$. Galois conjugation defines a strict $k_0$-linear involution $\overline{(-)} : \mathsf{Vect}_k \rightarrow \mathsf{Vect}_k$. Note that Section \ref{sec:RealFGRep} could have been formulated in terms of this involution, in much the same way that Section \ref{sec:linearRealFGRep} was formulated in terms of $(-)^{\vee}$.

Given a $k$-linear category $\mathcal{C}$, denote by $\overline{\mathcal{C}}$ the category with
\[
\Obj(\overline{\mathcal{C}}) = \Obj(\mathcal{C}), \qquad
\Hom_{\overline{\mathcal{C}}}(x,y) = \overline{\Hom_{\mathcal{C}}(x,y)}.
\]
The assignment $\mathcal{C} \mapsto \overline{\mathcal{C}}$ extends to a strict involutive $2$-functor $\overline{(-)} : \mathsf{Cat}_k \rightarrow \mathsf{Cat}_k$.

Let $\mathcal{V}$ be a $k$-linear bicategory. Define an involutive pseudofunctor $\overline{(-)}: \mathcal{V} \rightarrow \mathcal{V}$ by acting trivially on objects and by acting by conjugation on $1$-morphism categories. In particular, the action of $\overline{(-)}$ on $2$-morphisms is antilinear.

\begin{Def}
A linear Real $2$-representation of $\mathsf{G}$ on a $k$-linear $2$-category $\mathcal{V}$ (in the antilinear approach) consists of data $V \in \Obj(\mathcal{V})$, $\rho(\omega)$, $\psi_{\omega_2, \omega_1}$ and $\psi_e$ as in Section \ref{sec:basicDef}, with the constraint \eqref{eq:idenCond} unchanged but with the constraint \eqref{eq:assCond} replaced by the constraint
\begin{equation}
\label{eq:assCondReal}
\psi_{\omega_3 \omega_2, \omega_1} \circ  \left( \psi_{\omega_3, \omega_2} \circ \prescript{\pi(\omega_3 \omega_2)}{}{\rho(\omega_1)} \right) = \psi_{\omega_3, \omega_2 \omega_1} \circ \left( \rho(\omega_3) \circ \prescript{\pi(\omega_3)}{}{\psi}_{\omega_2, \omega_1} \right).
\end{equation}
Left superscripts now determine whether or not the $2$-functor $\overline{(-)}$ is applied.
\end{Def}

The above definition admits an obvious twisted generalization. All results of Section \ref{sec:RealProjReps}, and Section \ref{sec:indReal2Rep} below, continue to hold in the antilinear approach with essentially the same proofs, although Galois conjugation is used to define the coefficients of Real $2$-Schur multipliers. The key point to keep in mind is that while $\prescript{\pi(\omega_3)}{}{\psi}_{\omega_2, \omega_1}$ instead of $\prescript{\pi(\omega_3)}{}{\psi}^{\pi(\omega_3)}_{\omega_2, \omega_1}$ appears in equation \eqref{eq:assCondReal}, the pseudofunctor $\overline{(-)}: \mathcal{V} \rightarrow \mathcal{V}$ is antilinear on $2$-morphism spaces, as opposed to linear and direction reversing.

\section{Twisted Real induction}
\label{sec:indRealRep}

This section, which contains motivation and background material for that which follows, describes relevant aspects of the theory of twisted (Real) induction.

\subsection{Induction}
\label{sec:complexInd}

Let $\mathsf{G}$ be a finite group. Fix a cocycle $\theta \in Z^2(B \mathsf{G}, k^{\times})$. For later use, note that a $\theta$-twisted representation $\rho$ of $\mathsf{G}$ satisfies
\begin{equation}
\label{eq:projInverse}
\rho(g)^{-1} = \theta([g \vert g^{-1}])^{-1} \rho(g^{-1}), \qquad g \in \mathsf{G}.
\end{equation}
The character of $\rho$ satisfies (\textit{cf.} equation \eqref{eq:real1Char})
\begin{equation}
\label{eq:projCharInvar}
\chi_{\rho}(h g h^{-1}) = \uptau(\theta)([h]g) \chi_{\rho}(g), \qquad g,h \in \mathsf{G}.
\end{equation}

Let $\mathsf{H} \leq \mathsf{G}$ be a subgroup. The induction functor $\Ind_{\mathsf{H}}^{\mathsf{G}} : \mathsf{Rep}_k^{\theta_{\vert \mathsf{H}}} (\mathsf{H}) \rightarrow \mathsf{Rep}_k^{\theta} (\mathsf{G})$ can be defined as follows. Fix a complete set of representatives $\{r_1, \dots, r_p\}$ of $\mathsf{G} \slash \mathsf{H}$. Up to suitable equivalence, all constructions in this section, and the analogous constructions in those which follow, are independent of the choice of coset representatives. Given a $\theta_{\vert \mathsf{H}}$-twisted representation $\rho$ of $\mathsf{H}$ on $V$, set 
\[
\Ind_{\mathsf{H}}^{\mathsf{G}}(\rho)= \bigoplus_{i=1}^p r_i \cdot V.
\]
Here $r_i \cdot V$ is an isomorphic copy of $V$. For each $g \in \mathsf{G}$ and $i,j \in \{1, \dots, p\}$, set
\[
\Ind_{\mathsf{H}}^{\mathsf{G}}(\rho)(g)_{r_j r_i}=
\begin{cases}
\frac{\theta([g \vert r_i])}{\theta([r_j \vert r_j^{-1} g r_i])}  \rho(r_j^{-1} g r_i) & \mbox{if } r_j^{-1} g r_i \in \mathsf{H},\\
0 & \mbox{else}
\end{cases}
\]
where we view $\Ind_{\mathsf{H}}^{\mathsf{G}}(\rho)(g)$ as a $p \times p$ block matrix.

\begin{Prop}
\label{prop:indTwistChar}
The character of $\Ind_{\mathsf{H}}^{\mathsf{G}}(\rho)$ is
\[
\chi_{\Ind_{\mathsf{H}}^{\mathsf{G}}(\rho)}(g) = \frac{1}{\vert \mathsf{H} \vert} \sum_{ \substack{ r \in \mathsf{G} \\ r g r^{-1} \in \mathsf{H}} }
\uptau(\theta)([r] g)^{-1} \chi_{\rho}(r g r^{-1}).
\]
\end{Prop}

\begin{proof}
After a short calculation using the $2$-cocycle condition, the claimed equality becomes \cite[Proposition 4.1]{karpilovsky1985}.
\end{proof}

Proposition \ref{prop:indTwistChar} admits the following generalization to finite groupoids.

\begin{Prop}
\label{prop:twistedGrpdCharInd}
Let $f: \mathfrak{H} \rightarrow \mathfrak{G}$ be a faithful functor of finite groupoids. Fix $\theta \in Z^2(\mathfrak{G}, k^{\times})$. For each $\theta_{\vert \mathfrak{H}}$-twisted representation $\rho$ of $\mathfrak{H}$ and loop $(x,\gamma)$ in $\mathfrak{G}$, the equality
\[
\chi_{\Ind_{\mathfrak{H}}^{\mathfrak{G}}(\rho)}(x \xrightarrow[]{\gamma} x) = \sum_{y \in \Obj(\mathfrak{H})} \frac{1}{\vert \mathcal{O}_{\mathfrak{H}}(y) \vert \vert \Aut_{\mathfrak{H}}(y)\vert} \sum_{\substack{(x \xrightarrow[]{s} y) \in \mathfrak{G} \\ s \gamma s^{-1} \in \Aut_{\mathfrak{H}}(y)}} \uptau(\theta)([s]\gamma)^{-1} \chi_{\rho}(y \xrightarrow[]{s \gamma s^{-1}} y)
\]
holds, where $\mathcal{O}_{\mathfrak{H}}(y)$ denotes the orbit of $y$ in $\mathfrak{H}$.
\end{Prop}

\begin{proof}
The proof is a twisted generalization of \cite[Proposition 6.11]{ganter2008}. Let $\{y_1, \dots, y_n\}$ be a complete set of representatives for the $\mathfrak{H}$-isomorphism classes of objects which map to $\mathcal{O}_{\mathfrak{G}}(x)$ via $f$. For each $j \in \{1, \dots, n\}$, fix a morphism $x \xrightarrow[]{s_j} f(y_j)$. We have
\begin{eqnarray*}
\chi_{\Ind_{\mathfrak{H}}^{\mathfrak{G}}(\rho)}(x \xrightarrow[]{\gamma} x) &=& \sum_{j=1}^n \chi_{\Ind_{B\Aut_{\mathfrak{H}}(y_j)}^{\mathfrak{G}}(\rho)}(x \xrightarrow[]{\gamma} x) \\
&\overset{\scriptsize \mbox{Eq. } \eqref{eq:projCharInvar}}{=}& \sum_{j=1}^n \uptau(\theta)([s_j] \gamma)^{-1} \chi_{\Ind_{B\Aut_{\mathfrak{H}}(y_j)}^{\mathfrak{G}}(\rho)} \big( f(y_j) \xrightarrow[]{s_j \gamma s_j^{-1}} f(y_j) \big) \\
&\overset{\scriptsize \mbox{Prop. } \ref{prop:indTwistChar}}{=}&
\sum_{j=1}^n \frac{1}{\vert \Aut_{\mathfrak{H}}(y_j) \vert} \times \\
&& \sum_{\substack{ x \xrightarrow[]{s} y_j \\ s \gamma s^{-1} \in \Aut_{\mathfrak{H}}(y_j)}} \uptau(\theta)([s_j] \gamma)^{-1} \uptau(\theta)([s s_j^{-1}] s_j \gamma s_j^{-1})^{-1} \chi_{\rho}(s \gamma s^{-1}) \\
&\overset{\scriptsize \mbox{Eq. } \eqref{eq:twistTransPartialCocycle}}{=}& 
\sum_{j=1}^n \frac{1}{\vert \Aut_{\mathfrak{H}}(y_j) \vert} \sum_{\substack{ x \xrightarrow[]{s} y_j \\ s \gamma s^{-1} \in \Aut_{\mathfrak{H}}(y_j)}} \uptau(\theta)([s] \gamma )^{-1} \chi_{\rho}(s \gamma s^{-1}),
\end{eqnarray*}
which is easily seen to equal the desired expression. Note that Proposition \ref{prop:indTwistChar} applies because of the assumption that $f$ is faithful.
\end{proof}

\subsection{Real induction}
\label{sec:RealInd}

Let $\hat{\mathsf{H}} \leq \hat{\mathsf{G}}$ be a (non-trivially graded) $\mathbb{Z}_2$-graded subgroup. Fix $\hat{\theta} \in Z^2(B \hat{\mathsf{G}} , k^{\times}_{\pi})$. In this section we interpret Real representations as generalized symmetric representations, as in Section \ref{sec:linearRealFGRep}. The Real induction functor $\RInd^{\hat{\mathsf{G}}}_{\hat{\mathsf{H}}} : \mathsf{RRep}_k^{\hat{\theta}_{\vert \hat{\mathsf{H}}}}(\mathsf{H}) \rightarrow \mathsf{RRep}_k^{\hat{\theta}}(\mathsf{G})$ is defined as follows. Fix a complete set $\mathcal{S}=\{\sigma_1, \dots, \sigma_q\}$ of representatives of $\hat{\mathsf{G}} \slash \hat{\mathsf{H}}$. Let $\rho$ be a $\hat{\theta}_{\vert \hat{\mathsf{H}}}$-twisted Real representation of $\mathsf{H}$ on $V$. As a vector space, put 
\[
\RInd^{\hat{\mathsf{G}}}_{\hat{\mathsf{H}}}(\rho) = 
\bigoplus_{i=1}^q \sigma_i \cdot {^{\pi(\sigma_i)}}V.
\]
For each $\omega \in \hat{\mathsf{G}}$ and $i,j \in \{1, \dots, q\}$, set
\[
\RInd_{\hat{\mathsf{H}}}^{\hat{\mathsf{G}}}(\rho)(\omega)_{\sigma_j \sigma_i}=
\begin{cases}
\frac{\hat{\theta}([\omega \vert \sigma_i])}{\hat{\theta}([\sigma_j \vert \sigma_j^{-1} \omega \sigma_i])} \cdot \prescript{\pi(\sigma_j)}{}{\rho}(\sigma_j^{-1} \omega \sigma_i)^{\pi(\sigma_j)} & \mbox{if } \sigma_j^{-1} \omega \sigma_i \in \hat{\mathsf{H}},\\
0 & \mbox{else}.
\end{cases}
\]
Given a morphism $\phi: \rho_1 \rightarrow \rho_2$ of twisted Real representations, the value of $\RInd_{\hat{\mathsf{H}}}^{\hat{\mathsf{G}}}(\phi)$ on the $i$\textsuperscript{th} summand is $\phi$ if $\pi(\sigma_i)=1$ and is $\rho_2(\sigma_i)^{-1} \circ \phi \circ \rho_1(\sigma_i)$ if $\pi(\sigma_i)=-1$.

Observe that there is a natural isomorphism
\[
\Res_{\mathsf{G}}^{\hat{\mathsf{G}}} \circ \RInd_{\hat{\mathsf{H}}}^{\hat{\mathsf{G}}} \simeq \Ind_{\mathsf{H}}^{\mathsf{G}} \circ \Res_{\mathsf{H}}^{\hat{\mathsf{H}}}.
\]
This isomorphism can be used to compute $\chi_{\RInd_{\hat{\mathsf{H}}}^{\hat{\mathsf{G}}}(\rho)}$. However, for comparison with the case of Real $2$-representations, it is instructive to compute $\chi_{\RInd_{\hat{\mathsf{H}}}^{\hat{\mathsf{G}}}(\rho)}$ directly.

\begin{Prop}
\label{prop:RealIndChar}
The Real character of $\RInd_{\hat{\mathsf{H}}}^{\hat{\mathsf{G}}}(\rho)$ is
\[
\chi_{\RInd_{\hat{\mathsf{H}}}^{\hat{\mathsf{G}}}(\rho)}(g) = \frac{1}{2 \vert \mathsf{H} \vert} \sum_{ \substack{ \omega \in \hat{\mathsf{G}} \\ \omega g^{\pi(\omega)} \omega^{-1} \in \mathsf{H}} }
\uptau_{\pi}^{\refl}(\hat{\theta})([\omega] g)^{-1} \chi_{\rho}(\omega g^{\pi(\omega)} \omega^{-1}).
\]
\end{Prop}

\begin{proof}
For each $g \in \mathsf{G}$, we compute
\begin{eqnarray*}
\chi_{\RInd_{\hat{\mathsf{H}}}^{\hat{\mathsf{G}}}(\rho)}(g) 
&=&
\sum_{ \substack{ i \in \{1, \dots, q\} \\ \sigma_i^{-1} g \sigma_i \in \mathsf{H}} }
\frac{\hat{\theta}([g \vert \sigma_i])}{\hat{\theta}([\sigma_i \vert \sigma_i^{-1} g \sigma_i])} \tr_{\prescript{\pi(\sigma_i)}{}{V}} \left( \prescript{\pi(\sigma_i)}{}{\rho(\sigma_i^{-1} g \sigma_i)}^{\pi(\sigma_i)} \right) \\
&=&
\frac{1}{2 \vert \mathsf{H} \vert}
\sum_{ \substack{\sigma \in \hat{\mathsf{G}} \\ \sigma g^{\pi(\sigma)} \sigma^{-1} \in \mathsf{H}} }
\frac{\hat{\theta}([g \vert \sigma^{-1}])}{\hat{\theta}([\sigma^{-1} \vert \sigma g \sigma^{-1}])} \tr_V \left( \rho(\sigma^{-1} g \sigma)^{\pi(\sigma)} \right).
\end{eqnarray*}
After using equation \eqref{eq:projInverse}, this is seen to equal
\[\frac{1}{2 \vert \mathsf{H} \vert}
\sum_{ \substack{\sigma \in \hat{\mathsf{G}} \\ \sigma g^{\pi(\sigma)} \sigma^{-1} \in \mathsf{H}} }
\frac{\hat{\theta}([g \vert \sigma^{-1}])}{\hat{\theta}([\sigma^{-1} \vert \sigma g \sigma^{-1}])} \hat{\theta}([\sigma g \sigma^{-1} \vert \sigma g^{-1} \sigma^{-1}])^{\frac{\pi(\sigma)-1}{2}} \chi_{\rho}(\sigma g^{\pi(\sigma)} \sigma^{-1}).
\]
A short calculation using equation \eqref{eqn:keyIdentity} now completes the proof.
\end{proof}

\subsection{Hyperbolic induction}
\label{sec:hypInd}

A second form of Real induction, different from that of Section \ref{sec:RealInd}, is hyperbolic induction $\HInd^{\hat{\mathsf{G}}}_{\mathsf{G}} : \mathsf{Rep}_k^{\theta}(\mathsf{G}) \rightarrow \mathsf{RRep}_k^{\hat{\theta}}(\mathsf{G})$. This is simply the hyperbolic functor from Grothendieck--Witt theory. It admits the following explicit description. Fix an element $\varsigma \in \hat{\mathsf{G}} \backslash \mathsf{G}$. Given a $\theta$-twisted representation $\rho$ of $\mathsf{G}$ on $V$, the underlying vector space of $\HInd^{\hat{\mathsf{G}}}_{\mathsf{G}}(\rho)$ is $V \oplus V^{\vee}$. By Lemma \ref{lem:dualProjRep} and Proposition \ref{prop:symmGenSymmEquiv}, the $\hat{\mathsf{G}}$-action is given by
\[
\HInd_{\mathsf{G}}^{\hat{\mathsf{G}}} (\rho)(g)= 
\left(
\begin{array}{cc}
\rho(g) & 0 \\
0& \uptau_{\pi}^{\refl}(\hat{\theta})([\varsigma^{-1}] g)^{-1} \rho(\varsigma^{-1} g^{-1} \varsigma)^{\vee}
\end{array}
\right), \qquad g \in \mathsf{G}
\]
and
\[
\HInd_{\mathsf{G}}^{\hat{\mathsf{G}}} (\rho)(\omega)= 
\left(
\begin{array}{cc}
0 & \frac{\hat{\theta}([\omega \vert \varsigma])}{\hat{\theta}([\varsigma^{-1} \vert \varsigma])} \rho( \omega \varsigma) \\
\frac{\hat{\theta}([\omega \vert \omega^{-1}]) \hat{\theta}([\omega^{-1} \vert \varsigma])}{\hat{\theta}([\varsigma^{-1} \vert \varsigma])} \rho( \omega^{-1} \varsigma)^{\vee} & 0
\end{array}
\right), \qquad \omega \in \hat{\mathsf{G}} \backslash \mathsf{G}.
\]

More generally, given a subgroup $\mathsf{H} \leq \mathsf{G}$, the functor $\HInd^{\hat{\mathsf{G}}}_{\mathsf{H}} : \mathsf{Rep}_k^{\theta_{\vert \mathsf{H}}}(\mathsf{H}) \rightarrow \mathsf{RRep}_k^{\hat{\theta}}(\mathsf{G})$ is defined to be the composition $\HInd^{\hat{\mathsf{G}}}_{\mathsf{G}} \circ \Ind^{\mathsf{G}}_{\mathsf{H}}$.

\begin{Prop}
\label{prop:RealHIndChar}
The Real character of $\HInd^{\hat{\mathsf{G}}}_{\mathsf{H}}(\rho)$ is
\[
\chi_{\HInd^{\hat{\mathsf{G}}}_{\mathsf{H}}(\rho)}(g) = \frac{1}{\vert \mathsf{H} \vert} \sum_{\substack{ \omega \in \hat{\mathsf{G}} \\ \omega g^{\pi(\omega)} \omega^{-1} \in \mathsf{H} } }
\uptau_{\pi}^{\refl}(\hat{\theta})([\omega] g)^{-1} \chi_{\rho}(\omega g^{\pi(\omega)} \omega^{-1}).
\]
\end{Prop}

\begin{proof}
Suppose first that $\mathsf{H} = \mathsf{G}$. For each $g \in \mathsf{G}$, we compute
\begin{eqnarray*}
\chi_{\HInd^{\hat{\mathsf{G}}}_{\mathsf{G}}(\rho)}(g) &=& \chi_{\rho}(g) + \uptau_{\pi}^{\refl}(\hat{\theta})([\varsigma^{-1}] g)^{-1} \chi_{\rho}(\varsigma^{-1} g^{-1} \varsigma) \\
&\overset{\scriptsize \mbox{Eq. }\eqref{eq:projCharInvar}}{=}& \frac{1}{\vert \mathsf{G} \vert} \sum_{s \in \mathsf{G}} \uptau(\theta)([s]g)^{-1} \chi_{\rho}(s g s^{-1}) + \\
&& \frac{1}{\vert \mathsf{G} \vert} \sum_{s \in \mathsf{G}} \uptau_{\pi}^{\refl}(\hat{\theta})([\varsigma^{-1}] g)^{-1} \uptau(\theta)([s]\varsigma^{-1} g^{-1} \varsigma)^{-1} \chi_{\rho}(s \varsigma^{-1} g^{-1} \varsigma s^{-1}) \\
&\overset{\scriptsize \mbox{Eq. }\eqref{eq:twistTransPartialCocycle}}{=}& \frac{1}{\vert \mathsf{G} \vert} \sum_{\omega \in \hat{\mathsf{G}}} \uptau_{\pi}^{\refl}(\hat{\theta})([\omega] g)^{-1}  \chi_{\rho}(\omega g^{\pi(\omega)} \omega^{-1}).
\end{eqnarray*}
The case of an arbitrary subgroup follows by combining the previous case with equation \eqref{eq:twistTransPartialCocycle} and Proposition \ref{prop:indTwistChar}.
\end{proof}

\section{Twisted Real \texorpdfstring{$2$}{}-induction}
\label{sec:indReal2Rep}

We define and study various forms of induction for linear Real representations of finite categorical groups. In particular, we compute the result of induction at the level of Real $2$-characters.

\subsection{Twisted \texorpdfstring{$2$}{}-induction}

Let $\mathsf{G}$ be a finite group with subgroup $\mathsf{H}$. Fix $\alpha \in Z^3(B \mathsf{G}, k^{\times})$. Let $\rho$ be a linear representation of $\mathcal{H} = \mathcal{G}(\mathsf{H}, \alpha_{\vert \mathsf{H}})$ on a category $V$. The induced linear representation $\Ind_{\mathcal{H}}^{\mathcal{G}}(\rho)$ of $\mathcal{G}= \mathcal{G}(\mathsf{G}, \alpha)$ was constructed in \cite[Proposition 5.6]{ganter2016}. An explicit construction, generalizing that of \cite[\S 7.1]{ganter2008} in the untwisted case, is as follows. Keeping the notation from Section \ref{sec:complexInd}, set
\[
\Ind_{\mathcal{H}}^{\mathcal{G}} (\rho) = \prod_{i=1}^p r_i \cdot V.
\]
An element $g \in \mathsf{G}$ then acts via the $p \times p$ matrix whose $(j,i)$\textsuperscript{th} entry is
\[
\Ind_{\mathcal{H}}^{\mathcal{G}} (\rho)(g)_{r_j r_i} =
\begin{cases}
\rho(h) & \mbox{if } g r_i = r_j h \mbox{ for some } h \in \mathsf{H}, \\
0& \mbox{ else.}
\end{cases}
\]
The $(k, i)$\textsuperscript{th} entry of $\Ind_{\mathcal{H}}^{\mathcal{G}} (\rho)(g_2) \circ \Ind_{\mathcal{H}}^{\mathcal{G}} (\rho)(g_1)$ is $\rho(h_2) \circ \rho(h_1)$ if $g_1 r_i = r_j h_1$ and $g_2 r_j = r_k h_2$ for some $h_1, h_2 \in \mathsf{H}$ and is zero otherwise. If non-trivial, the corresponding entry of the $2$-isomorphism $\Ind_{\mathcal{H}}^{\mathcal{G}}(\psi)_{g_2, g_1}$ is defined to be
\[
\frac{\alpha([g_2 \vert r_j \vert h_1])}{\alpha([g_2 \vert g_1 \vert r_i]) \alpha([r_k \vert h_2 \vert h_1])}
\psi_{h_2, h_1}.
\]
The coefficient of $\psi_{h_2, h_1}$ ensures that this defines a Real representation of $\mathcal{G}$.

\begin{Thm}
\label{thm:twistedIndFunctProj}
There is an isomorphism
\[
\Tr(\Ind_{\mathcal{H}}^{\mathcal{G}} (\rho)) \simeq \Ind_{\Lambda B \mathsf{H}}^{\Lambda B \mathsf{G}} \left( \Tr(\rho) \right)
\]
of $\uptau(\alpha)$-twisted representations of $\Lambda B \mathsf{G}$.
\end{Thm}

Here $\Tr(\rho)$ is viewed as a $\uptau(\alpha_{\vert \mathsf{H}})$-twisted representation of $\Lambda B \mathsf{H}$ and $\Ind_{\Lambda B \mathsf{H}}^{\Lambda B \mathsf{G}}$ denotes twisted induction for groupoids. The untwisted version of Theorem \ref{thm:twistedIndFunctProj} was proved by Ganter--Kapranov \cite[Theorem 7.5]{ganter2008}. The twisted case can be handled by an elaboration of their argument. We omit this argument, as a further elaboration will be used to prove Theorem \ref{thm:Real2RealIndFunct} below.

\begin{Cor}[{\textit{cf.} \cite[Corollary 7.6]{ganter2008}}]
\label{cor:ind2Char}
The $2$-character of $\Ind_{\mathcal{H}}^{\mathcal{G}} (\rho)$ is
\[
\chi_{\Ind_{\mathcal{H}}^{\mathcal{G}} (\rho)}(g_1, g_2) = \frac{1}{\vert \mathsf{H} \vert} \sum_{\substack{s \in \mathsf{G} \\ s(g_1,g_2) s^{-1} \in \mathsf{H}^2}} \uptau^2(\alpha)([s]g_1 \xrightarrow[]{g_2} g_1)^{-1} \cdot \chi_{\rho}(s g_1 s^{-1}, s g_2 s^{-1}).
\]
\end{Cor}

\begin{proof}
By Theorem \ref{thm:twistedIndFunctProj}, it suffices to compute the character of $\Ind_{\Lambda B \mathsf{H}}^{\Lambda B \mathsf{G}} \left( \Tr(\rho) \right)$. As the canonical functor $\Lambda B \mathsf{H} \rightarrow \Lambda B \mathsf{G}$ is faithful, this character can be computed using Proposition \ref{prop:twistedGrpdCharInd}. Doing so gives the claimed result.
\end{proof}

\subsection{Real \texorpdfstring{$2$}{}-induction}
\label{sec:Real2Ind}

Consider now a finite $\mathbb{Z}_2$-graded group $\hat{\mathsf{G}}$ with $\mathbb{Z}_2$-graded subgroup $\hat{\mathsf{H}}$. Let $\hat{\alpha} \in Z^3 (B \hat{\mathsf{G}}, k^{\times}_{\pi})$ and let $\rho$ be a linear Real representation of $\mathcal{H}=\mathcal{G}(\mathsf{H}, \alpha_{\vert \mathsf{H}})$ on a category $V$, the Real structure being $\hat{\mathcal{H}}=\mathcal{G}(\hat{\mathsf{H}}, \hat{\alpha}_{\vert \hat{\mathsf{H}}})$. We define a Real $2$-representation $\RInd_{\hat{\mathcal{H}}}^{\hat{\mathcal{G}}}(\rho)$ of $\mathcal{G}=\mathcal{G}(\mathsf{G}, \alpha)$ with Real structure $\hat{\mathcal{G}}=\mathcal{G}(\hat{\mathsf{G}}, \hat{\alpha})$. As a category, set
\[
\RInd_{\hat{\mathcal{H}}}^{\hat{\mathcal{G}}}(\rho) = \prod_{i=1}^q \sigma_i \cdot \prescript{\pi(\sigma_i)}{}{V}.
\]
An element $\omega \in \hat{\mathsf{G}}$ acts by the $q \times q$ matrix whose $(j,i)$\textsuperscript{th} entry is
\[
\RInd_{\hat{\mathcal{H}}}^{\hat{\mathcal{G}}}(\rho)(\omega)_{\sigma_j \sigma_i} =
\begin{cases}
\prescript{\pi(\sigma_j)}{}{\rho}(\eta) & \mbox{if } \omega \sigma_i = \sigma_j \eta \mbox{ for some } \eta \in \hat{\mathsf{H}}, \\
0& \mbox{ else.}
\end{cases}
\]
The $(k,i)$\textsuperscript{th} entry of $\RInd_{\hat{\mathcal{H}}}^{\hat{\mathcal{G}}}(\rho)(\omega_2) \circ {^{\pi(\omega_2)}} \RInd_{\hat{\mathcal{H}}}^{\hat{\mathcal{G}}}(\rho) (\omega_1)$ is $\prescript{\pi(\sigma_j)}{}{\rho}(\omega^{\prime}_2) \circ \prescript{\pi(\eta_2 \sigma_j)}{}{\rho}(\eta_1)$ if $\omega_1 \sigma_i = \sigma_j \eta_1$ and $\omega_2 \sigma_j = \sigma_k \eta_2$ for some $\eta_1, \eta_2 \in \hat{\mathsf{H}}$ and is zero otherwise. The component of $\RInd_{\hat{\mathcal{H}}}^{\hat{\mathcal{G}}}(\psi)_{\omega_2, \omega_1}$ at this entry is defined to be
\[
\frac{\hat{\alpha}([\omega_2 \vert \sigma_j \vert \eta_1])}{\hat{\alpha}([\omega _2 \vert \omega_1 \vert \sigma_i]) \hat{\alpha}([\sigma_k \vert \eta_2 \vert \eta_1])} \cdot 
\prescript{\pi(\sigma_k)}{}{\psi}^{\pi(\sigma_k)}_{\eta_2, \eta_1}.
\]
It is straightforward to verify that this defines a linear Real representation of $\mathcal{G}$.

\subsection{Induced Real categorical and \texorpdfstring{$2$}{}-characters}
\label{sec:Real2IndCar}

We generalize the work of Ganter--Kapranov in the untwisted ungraded case to compute the Real categorical and $2$-characters of $\RInd_{\hat{\mathcal{H}}}^{\hat{\mathcal{G}}}(\rho)$. We begin with the Real categorical character.

\begin{Thm}
\label{thm:Real2RealIndFunct}
There is a canonical isomorphism
\[
\Tr(\RInd_{\hat{\mathcal{H}}}^{\hat{\mathcal{G}}}(\rho))
\simeq
\Ind_{\Lambda_{\pi}^{\refl} B \hat{\mathsf{H}}}^{\Lambda_{\pi}^{\refl} B \hat{\mathsf{G}}} \left( \Tr(\rho) \right)
\]
of $\uptau_{\pi}^{\refl}(\hat{\alpha})$-twisted representations of $\Lambda_{\pi}^{\refl} B \hat{\mathsf{G}}$.
\end{Thm}

Let $\mathsf{Z}_{\hat{\mathsf{G}}}^{\varphi}(g)$ be the stabilizer of $g \in \mathsf{G}$ under Real $\hat{\mathsf{G}}$-conjugation. Fix an equivalence
\begin{equation}
\label{eq:unoriLoopDecomp}
\Lambda_{\pi}^{\refl} B \hat{\mathsf{G}} \simeq \bigsqcup_{g \in \pi_0(\Lambda_{\pi}^{\refl} B \hat{\mathsf{G}})} B \mathsf{Z}_{\hat{\mathsf{G}}}^{\varphi}(g).
\end{equation}
The set $\pi_0(\Lambda_{\pi}^{\refl} B \hat{\mathsf{G}})$ is identified with the set of Real conjugacy classes of $\mathsf{G}$. Denote by $\pser{g}_{\hat{\mathsf{G}}} \subset \mathsf{G}$ the Real conjugacy class of $g$. To prove Theorem \ref{thm:Real2RealIndFunct} we first describe the action of $\mathsf{Z}_{\hat{\mathsf{G}}}^{\varphi}(g)$ on $\Tr_{\RInd_{\hat{\mathcal{H}}}^{\hat{\mathcal{G}}} (\rho)}(g)$. We require some notation. The decomposition
\begin{equation}
\label{eq:RealConjDecomp}
\pser{g}_{\hat{\mathsf{G}}} \cap \mathsf{H} = \bigsqcup_{i=1}^l \pser{h_i}_{\hat{\mathsf{H}}}
\end{equation}
induces a decomposition
\[
\{\sigma \in \mathcal{S} \mid \sigma^{-1} g^{\pi(\sigma)} \sigma \in \mathsf{H}\} = \bigsqcup_{i=1}^l \mathcal{S}_i
\]
where $\mathcal{S}_i = \{ \sigma \in \mathcal{S} \mid \sigma^{-1} g^{\pi(\sigma)} \sigma \in \pser{h_i}_{\hat{\mathsf{H}}} \}$. For each $i \in \{1, \dots, l\}$, fix an element $\sigma_i \in \mathcal{S}_i$. Relabel the representatives of the Real $\hat{\mathsf{H}}$-conjugacy classes appearing in the decomposition \eqref{eq:RealConjDecomp} by $h_i = \sigma_i^{-1} g^{\pi(\sigma_i)} \sigma_i$.

\begin{Lem}[{\textit{cf.} \cite[Lemma 7.7]{ganter2008}}]
\label{lem:repChoice}
Elements of $\mathcal{S}_i$ can be chosen so that left multiplication by $\sigma_i^{-1}$ induces a bijection from $\mathcal{S}_i$ to a complete system of representatives of $\mathsf{Z}_{\hat{\mathsf{G}}}^{\varphi}(h_i) \slash \mathsf{Z}_{\hat{\mathsf{H}}}^{\varphi}(h_i)$.
\end{Lem}

\begin{proof}
The proof is nearly the same as that of \cite[Lemma 7.7]{ganter2008}; we include it for completeness. Let $\sigma \in \mathcal{S}_i$. Then $\sigma^{-1} g^{\pi(\sigma)} \sigma = \eta^{-1} h_i^{\pi(\eta)} \eta$ for some $\eta \in \hat{\mathsf{H}}$. It follows that $\sigma \eta^{-1}=\sigma$ in $\hat{\mathsf{G}} \slash \hat{\mathsf{H}}$ and $(\sigma \eta^{-1})^{-1} g^{\pi(\sigma \eta)} \sigma \eta^{-1} = h_i$ so that
\[
(\sigma_i^{-1} \sigma \eta^{-1})^{-1} h_i^{\pi(\sigma_i^{-1} \sigma \eta^{-1})} \sigma_i^{-1} \sigma \eta^{-1} = h_i
\]
and $\sigma_i^{-1} \sigma \eta^{-1} \in \mathsf{Z}_{\hat{\mathsf{G}}}^{\varphi}(h_i)$. Replacing $\sigma$ with $\sigma \eta^{-1}$, we henceforth assume that $\sigma \in \mathcal{S}_i$ is such that $\sigma^{-1} g^{\pi(\sigma)} \sigma= h_i$ and $\sigma_i^{-1} \sigma \in \mathsf{Z}_{\hat{\mathsf{G}}}^{\varphi}(h_i)$.

Let $\sigma, \sigma^{\prime} \in \mathcal{S}_i$ be distinct. Then $(\sigma_i^{-1} \sigma)^{-1} (\sigma_i^{-1} \sigma^{\prime}) = \sigma^{-1} \sigma^{\prime}$ does not lie in $\hat{\mathsf{H}}$. It follows that $\sigma_i^{-1} \sigma \neq \sigma_i^{-1} \sigma^{\prime}$ in $\mathsf{Z}_{\hat{\mathsf{G}}}^{\varphi}(h_i) \slash \mathsf{Z}_{\hat{\mathsf{H}}}^{\varphi}(h_i)$, proving injectivity of the map under consideration. To prove surjectivity, let $\mu \in \mathsf{Z}_{\hat{\mathsf{G}}}^{\varphi}(h_i)$. Then $\sigma_i \mu = \sigma \eta$ for some $\sigma \in \mathcal{S}$ and $\eta \in \hat{\mathsf{H}}$. We compute
\[
\sigma^{-1} g^{\pi(\sigma)} \sigma = \eta h_i^{\pi(\sigma)} \eta^{-1},
\]
whence $\sigma \in \mathcal{S}_i$. Since $\sigma^{-1} g^{\pi(\sigma)} \sigma = h_i$, we also find that $\eta \in \mathsf{Z}_{\hat{\mathsf{H}}}^{\varphi}(h_i)$. So $\sigma_i^{-1} \sigma = \mu \eta^{-1}$, showing that $\sigma_i^{-1} \sigma = \mu$ in $\mathsf{Z}_{\hat{\mathsf{G}}}^{\varphi}(h_i) \slash \mathsf{Z}_{\hat{\mathsf{H}}}^{\varphi}(h_i)$.
\end{proof}

\begin{Rem}
The representatives $\mathcal{S}$ of $\hat{\mathsf{G}} \slash \hat{\mathsf{H}}$ can be chosen to be a subset of $\mathsf{G}$. Such a choice simplifies the description of induced Real $2$-representations. However, it does not appear that there is a version of Lemma \ref{lem:repChoice} which produces a set of representatives which is again a subset of $\mathsf{G}$.
\end{Rem}

We henceforth assume that $\mathcal{S}$ is chosen as in Lemma \ref{lem:repChoice}. We have
\begin{eqnarray*}
\Tr_{\RInd_{\hat{\mathcal{H}}}^{\hat{\mathcal{G}}}(\rho)}(g)
&=&
\bigoplus_{i=1}^l \bigoplus_{\{\sigma \in \mathcal{S}_i \mid \sigma^{-1} g \sigma \in \mathsf{H}\}} \sigma \cdot \Tr_{\prescript{\pi(\sigma)}{}{\rho}}(h_i^{\pi(\sigma)}).
\end{eqnarray*}
If $\pi(\sigma)=-1$, then
\[
\Tr_{\rho^{\op}}(h_i^{-1}) \simeq 2\Hom_{\mathsf{Cat}}(\rho(h_i^{-1}),1_V).
\]
Define $k$-linear a map $F_i: 2\Hom_{\mathsf{Cat}}(\rho(h_i^{-1}),1_V) \rightarrow 2\Hom_{\mathsf{Cat}}(1_V, \rho(h_i))$
by
\[
\begin{gathered}
\begin{tikzpicture}[scale=0.2,inner sep=0.35mm, place/.style={circle,draw=black,fill=black,thick}]
\draw [decoration={markings, mark=at position 0.65 with {\arrow{>}}},        postaction={decorate}] (0,0) -- (0,4);
\draw (0,0) node [shape=circle,draw,fill] {};
\node [right,label={[label distance=0.4mm]0: \scriptsize$u$}] at (0,0) {};
\node [left,label={[label distance=0.4mm]180:\scriptsize$h_i$}] at (0,2.0) {};
\end{tikzpicture}
\end{gathered}
\qquad
\longmapsto
\qquad
\begin{gathered}
\begin{tikzpicture}[scale=0.2,inner sep=0.35mm, place/.style={circle,draw=black,fill=black,thick}]
\draw[black, decoration={markings, mark=at position 0.2 with {\arrow{>}}}, decoration={markings, mark=at position 0.82 with {\arrow{<}}},        postaction={decorate}](0,0) [partial ellipse=180:0:3 and 5];
\node at (4.9,2.5) {$\scriptstyle  h_i$};
\draw (3,0) node [shape=circle,draw,fill] {};
\node [right,label={[label distance=0.4mm]0:\scriptsize$u$}] at (3,0) {};
\end{tikzpicture}
\end{gathered}.
\]
Using equation \eqref{diag:snakeRemoval}, the inverse $F_i^{-1} : 2\Hom_{\mathsf{Cat}}(1_V, \rho(h_i)) \rightarrow 2\Hom_{\mathsf{Cat}}(\rho(h_i^{-1}),1_V)$ is seen to be
\[
\begin{gathered}
\begin{tikzpicture}[scale=0.2,inner sep=0.35mm, place/.style={circle,draw=black,fill=black,thick}]
\draw [decoration={markings, mark=at position 0.65 with {\arrow{>}}},        postaction={decorate}] (0,0) -- (0,4);
\draw (0,4) node [shape=circle,draw,fill] {};
\node [right,label={[label distance=0.4mm]0: \scriptsize$v$}] at (0,4) {};
\node [left,label={[label distance=0.4mm]180:\scriptsize$h^{-1}_i$}] at (0,2.0) {};
\end{tikzpicture}
\end{gathered}
\qquad
\longmapsto
\qquad
\hat{\alpha}([h_i \vert h^{-1}_i \vert h_i]) \times
\begin{gathered}
\begin{tikzpicture}[scale=0.2,inner sep=0.35mm, place/.style={circle,draw=black,fill=black,thick}]
\draw[black, decoration={markings, mark=at position 0.15 with {\arrow{<}}}, decoration={markings, mark=at position 0.895 with {\arrow{>}}},        postaction={decorate}](0,0) [partial ellipse=-180:0:3 and 5];
\node at (4.9,-2.5) {$\scriptstyle  h_i^{-1}$};
\draw (3,0) node [shape=circle,draw,fill] {};
\node [right,label={[label distance=0.4mm]0:\scriptsize$v$}] at (3,0) {};
\end{tikzpicture}
\end{gathered}.
\]
The maps $F_i$ induce a vector space isomorphism
\[
\Tr_{\RInd_{\hat{\mathcal{H}}}^{\hat{\mathcal{G}}}(\rho)}(g)
\simeq
\bigoplus_{i=1}^l 
\bigoplus_{\{\sigma \in \mathcal{S}_i \mid \sigma^{-1} g \sigma \in \mathsf{H}\}} \sigma \cdot \Tr_{\rho}(h_i)
\]
which we use below without mention.

\begin{Lem}[{\textit{cf.} \cite[Lemma 7.8]{ganter2008}}]
\label{lem:summandIdent}
There is an isomorphism
\[
\Tr_{\RInd_{\hat{\mathcal{H}}}^{\hat{\mathcal{G}}} (\rho)}(g) \simeq \bigoplus_{i=1}^l \Ind_{\mathsf{Z}_{\hat{\mathsf{H}}}^{\varphi}(h_i)}^{\mathsf{Z}_{\hat{\mathsf{G}}}^{\varphi}(g)} \left( \Tr_{\rho}(h_i) \right)
\]
of $\uptau_{\pi}^{\refl}(\hat{\alpha})$-twisted representations of $\mathsf{Z}_{\hat{\mathsf{G}}}^{\varphi}(g)$, the induction being along the composition $\mathsf{Z}_{\hat{\mathsf{H}}}^{\varphi}(h_i) \hookrightarrow \mathsf{Z}_{\hat{\mathsf{G}}}^{\varphi}(h_i) \xrightarrow[]{l \mapsto \sigma_i l \sigma_i^{-1}} \mathsf{Z}_{\hat{\mathsf{G}}}^{\varphi}(g)$.
\end{Lem}

\begin{proof}
Let $\mu \in \mathsf{Z}_{\hat{\mathsf{G}}}^{\varphi}(g)$ and $\sigma \in \mathcal{S}_i$. Then $
\mu \sigma = \tilde{\sigma} \eta$ for some $\tilde{\sigma} \in \mathcal{S}$ and $\eta \in \hat{\mathsf{H}}$. It is straightforward to verify that in fact $\tilde{\sigma} \in \mathcal{S}_i$ and $\eta \in \mathsf{Z}_{\hat{\mathsf{H}}}^{\varphi}(h_i)$. The equations
\[
\mu^{-1} \tilde{\sigma} = \sigma \eta^{-1} , \qquad g \sigma = \sigma h_i^{\pi(\sigma)}, \qquad \mu \sigma = \tilde{\sigma} \eta
\]
imply that $\mu$ acts on $\Tr_{\RInd_{\hat{\mathcal{H}}}^{\hat{\mathcal{G}}} (\rho)}(g)$ by a linear map $\xi_1(\mu) : \sigma \cdot \Tr_{\rho}(h_i) \rightarrow \tilde{\sigma} \cdot \Tr_{\rho}(h_i)$. We claim that $\xi_1(\mu)$ is equal to $c_1 c_2 \cdot \beta_{h_i, \eta}$, where
\begin{multline*}
c_1 = 
\left( \frac{\hat{\alpha}([g \vert \mu \vert \sigma]) \hat{\alpha}([\tilde{\sigma} \vert h_i^{\pi(\tilde{\sigma})} \vert \eta])}{\hat{\alpha}([g \vert \tilde{\sigma} \vert \eta])} \right)
\left(\frac{\hat{\alpha}([\mu \vert \sigma \vert h_i^{\pi(\mu \sigma)}])}{\hat{\alpha}([\mu \vert g^{\pi(\mu)} \vert \sigma]) \hat{\alpha}([\tilde{\sigma} \vert \eta \vert h_i^{\pi(\mu \sigma)}])} \right)  \times \\
\left( \frac{\hat{\alpha}([g^{-1} \vert \sigma \vert h_i^{\pi(\sigma)}])}{\hat{\alpha}([g^{-1} \vert g \vert \sigma]) \hat{\alpha}([\sigma \vert h_i^{-\pi(\sigma)} \vert h_i^{\pi(\sigma)}])} \right)^{-\frac{\pi(\mu)-1}{2}}
\end{multline*}
and
\[
c_2 =
\hat{\alpha}([h_i \vert h_i^{-1} \vert h_i])^{\delta_{\pi(\mu),\pi(\sigma),-1} -\delta_{\pi(\tilde{\sigma}),\pi(\eta),-1}} 
\left( \frac{\hat{\alpha}([h_i^{-1} \vert h_i \vert \eta]) \hat{\alpha}([\eta \vert h_i^{-\pi(\eta)} \vert h_i^{\pi(\eta)}])}{\hat{\alpha}([h_i^{-1} \vert \eta \vert h_i^{\pi(\eta)}])} \right)^{-\frac{\pi(\tilde{\sigma})-1}{2}}.
\]
Indeed, the factor $c_1$ is due to the scalars relating $\RInd_{\hat{\mathcal{H}}}^{\hat{\mathcal{G}}}(\psi)_{g,\mu}^{-1}$, $\RInd_{\hat{\mathcal{H}}}^{\hat{\mathcal{G}}}(\psi)_{\mu, g^{\pi(\mu)}}$ and $
\RInd_{\hat{\mathcal{H}}}^{\hat{\mathcal{G}}}(\psi)^{\op}_{g^{-1}, g}$ to $\psi^{-1}_{h_i^{\pi(\tilde{\sigma})},\eta}$, $\psi_{\eta, h_i^{\pi(\mu \sigma)}}$ and $\psi_{h_i^{-\pi(\sigma)}, h_i^{\pi(\sigma)}}$, respectively. Note that $
\RInd_{\hat{\mathcal{H}}}^{\hat{\mathcal{G}}}(\psi)^{\op}_{g^{-1}, g}$ appears only when $\pi(\mu)=-1$. These maps contribute to $\xi_1(\mu)$ regardless of the degrees of $\sigma$ and $\tilde{\sigma}$. The factor $c_2$ is due the maps $F^{\pm 1}_i$ and is best understood using string diagrams. For example, when $\mu$ and $\eta$ are of degree $+1$ and $\sigma$ and $\tilde{\sigma}$ are of degree $-1$, we have
\[
\xi_1(\mu)(\phi)
=
c_1 \cdot
\begin{tikzpicture}[scale=0.25,color=black, baseline]
\draw [decoration={markings, mark=at position 0.575 with {\arrow{<}}}, decoration={markings, mark=at position 0.945 with {\arrow{>}}},        postaction={decorate}] (0,0) ellipse (3 and 5);
\node (n1) at (-0.17,1.9) [circle] {};
\node at (1.5,0.0) {$\scriptstyle \phi$};
\node at (-1.15,-4.55) [circle,draw,fill,inner sep=0.35mm] {};
\node at (-2.25,-3.5) [circle,draw,fill,inner sep=0.35mm] {};
\draw[decoration={markings, mark=at position 0.5 with {\arrow{<}}},        postaction={decorate}] (n1) -- (-2.25,-3.5);
\node at (-3.7,-2) {$\scriptstyle \eta$};
\node at (4.7,-1.9) {$\scriptstyle \eta^{-1}$};
\node at (-0.0,-1.5) {$\scriptstyle h_i^{-1}$};
\node at (-0.0,-6.5) {$\scriptstyle h_i^{-1}$};
\node at (-6.7,0) {$\scriptstyle h_i$};
\node at (1.5,2.5) {$\scriptstyle h_i$};
\draw[black, decoration={markings, mark=at position 0.895 with {\arrow{>}}},        postaction={decorate}](-3.2,-4.55) [partial ellipse=-180:0:2 and 2];
\draw[decoration={markings, mark=at position 0.5 with {\arrow{>}}},        postaction={decorate}] (-5.2,-4.55) -- (-6.0,5);
\draw[black, decoration={markings, mark=at position 0.795 with {\arrow{<}}},        postaction={decorate}](0.5,1.0) [partial ellipse=180:0:1 and 1];
\node at (1.5,1.0) [circle,draw,fill,inner sep=0.3mm] {};
\end{tikzpicture}
\]
This string diagram can be evaluated as follows:
\[
\begin{tikzpicture}[scale=0.25,color=black, baseline]
\draw [decoration={markings, mark=at position 0.575 with {\arrow{<}}}, decoration={markings, mark=at position 0.945 with {\arrow{>}}},        postaction={decorate}] (0,0) ellipse (3 and 5);
\node (n1) at (-0.17,1.9) [circle] {};
\node at (1.5,0.0) {$\scriptstyle \phi$};
\node at (-1.15,-4.55) [circle,draw,fill,inner sep=0.35mm] {};
\node at (-2.25,-3.5) [circle,draw,fill,inner sep=0.35mm] {};
\draw[decoration={markings, mark=at position 0.5 with {\arrow{<}}},        postaction={decorate}] (n1) -- (-2.25,-3.5);
\node at (-3.7,-2) {$\scriptstyle \eta$};
\node at (4.7,-1.9) {$\scriptstyle \eta^{-1}$};
\node at (-0.0,-1.5) {$\scriptstyle h_i^{-1}$};
\node at (-0.0,-6.5) {$\scriptstyle h_i^{-1}$};
\node at (-6.7,0) {$\scriptstyle h_i$};
\node at (1.5,2.5) {$\scriptstyle h_i$};
\draw[black, decoration={markings, mark=at position 0.895 with {\arrow{>}}},        postaction={decorate}](-3.2,-4.55) [partial ellipse=-180:0:2 and 2];
\draw[decoration={markings, mark=at position 0.5 with {\arrow{>}}},        postaction={decorate}] (-5.2,-4.55) -- (-6.0,5);
\draw[black, decoration={markings, mark=at position 0.795 with {\arrow{<}}},        postaction={decorate}](0.5,1.0) [partial ellipse=180:0:1 and 1];
\node at (1.5,1.0) [circle,draw,fill,inner sep=0.3mm] {};
\end{tikzpicture}
\xrightarrow[\scriptstyle \text{Eq. \eqref{diag:moveToVertex}}]{\hat{\alpha}([h_i^{-1} \vert \eta \vert \eta^{-1}])}
\begin{tikzpicture}[scale=0.25,color=black, baseline]
\draw [decoration={markings, mark=at position 0.575 with {\arrow{<}}}, decoration={markings, mark=at position 0.945 with {\arrow{>}}},        postaction={decorate}] (0,0) ellipse (3 and 5);
\node (n1) at (-0.17,1.9) [circle] {};
\node at (1.5,0.0) {$\scriptstyle \phi$};
\node at (-0,-5.0) [circle,draw,fill,inner sep=0.35mm] {};
\node at (-2.25,-3.5) [circle,draw,fill,inner sep=0.35mm] {};
\draw[decoration={markings, mark=at position 0.5 with {\arrow{<}}},        postaction={decorate}] (n1) -- (-2.25,-3.5);
\node at (-3.7,-2) {$\scriptstyle \eta$};
\node at (4.4,-1.9) {$\scriptstyle \eta^{-1}$};
\node at (-0.0,-1.5) {$\scriptstyle h_i^{-1}$};
\node at (1.5,-6.5) {$\scriptstyle h_i^{-1}$};
\node at (-6.7,0) {$\scriptstyle h_i$};
\node at (1.5,2.5) {$\scriptstyle h_i$};
\draw[black, decoration={markings, mark=at position 0.895 with {\arrow{>}}},        postaction={decorate}](-1.9,-4.95) [partial ellipse=-180:0:2 and 2];
\draw[decoration={markings, mark=at position 0.5 with {\arrow{>}}},        postaction={decorate}] (-3.95,-4.95) -- (-6.0,5);
\draw[black, decoration={markings, mark=at position 0.795 with {\arrow{<}}},        postaction={decorate}](0.5,1.0) [partial ellipse=180:0:1 and 1];
\node at (1.5,1.0) [circle,draw,fill,inner sep=0.3mm] {};
\end{tikzpicture}
\xrightarrow[\scriptstyle \text{Eq. \eqref{diag:associativity}}]{\hat{\alpha}([\eta \vert h_i^{-1} \vert \eta^{-1}])^{-1}}
\]
\[
\begin{tikzpicture}[scale=0.25,color=black, baseline]
\draw [decoration={markings, mark=at position 0.575 with {\arrow{<}}}, decoration={markings, mark=at position 0.945 with {\arrow{>}}},        postaction={decorate}] (0,0) ellipse (3 and 5);
\node (n1) at (-0.97,1.7) [circle] {};
\node at (1.5,0.0) {$\scriptstyle \phi$};
\node at (-0,-5.0) [circle,draw,fill,inner sep=0.35mm] {};
\node at (2.25,-3.5) [circle,draw,fill,inner sep=0.35mm] {};
\draw[decoration={markings, mark=at position 0.5 with {\arrow{<}}},        postaction={decorate}] (n1) -- (2.25,-3.5);
\node at (-3.7,-2) {$\scriptstyle \eta$};
\node at (4.7,-1.9) {$\scriptstyle \eta^{-1}$};
\node at (-0.7,-0.9) {$\scriptstyle h_i^{-1}$};
\node at (1.5,-6.5) {$\scriptstyle h_i^{-1}$};
\node at (-5,6) {$\scriptstyle h_i$};
\node at (1.5,2.5) {$\scriptstyle h_i$};
\draw[black, decoration={markings, mark=at position 0.895 with {\arrow{>}}},        postaction={decorate}](-1.9,-4.95) [partial ellipse=-180:0:2 and 2];
\draw[decoration={markings, mark=at position 0.5 with {\arrow{>}}},        postaction={decorate}] (-3.95,-4.95) -- (-6.0,5);
\draw[black,decoration={markings, mark=at position 0.795 with {\arrow{<}}}, postaction={decorate}](0.5,1.0) [partial ellipse=180:0:1 and 1];
\node at (1.5,1.0) [circle,draw,fill,inner sep=0.3mm] {};
\end{tikzpicture}
\xrightarrow[\scriptstyle \text{Eq. \eqref{diag:associativity}}]{\hat{\alpha}([h_i \vert \eta \vert \eta^{-1} h_i^{-1}])}
\begin{tikzpicture}[scale=0.25,color=black, baseline]
\draw [decoration={markings, mark=at position 0.575 with {\arrow{<}}}, decoration={markings, mark=at position 0.945 with {\arrow{>}}},        postaction={decorate}] (0,0) ellipse (3 and 5);
\node (n1) at (-0.97,1.7) [circle] {};
\node at (1.5,0.0) {$\scriptstyle \phi$};
\node at (-2.2,3.3) [circle,draw,fill,inner sep=0.35mm] {};
\node at (2.25,-3.5) [circle,draw,fill,inner sep=0.35mm] {};
\draw[decoration={markings, mark=at position 0.5 with {\arrow{<}}},        postaction={decorate}] (n1) -- (2.25,-3.5);
\node at (-4.1,-2) {$\scriptstyle \eta h_i$};
\node at (4.7,-1.9) {$\scriptstyle \eta^{-1}$};
\node at (-0.7,-0.9) {$\scriptstyle h_i^{-1}$};
\node at (-5,6) {$\scriptstyle h_i$};
\node at (1.5,2.5) {$\scriptstyle h_i$};
\draw[black, decoration={markings, mark=at position 0.795 with {\arrow{<}}},        postaction={decorate}](0.5,1.0) [partial ellipse=180:0:1 and 1];
\node at (1.5,1.0) [circle,draw,fill,inner sep=0.3mm] {};
\draw[decoration={markings, mark=at position 0.7 with {\arrow{>}}},        postaction={decorate}] (-2.2,3.3) -- (-3.5,5);
\end{tikzpicture}
\xrightarrow[\scriptstyle  \text{Eq. \eqref{diag:associativity}}]{\hat{\alpha}([\eta h_i \vert h_i^{-1} \vert \eta^{-1}])}
\]
\[
\begin{tikzpicture}[scale=0.25,color=black, baseline]
\draw [decoration={markings, mark=at position 0.575 with {\arrow{<}}}, decoration={markings, mark=at position 0.945 with {\arrow{>}}},        postaction={decorate}] (0,0) ellipse (3 and 5);
\node (n1) at (-0.17,1.9) [circle] {};
\node at (1.5,0.0) {$\scriptstyle \phi$};
\node at (-2.2,3.3) [circle,draw,fill,inner sep=0.35mm] {};
\node at (-2.25,-3.5) [circle,draw,fill,inner sep=0.35mm] {};
\draw[decoration={markings, mark=at position 0.5 with {\arrow{<}}},        postaction={decorate}] (n1) -- (-2.25,-3.5);
\node at (-4.0,-2) {$\scriptstyle \eta h_i$};
\node at (4.7,-1.9) {$\scriptstyle \eta^{-1}$};
\node at (-0.0,-1.5) {$\scriptstyle h_i^{-1}$};
\node at (1.0,2.6) {$\scriptstyle h_i$};
\node at (-5,6) {$\scriptstyle h_i$};
\draw[black, decoration={markings, mark=at position 0.795 with {\arrow{<}}},        postaction={decorate}](0.5,1.0) [partial ellipse=180:0:1 and 1];
\node at (1.5,1.0) [circle,draw,fill,inner sep=0.3mm] {};
\draw[decoration={markings, mark=at position 0.7 with {\arrow{>}}},        postaction={decorate}] (-2.2,3.3) -- (-3.5,5);
\end{tikzpicture}
\xrightarrow[\scriptstyle \text{Eq. \eqref{diag:moveToVertex}}]{\hat{\alpha}([\eta \vert h_i \vert h_i^{-1}])^{-1}}
\begin{tikzpicture}[scale=0.25,color=black, baseline]
\draw [decoration={markings, mark=at position 0.575 with {\arrow{<}}}, decoration={markings, mark=at position 0.945 with {\arrow{>}}},        postaction={decorate}] (0,0) ellipse (3 and 5);
\node (n1) at (-0.17,1.9) [circle] {};
\node at (1.5,0.0) {$\scriptstyle \phi$};
\node at (-2.2,3.3) [circle,draw,fill,inner sep=0.35mm] {};
\node at (-2.75,2) [circle,draw,fill,inner sep=0.35mm] {};
\node at (-4.0,-2) {$\scriptstyle \eta$};
\node at (4.7,-1.9) {$\scriptstyle \eta^{-1}$};
\node at (1.0,2.6) {$\scriptstyle h_i$};
\node at (-5,6) {$\scriptstyle h_i$};
\draw[black, decoration={markings, mark=at position 0.795 with {\arrow{<}}},        postaction={decorate}](0.6,1.0) [partial ellipse=180:0:0.85 and 0.85];
\draw[black](-1.1,1.0) [partial ellipse=180:360:0.8 and 0.8];
\draw[] (-1.85,0.95) -- (-2.75,2);
\node at (1.5,1.0) [circle,draw,fill,inner sep=0.3mm] {};
\draw[decoration={markings, mark=at position 0.7 with {\arrow{>}}},        postaction={decorate}] (-2.2,3.3) -- (-3.5,5);
\end{tikzpicture}
\xrightarrow[\scriptstyle  \text{Eq. \eqref{diag:snakeRemoval}}]{\hat{\alpha}([h_i \vert h_i^{-1} \vert h_i])^{-1}}
\begin{tikzpicture}[scale=0.25,color=black, baseline]
\draw [decoration={markings, mark=at position 0.575 with {\arrow{<}}}, decoration={markings, mark=at position 0.945 with {\arrow{>}}},        postaction={decorate}] (0,0) ellipse (3 and 5);
\node (n1) at (0,1.0) [circle,draw,fill,inner sep=0.35mm] {};
\node[below=0.01mm of n1] {$\scriptstyle \phi$};
\node at (-2.75,2) [circle,draw,fill,inner sep=0.35mm] {};
\node at (-2.2,3.3) [circle,draw,fill,inner sep=0.35mm] {};
\draw[decoration={markings, mark=at position 0.7 with {\arrow{>}}},        postaction={decorate}] (-2.2,3.3) -- (-3.5,5);
\draw[decoration={markings, mark=at position 0.5 with {\arrow{>}}},        postaction={decorate}] (n1) -- (-2.75,2);
\node at (-4.5,-2) {$\scriptstyle \eta$};
\node at (4.7,-1.9) {$\scriptstyle \eta^{-1}$};
\node at (-1.0,2.5) {$\scriptstyle h_i$};
\node at (-3.5,6) {$\scriptstyle h_i$};
\end{tikzpicture}
\]
so that, after a short calculation, we arrive at
\[
\begin{tikzpicture}[scale=0.25,color=black, baseline]
\draw [decoration={markings, mark=at position 0.575 with {\arrow{<}}}, decoration={markings, mark=at position 0.945 with {\arrow{>}}},        postaction={decorate}] (0,0) ellipse (3 and 5);
\node (n1) at (-0.17,1.9) [circle] {};
\node at (1.5,0.0) {$\scriptstyle \phi$};
\node at (-1.15,-4.55) [circle,draw,fill,inner sep=0.35mm] {};
\node at (-2.25,-3.5) [circle,draw,fill,inner sep=0.35mm] {};
\draw[decoration={markings, mark=at position 0.5 with {\arrow{<}}},        postaction={decorate}] (n1) -- (-2.25,-3.5);
\node at (-3.7,-2) {$\scriptstyle \eta$};
\node at (4.7,-1.9) {$\scriptstyle \eta^{-1}$};
\node at (-0.0,-1.5) {$\scriptstyle h_i^{-1}$};
\node at (-0.0,-6.5) {$\scriptstyle h_i^{-1}$};
\node at (-6.7,0) {$\scriptstyle h_i$};
\node at (1.5,2.5) {$\scriptstyle h_i$};
\draw[black, decoration={markings, mark=at position 0.895 with {\arrow{>}}},        postaction={decorate}](-3.2,-4.55) [partial ellipse=-180:0:2 and 2];
\draw[decoration={markings, mark=at position 0.5 with {\arrow{>}}},        postaction={decorate}] (-5.2,-4.55) -- (-6.0,5);
\draw[black, decoration={markings, mark=at position 0.795 with {\arrow{<}}},        postaction={decorate}](0.5,1.0) [partial ellipse=180:0:1 and 1];
\node at (1.5,1.0) [circle,draw,fill,inner sep=0.3mm] {};
\end{tikzpicture}
=
\frac{\hat{\alpha}([h_i^{-1} \vert h_i \vert \eta]) \hat{\alpha}([\eta \vert h_i^{-1} \vert h_i])}{\hat{\alpha}([h_i^{-1} \vert \eta \vert h_i])} \beta_{h_i,\eta}(\phi) = c_2 \beta_{h_i, \eta}(\phi).
\]
Similarly, when $\mu, \sigma$ are of degree $+1$ and $\tilde{\sigma}, \eta$ are of degree $-1$, we have
\[
\xi_1(\mu)(\phi)
=
c_1 \cdot
\begin{tikzpicture}[scale=0.25,color=black, baseline]
\draw [decoration={markings, mark=at position 0.575 with {\arrow{<}}}, decoration={markings, mark=at position 0.945 with {\arrow{>}}},        postaction={decorate}] (0,0) ellipse (3 and 5);
\node (n1) at (-0.17,1.6) [circle,draw,fill,inner sep=0.3mm] {};
\node at (-0.5,2.5) {$\scriptstyle \phi$};
\node at (-1.15,-4.55) [circle,draw,fill,inner sep=0.35mm] {};
\node at (-2.25,-3.5) [circle,draw,fill,inner sep=0.35mm] {};
\draw[decoration={markings, mark=at position 0.5 with {\arrow{<}}},        postaction={decorate}] (n1) -- (-2.25,-3.5);
\node at (-3.7,-2) {$\scriptstyle \eta$};
\node at (4.7,-1.9) {$\scriptstyle \eta^{-1}$};
\node at (0.2,-1.5) {$\scriptstyle h_i$};
\node at (-0.0,-6.5) {$\scriptstyle h_i^{-1}$};
\node at (-6.7,0) {$\scriptstyle h_i$};
\draw[black, decoration={markings, mark=at position 0.895 with {\arrow{>}}},        postaction={decorate}](-3.2,-4.55) [partial ellipse=-180:0:2 and 2];
\draw[decoration={markings, mark=at position 0.5 with {\arrow{>}}},        postaction={decorate}] (-5.2,-4.55) -- (-6.0,5);
\end{tikzpicture}
\]
This string diagram can be evaluated in the same way as the previous diagram, the main difference being that the final step is not required. This leads to an additional factor of $\hat{\alpha}([h_i \vert h_i^{-1} \vert h_i])^{-1}$:
\[
\begin{tikzpicture}[scale=0.25,color=black, baseline]
\draw [decoration={markings, mark=at position 0.575 with {\arrow{<}}}, decoration={markings, mark=at position 0.945 with {\arrow{>}}},        postaction={decorate}] (0,0) ellipse (3 and 5);
\node (n1) at (-0.17,1.6) [circle,draw,fill,inner sep=0.3mm] {};
\node at (-0.5,2.5) {$\scriptstyle \phi$};
\node at (-1.15,-4.55) [circle,draw,fill,inner sep=0.35mm] {};
\node at (-2.25,-3.5) [circle,draw,fill,inner sep=0.35mm] {};
\draw[decoration={markings, mark=at position 0.5 with {\arrow{<}}},        postaction={decorate}] (n1) -- (-2.25,-3.5);
\node at (-3.7,-2) {$\scriptstyle \eta$};
\node at (4.7,-1.9) {$\scriptstyle \eta^{-1}$};
\node at (0.2,-1.5) {$\scriptstyle h_i$};
\node at (-0.0,-6.5) {$\scriptstyle h_i^{-1}$};
\node at (-6.7,0) {$\scriptstyle h_i$};
\draw[black, decoration={markings, mark=at position 0.895 with {\arrow{>}}},        postaction={decorate}](-3.2,-4.55) [partial ellipse=-180:0:2 and 2];
\draw[decoration={markings, mark=at position 0.5 with {\arrow{>}}},        postaction={decorate}] (-5.2,-4.55) -- (-6.0,5);
\end{tikzpicture}
=
\hat{\alpha}([h_i \vert h_i^{-1} \vert h_i])^{-1} \frac{\hat{\alpha}([h_i^{-1} \vert h_i \vert \eta]) \hat{\alpha}([\eta \vert h_i \vert h_i^{-1}])}{\hat{\alpha}([h_i^{-1} \vert \eta \vert h_i^{-1}])}
\beta_{h_i, \eta}(\phi).
\]
The coefficient of $\beta_{h_i, \eta}(\phi)$ is again $c_2$. The remaining cases are dealt with similarly.

On the other hand, as $(\sigma_i^{-1} \mu \sigma_i)(\sigma_i^{-1} \sigma) = (\sigma_i^{-1} \tilde{\sigma}) \eta$, the results of Section \ref{sec:complexInd} imply that $\sigma_i^{-1} \mu \sigma_i \in \mathsf{Z}_{\hat{\mathsf{G}}}^{\varphi}(h_i)$ acts on
\[
\Ind_{\mathsf{Z}_{\hat{\mathsf{H}}}^{\varphi}(h_i)}^{\mathsf{Z}_{\hat{\mathsf{G}}}^{\varphi}(h_i)} \left( \Tr_{\rho}(h_i) \right) \simeq \bigoplus_{\sigma \in \mathcal{S}_i} \sigma_i^{-1} \sigma \cdot \Tr_{\rho}(h_i)
\]
by the linear map $\xi_2(\sigma_i^{-1} \mu \sigma_i): \sigma_i^{-1} \sigma\cdot \Tr_{\rho}(h_i) \rightarrow \sigma_i^{-1} \tilde{\sigma} \cdot \Tr_{\rho}(h_i)$ given by
\[
\xi_2(\sigma_i^{-1} \mu \sigma_i) = \frac{\theta_{h_i}([\sigma_i^{-1} \mu \sigma_i \vert \sigma_i^{-1} \sigma])}{\theta_{h_i}([\sigma_i^{-1} \tilde{\sigma} \vert \eta])} \beta_{h_i,\eta}.
\]
Here we have set $\theta_{\gamma}([\omega_2 \vert \omega_1]) = \uptau_{\pi}^{\refl}(\hat{\alpha})([\omega_2 \vert \omega_1]\gamma)$. Noting that
\[
\xi_2(\mu) =
\frac{\theta_{g}([\sigma_i^{-1} \vert \mu])}{\theta_{g}([\sigma_i^{-1} \mu \sigma_i \vert \sigma_i^{-1}])} \xi_2(\sigma_i^{-1} \mu \sigma_i),
\]
closedness of $\uptau_{\pi}^{\refl}(\hat{\alpha})$ then gives
\[
\xi_2(\mu) = \frac{\theta_{h_i}([\mu \vert \sigma])}{\theta_{h_i}([\tilde{\sigma} \vert \eta])} \frac{\theta_{h_i}([\sigma_i^{-1} \vert \tilde{\sigma}])}{\theta_{h_i}([\sigma_i^{-1} \vert \sigma])} \beta_{h_i,\eta}.
\]
The explicit expression for $\uptau_{\pi}^{\refl}(\hat{\alpha})$ shows that $\frac{\theta_{h_i}([\mu \vert \sigma]) }{\theta_{h_i}([\tilde{\sigma} \vert \eta])}$ is equal to $c_1 c_2$ above. We therefore arrive at the equality
\[
\xi_2(\mu) = \frac{\theta_{h_i}([\sigma_i^{-1} \vert \tilde{\sigma}])}{\theta_{h_i}([\sigma_i^{-1} \vert \sigma])} \xi_1(\mu).
\]
In other words, the diagram
\[
\begin{tikzpicture}[baseline= (a).base]
\node[scale=1] (a) at (0,0){
\begin{tikzcd}[row sep=large, column sep = huge]
\sigma \cdot \Tr_{\rho}(h_i) \arrow{r}{\xi_1(\mu)} \arrow{d}[left]{\uptau_{\pi}^{\refl}(\hat{\alpha})([\sigma_i^{-1} \vert \sigma] h_i)} & \tilde{\sigma} \cdot \Tr_{\rho}(h_i) \arrow{d}{\uptau_{\pi}^{\refl}(\hat{\alpha})([\sigma_i^{-1} \vert \tilde{\sigma}] h_i)} \\ \sigma_i^{-1} \sigma \cdot \Tr_{\rho}(h_i) \arrow{r}[below]{\xi_2(\mu)} &  \sigma_i^{-1} \tilde{\sigma} \cdot \Tr_{\rho}(h_i)
\end{tikzcd}
};
\end{tikzpicture}
\]
commutes. The vertical (scalar multiplication) maps of this diagram assemble to define the desired isomorphism $\Tr_{\RInd_{\hat{\mathcal{H}}}^{\hat{\mathcal{G}}} (\rho)}(g) \xrightarrow[]{\sim} \bigoplus_{i=1}^l \Ind_{\mathsf{Z}_{\hat{\mathsf{H}}}^{\varphi}(h_i)}^{\mathsf{Z}_{\hat{\mathsf{G}}}^{\varphi}(g)} \left( \Tr_{\rho}(h_i) \right)$ of twisted representations of $\mathsf{Z}_{\hat{\mathsf{G}}}^{\varphi}(g)$.
\end{proof}

Theorem \ref{thm:Real2RealIndFunct} follows at once from Lemma \ref{lem:summandIdent}.

\begin{Thm}
\label{thm:RInd2Char}
The Real $2$-character of $\RInd_{\hat{\mathcal{H}}}^{\hat{\mathcal{G}}}(\rho)$ is
\[
\chi_{\RInd_{\hat{\mathcal{H}}}^{\hat{\mathcal{G}}}(\rho)}(g, \omega) = \frac{1}{2 \vert \mathsf{H} \vert} \sum_{\substack{\sigma \in \hat{\mathsf{G}} \\ \sigma(g, \omega) \sigma^{-1} \in \hat{\mathsf{H}}^2}} \uptau \uptau_{\pi}^{\refl}(\hat{\alpha})([\sigma]g \xrightarrow[]{\omega} g)^{-1} \cdot \chi_{\rho}(\sigma g^{\pi(\sigma)} \sigma^{-1}, \sigma \omega \sigma^{-1}).
\]
\end{Thm}

\begin{proof}
By Theorem \ref{thm:Real2RealIndFunct}, it suffices to compute the character of $\Ind_{\Lambda_{\pi}^{\refl} B \hat{\mathsf{H}}}^{\Lambda_{\pi}^{\refl} B \hat{\mathsf{G}}} \left( \Tr(\rho) \right)$. As the canonical functor $\Lambda_{\pi}^{\refl} B \hat{\mathsf{H}} \rightarrow \Lambda_{\pi}^{\refl} B \hat{\mathsf{G}}$ is faithful, this character can be computed using Proposition \ref{prop:twistedGrpdCharInd}.
\end{proof}

\subsection{Hyperbolic \texorpdfstring{$2$}{}-induction}
\label{sec:hyp2Ind}

We now turn to the categorification of hyperbolic induction. Let $\hat{\mathsf{G}}$ be a finite $\mathbb{Z}_2$-graded group. Fix $\hat{\alpha} \in Z^3(B \hat{\mathsf{G}}, k^{\times}_{\pi})$ and $\varsigma \in \hat{\mathsf{G}} \backslash \mathsf{G}$. Let $\rho$ be a linear representation of $\mathcal{G}=\mathcal{G}(\mathsf{G},\alpha)$ on $V$. The underlying category of $\HInd_{\mathcal{G}}^{\hat{\mathcal{G}}} (\rho)$ is $V \times V^{\op}$. The required $1$-morphisms are
\[
\HInd_{\mathcal{G}}^{\hat{\mathcal{G}}} (\rho)(g)= 
\left(
\begin{array}{cc}
\rho(g) & 0 \\
0& \rho(\varsigma^{-1} g \varsigma)^{\op}
\end{array}
\right), \qquad g \in \mathsf{G}
\]
and
\[
\HInd_{\mathcal{G}}^{\hat{\mathcal{G}}} (\rho)(\omega)= 
\left(
\begin{array}{cc}
0 & \rho(\omega \varsigma) \\
\rho(\varsigma^{-1} \omega)^{\op} & 0
\end{array}
\right), \qquad \omega \in \hat{\mathsf{G}} \backslash \mathsf{G}.
\]
The associativity $2$-isomorphisms are
\[
\HInd_{\mathcal{G}}^{\hat{\mathcal{G}}}(\psi)_{g_2,g_1} = 
\left(
\begin{array}{cc}
\psi_{g_2,g_1} & 0 \\
0 & \frac{\hat{\alpha}([g_2 \vert \varsigma \vert \varsigma^{-1} g_1 \varsigma])}{\hat{\alpha}([g_2 \vert g_1 \vert \varsigma]) \hat{\alpha}([\varsigma \vert \varsigma^{-1} g_2 \varsigma \vert \varsigma^{-1} g_1 \varsigma])} \psi_{\varsigma^{-1} g_2 \varsigma, \varsigma^{-1} g_1 \varsigma}^{- \op}
\end{array}
\right)
\]
and
\[
\HInd_{\mathcal{G}}^{\hat{\mathcal{G}}}(\psi)_{\omega_2,g_1} = 
\left(
\begin{array}{cc}
\frac{\hat{\alpha}([\omega_2 \vert \varsigma \vert \varsigma^{-1} g_1 \varsigma])}{\hat{\alpha}([\omega_2 \vert g_1 \vert \varsigma])}\psi_{ \omega_2 \varsigma, \varsigma^{-1} g_1 \varsigma} & 0 \\
0 & \hat{\alpha}([\varsigma \vert \varsigma^{-1} \omega_2 \vert g_1])^{-1} \psi_{\varsigma^{-1} \omega_2, g_1}^{- \op}
\end{array}
\right)
\]
and
\[
\HInd_{\mathcal{G}}^{\hat{\mathcal{G}}}(\psi)_{g_2,\omega_1} = 
\left(
\begin{array}{cc}
\hat{\alpha}([g_2 \vert \omega_1 \vert \varsigma])^{-1} \psi_{g_2, \omega_1 \varsigma} & 0 \\
0 & \frac{\hat{\alpha}([g_2 \vert \varsigma \vert \varsigma^{-1} \omega_1])}{\hat{\alpha}([\varsigma \vert \varsigma^{-1} g_2 \varsigma \vert \varsigma^{-1} \omega_1])}\psi_{\varsigma^{-1} g_2 \varsigma, \varsigma^{-1} \omega_1}^{- \op}
\end{array}
\right)
\]
and
\[
\HInd_{\mathcal{G}}^{\hat{\mathcal{G}}}(\psi)_{\omega_2,\omega_1} = 
\left(
\begin{array}{cc}
\hat{\alpha}([\omega_2 \vert \varsigma \vert \varsigma^{-1} \omega_1]) \psi_{\omega_2 \varsigma, \varsigma^{-1} \omega_1} & 0 \\
0 & \frac{1}{\hat{\alpha}([\omega_2 \vert \omega_1 \vert \varsigma]) \hat{\alpha}([\varsigma \vert \varsigma^{-1} \omega_2 \vert \omega_2 \varsigma])} \psi_{\varsigma^{-1} \omega_2 , \omega_1 \varsigma}^{- \op}
\end{array}
\right)
\]
where $g_1, g_2 \in \mathsf{G}$ and $\omega_1, \omega_2 \in \hat{\mathsf{G}} \backslash \mathsf{G}$. These expressions are most easily obtained by interpreting Real representations as homotopy fixed points of $\mathsf{Rep}_{\mathsf{Cat}_k,k}(\mathcal{G})$ and applying a categorified hyperbolic construction. In any case, it is straightforward to verify that this defines a linear Real representation of $\mathcal{G}$.

More generally, for a subgroup $\mathsf{H} \subset \mathsf{G}$, we define $\HInd_{\mathcal{H}}^{\hat{\mathcal{G}}} = \HInd_{\mathcal{G}}^{\hat{\mathcal{G}}} \circ \Ind_{\mathcal{H}}^{\mathcal{G}}$.

\subsection{Hyperbolically induced Real categorical and \texorpdfstring{$2$}{}-characters}
\label{sec:hyp2IndCar}

We compute the Real categorical and $2$-characters of $\HInd_{\mathcal{H}}^{\hat{\mathcal{G}}}(\rho)$. Since the method is similar to that of Section \ref{sec:Real2IndCar}, we will at points be brief.

\begin{Thm}
\label{thm:RealIndFunct}
There is a canonical isomorphism
\[
\Tr(\HInd_{\mathcal{H}}^{\hat{\mathcal{G}}}(\rho))
\simeq
\Ind_{\Lambda B \mathsf{H}}^{\Lambda_{\pi}^{\refl} B \hat{\mathsf{G}}} \left( \Tr(\rho) \right)
\]
of $\uptau_{\pi}^{\refl}(\hat{\alpha})$-twisted representations of $\Lambda_{\pi}^{\refl} B \hat{\mathsf{G}}$.
\end{Thm}

Standard properties of induction for representations of groupoids yield an isomorphism
\[
\Ind^{\Lambda_{\pi}^{\refl} B \mathsf{G}}_{\Lambda B \mathsf{H}} (\Tr_{\rho}(g)) \simeq \Ind^{\Lambda_{\pi}^{\refl} B \mathsf{G}}_{\Lambda B \mathsf{G}} \left( \Ind^{\Lambda B \mathsf{G}}_{\Lambda B \mathsf{H}} (\Tr_{\rho}(g)) \right)
\]
of $\uptau_{\pi}^{\refl}(\hat{\alpha})$-twisted representations. Since $\HInd_{\mathcal{H}}^{\hat{\mathcal{G}}} = \HInd_{\mathcal{G}}^{\hat{\mathcal{G}}} \circ \Ind_{\mathcal{H}}^{\mathcal{G}}$ and, by Theorem \ref{thm:Real2RealIndFunct}, we have
\[
\Tr(\Ind_{\mathcal{H}}^{\mathcal{G}} (\rho)) \simeq \Ind_{\Lambda B \mathsf{H}}^{\Lambda B \mathsf{G}} \left( \Tr(\rho) \right),
\]
it suffices to prove Theorem \ref{thm:RealIndFunct} under the assumption that $\mathsf{H}=\mathsf{G}$.

We proceed as in the proof of Theorem \ref{thm:Real2RealIndFunct}. Fix again an equivalence of the form \eqref{eq:unoriLoopDecomp}. Instead of \eqref{eq:RealConjDecomp} we consider a decomposition
\begin{equation}
\label{eq:conjDecomp}
\pser{g}_{\hat{\mathsf{G}}} = \bigsqcup_{i=1}^l [g_i]_{\mathsf{G}}
\end{equation}
with $[g_i]_{\mathsf{G}} \subset \mathsf{G}$ the conjugacy class of $g_i$. We have $l \in \{1,2\}$ according to whether or not the conjugacy class $[g_1]_{\mathsf{G}}$ is Real ($l=1$) or non-Real $(l=2)$. The Real and non-Real cases have $\mathsf{Z}_{\mathsf{G}}(g) \subsetneq \mathsf{Z}_{\hat{\mathsf{G}}}^{\varphi}(g)$ and $\mathsf{Z}_{\mathsf{G}}(g) = \mathsf{Z}_{\hat{\mathsf{G}}}^{\varphi}(g)$, respectively. There is an induced decomposition
\[
\mathcal{S} = \bigsqcup_{i=1}^l \mathcal{S}_i
\]
with $\mathcal{S}_i = \{ \sigma \in \mathcal{S} \mid \sigma^{-1} g^{\pi(\sigma)} \sigma \in [g_i]_{\mathsf{G}} \}$. We set $\sigma_1 =e$ and, in the non-Real case, $\sigma_2 = \varsigma$. Relabel the representatives of the conjugacy classes appearing in the decomposition \eqref{eq:conjDecomp} by $g_i = \sigma_i^{-1} g^{\pi(\sigma_i)} \sigma_i$.

The obvious analogue of Lemma \ref{lem:repChoice} holds by inspection. Explicitly, in the Real case we take $\varsigma$ to be any element of $\mathsf{Z}_{\hat{\mathsf{G}}}^{\varphi}(g) \backslash \mathsf{Z}_{\mathsf{G}}(g)$. The maps $F_i^{\pm}$ then yield an identification
\[
\Tr_{\HInd_{\mathcal{G}}^{\hat{\mathcal{G}}}(\rho)}(g) \simeq \Tr_{\rho}(g) \oplus \varsigma \cdot \Tr_{\rho}(g_2)
\]
where, by convention, $g_2=g$ in the Real case.

\begin{Lem}
\label{lem:summandIdentHyp}
There is an isomorphism
\[
\Tr_{\HInd_{\mathcal{G}}^{\hat{\mathcal{G}}} (\rho)}(g) \simeq \bigoplus_{i=1}^l \Ind_{\mathsf{Z}_{\mathsf{G}}(g_i)}^{\mathsf{Z}_{\hat{\mathsf{G}}}^{\varphi}(g)} \left( \Tr_{\rho}(g_i) \right)
\]
of $\uptau_{\pi}^{\refl}(\hat{\alpha})$-twisted representations of $\mathsf{Z}_{\hat{\mathsf{G}}}^{\varphi}(g)$, the induction being along the composition $\mathsf{Z}_{\mathsf{G}}(g_i) \hookrightarrow \mathsf{Z}_{\hat{\mathsf{G}}}^{\varphi}(g_i) \xrightarrow[]{l \mapsto \sigma_i l \sigma^{-1}_i} \mathsf{Z}_{\hat{\mathsf{G}}}^{\varphi}(g)$.
\end{Lem}

\begin{proof}
Consider first the non-Real case. Let $\mu \in \mathsf{Z}_{\hat{\mathsf{G}}}^{\varphi}(g) = \mathsf{Z}_{\mathsf{G}}(g)$ and $\sigma \in \mathcal{S}_i$. Then $
\mu \sigma = \tilde{\sigma} p$ for some $\tilde{\sigma} \in \mathcal{S}_i$ and $p \in \mathsf{Z}_{\mathsf{G}}(g_i)$. For $i=1$ we have $\sigma = e = \tilde{\sigma}$ and $\mu$ corresponds to the map $\xi_1(\mu) : \Tr_{\rho}(g) \rightarrow \Tr_{\rho}(g)$ given by $\beta_{g, p}$. For $i=2$ we have $\sigma= \varsigma = \tilde{\sigma}$  and $\mu$ induces a map $\xi_1(\mu) : \varsigma \cdot \Tr_{\rho}(g_2) \rightarrow \varsigma \cdot \Tr_{\rho}(g_2)$ which is equal to $c_1 c_2 \cdot \beta_{g_2, p}$, where
\[
c_1 = 
\left( \frac{\hat{\alpha}([g_2^{-1} \vert p \vert \varsigma^{-1}]) \hat{\alpha}([\varsigma^{-1} \vert g \vert \mu])}{\hat{\alpha}([g_2^{-1} \vert \varsigma^{-1} \vert \mu])} \right)
\left(\frac{\hat{\alpha}([p \vert \varsigma^{-1} \vert g])}{\hat{\alpha}([p \vert g_2^{-1} \vert \varsigma^{-1}]) \hat{\alpha}([\varsigma^{-1} \vert \mu \vert g])} \right)
\]
and
\[
c_2 =
\frac{\hat{\alpha}([g_2^{-1} \vert g_2 \vert p]) \hat{\alpha}([p \vert g_2^{-1} \vert g_2])}{\hat{\alpha}([g_2^{-1} \vert p \vert g_2])}.
\]
The factor $c_1$ is due to the scalars relating $\HInd_{\mathcal{G}}^{\hat{\mathcal{G}}}(\psi)_{g,\mu}^{-1}$ and $\HInd_{\mathcal{G}}^{\hat{\mathcal{G}}}(\psi)_{\mu, g}$ to $\psi^{\op}_{g_2^{-1},p}$ and $\psi_{p, g_2^{-1}}^{-\op}$, respectively, while $c_2$ is due to the maps $F^{\pm 1}_i$, as in Lemma \ref{lem:summandIdent}.

On the other hand, as $(\sigma_i^{-1} \mu \sigma_i)(\sigma_i^{-1} \sigma) = (\sigma_i^{-1} \tilde{\sigma}) p$, the element $\sigma_i^{-1} \mu \sigma_i \in \mathsf{Z}_{\mathsf{G}}(g_i)$ acts on $\Ind_{\mathsf{Z}_{\mathsf{G}}(g_i)}^{\mathsf{Z}_{\hat{\mathsf{G}}}^{\varphi}(g_i)} \left( \Tr_{\rho}(g_i) \right) = \Tr_{\rho}(g_i)$ by the map
\[
\xi_2(\sigma_i^{-1} \mu \sigma_i) = \frac{\theta_{g_i}([\sigma_i^{-1} \mu \sigma_i \vert \sigma_i^{-1} \sigma])}{\theta_{g_i}([\sigma_i^{-1} \tilde{\sigma} \vert p])} \beta_{g_i,p}.
\]
Again, we have set $\theta_{\gamma}([\omega_2 \vert \omega_1]) = \uptau_{\pi}^{\refl}(\hat{\alpha})([\omega_2 \vert \omega_1]\gamma)$. As in Lemma \ref{lem:summandIdent}, this leads to the expression
\[
\xi_2(\mu) = \frac{\theta_{g_i}([\sigma_i^{-1} \vert \tilde{\sigma}])}{\theta_{g_i}([\sigma_i^{-1} \vert \sigma])} \frac{\theta_{g_i}([\mu \vert \sigma])}{\theta_{g_i}([\tilde{\sigma} \vert p]) } \beta_{g_i,p}.
\]
Since $\sigma=\tilde{\sigma}$, this becomes $\xi_2(\mu)= \beta_{g,p}$ for $i=1$ and $\xi_2(\mu) = c_1 c_2 \cdot  \beta_{g_2, p}$ for $i=2$, where in the latter case we have used the explicit expression for $\uptau_{\pi}^{\refl}(\hat{\alpha})$.

Consider now the Real case. Let $\mu \in \mathsf{Z}_{\hat{\mathsf{G}}}^{\varphi}(g)$ and $\sigma \in \mathcal{S} = \mathcal{S}_1$. Then $\mu \sigma = \tilde{\sigma} p$ for some $\tilde{\sigma} \in \mathcal{S}$ and $p \in \mathsf{Z}_{\mathsf{G}}(g)$. It follows that the action of $\mu$ on $\Tr_{\HInd_{\mathcal{G}}^{\hat{\mathcal{G}}}(\rho)}(g)$ induces a linear map $\xi_1(\mu): \sigma \cdot \Tr_{\rho}(g) \rightarrow \tilde{\sigma} \cdot \Tr_{\rho}(g)$ of the form $c_1 c_2 \cdot \beta_{g,p}$. The factors $c_1$ and $c_2$, whose explicit forms we omit, arise in the same way as above. Note that $c_1$ may now receive contributions from $\HInd_{\mathcal{G}}^{\hat{\mathcal{G}}}(\psi)^{\op}_{g^{-1}, g}$.

On the other hand, since $\mu \sigma = \tilde{\sigma} p$, the action of $\mu \in \mathsf{Z}_{\mathsf{G}}(g)$ on $\Tr_{\HInd_{\mathcal{G}}^{\hat{\mathcal{G}}}(\rho)}(g)$ induces a linear map $\xi_2(\mu): \sigma \cdot \Tr_{\rho}(g) \rightarrow \tilde{\sigma}\cdot \Tr_{\rho}(g)$ given by
\[
\xi_2(\mu) = \frac{\uptau_{\pi}^{\refl}(\hat{\alpha})([\mu \vert \sigma] g)}{\uptau_{\pi}^{\refl}(\hat{\alpha})([\tilde{\sigma} \vert p] g)} \beta_{g,p}.
\]
For example, when $\pi(\mu)=-1$, $\sigma= \varsigma$ and $\tilde{\sigma}=e$, the coefficient of $\beta_{g,p}$ reads
\[
\uptau_{\pi}^{\refl}(\hat{\alpha})([\mu \vert \varsigma]g) = \hat{\alpha}([g \vert g^{-1} \vert g]) 
\frac{\hat{\alpha}([g^{-1} \vert \varsigma \vert g^{-1}])}{\hat{\alpha}([g^{-1} \vert g \vert \varsigma]) \hat{\alpha}([\varsigma \vert g \vert g^{-1}])} \frac{\hat{\alpha}([\mu \vert \varsigma \vert g]) \hat{\alpha}([g \vert \mu \vert \varsigma])}{\hat{\alpha}([\mu \vert g^{-1} \vert \varsigma])}.
\]
The factor $\hat{\alpha}([g \vert g^{-1} \vert g])$ is equal to $c_2$ while the remaining terms multiply to $c_1$. Note that $\HInd_{\mathcal{G}}^{\hat{\mathcal{G}}}(\psi)^{\op}_{g^{-1}, g}$ contributes a factor of $1$ to $c_1$ in this case. The other cases are treated in the same way.
\end{proof}

\begin{Thm}
\label{thm:HInd2Char}
The Real $2$-character of $\HInd_{\mathcal{H}}^{\hat{\mathcal{G}}} (\rho)$ is
\[
\chi_{\HInd_{\mathcal{H}}^{\hat{\mathcal{G}}} (\rho)}(g, \omega) = \frac{1}{\vert \mathsf{H} \vert} \sum_{\substack{\sigma \in \hat{\mathsf{G}} \\ \sigma(g, \omega) \sigma^{-1} \in \mathsf{H}^2}} \uptau \uptau_{\pi}^{\refl}(\hat{\alpha})([\sigma] g \xrightarrow[]{\omega} g)^{-1} \cdot \chi_{\rho}(\sigma g^{\pi(\sigma)} \sigma^{-1}, \sigma \omega \sigma^{-1}).
\]
In particular, $\chi_{\HInd_{\mathcal{H}}^{\hat{\mathcal{G}}} (\rho)}$ is supported on the subset $\pi_0(\Lambda^2 B \mathsf{G}) \subset \pi_0(\Lambda \Lambda_{\pi}^{\refl} B \hat{\mathsf{G}})$.
\end{Thm}

\begin{proof}
This is proved in the same way as Theorem \ref{thm:RInd2Char}, using Theorem \ref{thm:RealIndFunct} instead of Theorem \ref{thm:Real2RealIndFunct}. In the present case we apply Proposition \ref{prop:twistedGrpdCharInd} to the functor $\Lambda B \mathsf{H} \rightarrow \Lambda_{\pi}^{\refl} B \hat{\mathsf{G}}$. This explains the coefficient $\frac{1}{\vert \mathsf{H} \vert}$, as opposed to $\frac{1}{2 \vert \mathsf{H} \vert}$, in $\chi_{\HInd_{\mathcal{H}}^{\hat{\mathcal{G}}} (\rho)}$.
\end{proof}

\section{Conjectural applications to Real Hopkins--Kuhn--Ravenel characters}
\label{sec:RealHKRTransfer}

In this final section, motivated by the analogy between Borel equivariant Morava $E$-theory and the $2$-representation theory of finite groups \cite{ganter2008}, \cite{hopkins2000}, we speculate on a homotopy theoretic interpretation of our results.

Let $\mathsf{G}$ be a finite group. For each $n \geq 1$, denote by $\mathsf{G}^{(n)} \subset \mathsf{G}^n$ the subset of commuting $n$-tuples. The group $\mathsf{G}$ acts on $\mathsf{G}^{(n)}$ by simultaneous conjugation and the groupoid $\mathsf{G}^{(n)} \git \mathsf{G}$ is equivalent to the iterated loop groupoid $\Lambda^n B \mathsf{G}$. The space of locally constant functions on $\Lambda^n B \mathsf{G}$ valued in a ring $S$ is denoted by $\textnormal{Cl}_n(\mathsf{G},S)$ and is called the space of $S$-valued $n$-class functions on $\mathsf{G}$.

Fix a prime $p$ and let $E_n^{\bullet}$ be Morava $E$-theory at $p$. Let $\mathbf{B} \mathsf{G}$ be a classifying space of $\mathsf{G}$. Hopkins, Kuhn and Ravenel proved in \cite[Theorem C]{hopkins2000} that there is a generalized $n$-character map
\begin{equation}
\label{eq:HKRCharDescr}
E_n^{\bullet}(\mathbf{B} \mathsf{G}) \rightarrow \textnormal{Cl}_{n,p}(\mathsf{G},C_{\bullet})
\end{equation}
which, after tensoring the source with $C_{\bullet}$, is an isomorphism. Here $C_{\bullet}$ is a $p^{-1}E_n^{\bullet}$-algebra, constructed in \cite[\S 6.2]{hopkins2000}, which is an $E^{\bullet}_n$-theoretic analogue of the field obtained from $\mathbb{Q}$ by adjoining all roots of unity, the latter field being that which one considers in classical character theory. The subscript $p$ in $\textnormal{Cl}_{n,p}$ indicates that we restrict attention to functions defined on commuting $n$-tuples of $\mathsf{G}$ which have $p$\textsuperscript{th} power order. Given a subgroup $\mathsf{H} \hookrightarrow \mathsf{G}$, it was proved in \cite[Theorem D]{hopkins2000} that the induced transfer map is
\[
ind_{\mathsf{H}}^{\mathsf{G}}(\chi)(g_1, g_2, \dots, g_n) = \frac{1}{\vert \mathsf{H} \vert} \sum_{\substack{g \in \mathsf{G} \\ g (g_1, g_2, \dots, g_n) g^{-1} \in \mathsf{H}^n}} \chi(g g_1 g^{-1}, \dots, g g_n g^{-1}). 
\]
For $n \leq 3$ this is a $p$-completed version of the formula for induced $n$-characters; for $n=1$ this is classical, while for $n=2$ and $n=3$ this is proved in \cite[Corollary 7.6]{ganter2008} and \cite[Theorem 6.7]{wang2015}, respectively.

It is known from the work of Sati and Westerland \cite[Theorem 2]{sati2015} that a cohomology class $\alpha \in H^{n+2}(-, \mathbb{Z})$ defines a twisted cohomology group $E_n^{\bullet + \alpha}(-)$. In particular, for each $\alpha \in Z^3(B \mathsf{G}, \mathsf{U}(1))$ there is a twisted form of transfer
\[
E_n^{\bullet+ \alpha_{\vert \mathsf{H}}}(\mathbf{B} \mathsf{H}) \rightarrow E_n^{\bullet+\alpha}(\mathbf{B} \mathsf{G}).
\]
It is natural to expect that, after tensoring with $C_{\bullet}$, these cohomology groups are isomorphic to the $p$-completions of the twisted $n$-class function spaces $\Gamma_{\Lambda^n B \mathsf{H}}(\uptau^n(\alpha_{\vert \mathsf{H}})_{C_{\bullet}})$ and $\Gamma_{\Lambda^n B \mathsf{G}}(\uptau^n(\alpha)_{C_{\bullet}})$ and that under these identifications the transfer map is given by the obvious generalization of the formula appearing in Corollary \ref{cor:ind2Char}.

Let now $\hat{\mathsf{G}}$ be a finite $\mathbb{Z}_2$-graded group. For each $n \geq 1$, define
\[
\hat{\mathsf{G}}^{(n)} = \{(g, \omega_2, \dots, \omega_n) \in \mathsf{G} \times \hat{\mathsf{G}}^{n-1} \mid \omega_i g = g^{\pi(\omega_i)} \omega_i, \; \omega_i \omega_j = \omega_j \omega_i, \; 2 \leq i,j \leq n \}.
\]
The group $\hat{\mathsf{G}}$ acts on $\hat{\mathsf{G}}^{(n)}$ by Real conjugation on $\mathsf{G}$ and by ordinary conjugation on the remaining factors. The resulting groupoid $\hat{\mathsf{G}}^{(n)} \git_{\varphi} \hat{\mathsf{G}}$ is equivalent to $\Lambda^{n-1} \Lambda_{\pi}^{\refl} B \hat{\mathsf{G}}$. Denote by $\textnormal{Cl}\mathbb{R}_n(\mathsf{G}, S)$ the vector space of locally constant $S$-valued functions on $\Lambda^{n-1} \Lambda_{\pi}^{\refl} B \hat{\mathsf{G}}$, which we call $S$-valued Real $n$-class functions on $\mathsf{G}$

Assume now that $p=2$. In order to have a homotopy theoretic interpretation of our results, we would like to find a family of genuinely $\mathbb{Z}_2$-equivariant generalized cohomology theories $R_n^{\bullet}$, $n \geq 1$, which has the following properties:
\begin{enumerate}[label=(\roman*)]
\item For each $n \geq 1$, there are maps of cohomology theories
\[
c: R_n^{\bullet} \rightarrow E_n^{\bullet}, \qquad
r: E_n^{\bullet} \rightarrow R_n^{\bullet}.
\]

\item The assignment $\hat{\mathsf{G}} \mapsto R_n(\mathbf{B} \mathsf{G})$ extends to a $\mathbb{Z}_2$-graded generalization of a Mackey functor. In particular, associated to a $\mathbb{Z}_2$-graded subgroup $\hat{\mathsf{H}} \leq \hat{\mathsf{G}}$ are transfer and restriction maps,
\[
rind_{\hat{\mathsf{H}}}^{\hat{\mathsf{G}}}: R_n^{\bullet}(\mathbf{B} \mathsf{H}) \rightarrow R_n^{\bullet}(\mathbf{B} \mathsf{G}), \qquad res_{\hat{\mathsf{H}}}^{\hat{\mathsf{G}}}: R_n^{\bullet}(\mathbf{B} \mathsf{G}) \rightarrow R_n^{\bullet}(\mathbf{B} \mathsf{H})
\]
and associated to a subgroup $\mathsf{H} \leq \mathsf{G}$ is a map
\[
h_{\mathsf{H}}^{\hat{\mathsf{G}}}: E_n^{\bullet}(\mathbf{B} \mathsf{H}) \rightarrow R_n^{\bullet}(\mathbf{B} \mathsf{G}).
\]
The map $res_{\hat{\mathsf{H}}}^{\hat{\mathsf{G}}}$ is induced by the map $\mathbf{B}\hat{\mathsf{H}} \rightarrow \mathbf{B}\hat{\mathsf{G}}$ over $\mathbf{B}\mathbb{Z}_2$ while $h_{\mathsf{H}}^{\hat{\mathsf{G}}}$ is the composition of $c$ with the $E_n^{\bullet}$-theory transfer map $ind_{\mathsf{H}}^{\mathsf{G}}$. The above maps should be compatible in the obvious sense.

\item
There exists a $p^{-1} R_n^{\bullet}$-algebra $\widetilde{C}_{\bullet}$ such that for each finite $\mathbb{Z}_2$-graded group $\hat{\mathsf{G}}$ there is a Real $n$-character map
\begin{equation}
\label{eq:RealHKRCharDescr}
\chi: R_n^{\bullet}(\mathbf{B} \mathsf{G}) \rightarrow \textnormal{Cl}\mathbb{R}_{n,p}(\mathsf{G},\widetilde{C}_{\bullet})
\end{equation}
which induces an isomorphism $E\mathbb{R}_n^{\bullet}(\mathbf{B} \mathsf{G}) \otimes_{R^{\bullet}_n} \widetilde{C}_{\bullet} \xrightarrow[]{\sim} \textnormal{Cl}\mathbb{R}_{n,p}(\mathsf{G},\widetilde{C}_{\bullet})$. Moreover, under the identifications \eqref{eq:HKRCharDescr} and \eqref{eq:RealHKRCharDescr}, the maps $rind_{\hat{\mathsf{H}}}^{\hat{\mathsf{G}}}$ and $h_{\mathsf{H}}^{\hat{\mathsf{G}}}$ agree with the natural generalizations of the maps $\RInd_{\hat{\mathsf{H}}}^{\hat{\mathsf{G}}}$ and $\HInd_{\mathsf{H}}^{\hat{\mathsf{G}}}$, respectively, at the level of (Real) $n$-class functions.

\item There is an isomorphism $R_1^{\bullet} \simeq KR^{\bullet}_{(2)}$ under which $r$ and $c$ become the hyperbolic (realification) and forgetful (complexification) maps, respectively, of equivariant $KR$-theory \cite{atiyah1966}.

\item Let $\pi: X \rightarrow \hat{X}$ be a double cover. A class $\hat{\beta} \in H^{n+2}(\hat{X}, \mathbb{Z}_{\pi})$ defines a twisted theory $R_n^{\bullet +\hat{\beta}}(X)$ and the obvious analogues of properties (i)-(iv) hold for the twisted theory. For example, when twisting by a cocycle $\hat{\beta} \in Z^{n+1}(B \hat{\mathsf{G}}, \mathsf{U}(1)_{\pi})$ the group $\textnormal{Cl}\mathbb{R}_{n,p}(\mathsf{G},\widetilde{C}_{\bullet})$ is replaced by $\textnormal{Cl}\mathbb{R}^{\hat{\beta}}_{n,p}(\mathsf{G},\widetilde{C}_{\bullet})$, the space of flat sections of the line bundle $\uptau^{n-1} \uptau_{\pi}^{\refl}(\hat{\beta}) \rightarrow \Lambda^{n-1} \Lambda_{\pi}^{\refl} B \hat{\mathsf{G}}$ over $\widetilde{C}_{\bullet}$.
\end{enumerate}

We briefly comment on the second part of (iii). Using the results of Sections \ref{sec:RealInd}, \ref{sec:hypInd} and \ref{sec:Real2IndCar}, \ref{sec:hyp2IndCar}, which should be viewed as being of height $n=1$ and $n=2$, respectively, we can write down explicit descriptions of the maps $\RInd$ and $\HInd$ at the level of Real $n$-class functions. For example, given $\chi \in \textnormal{Cl}\mathbb{R}^{\hat{\beta}}_{n,p}(\mathsf{H},\widetilde{C}_{\bullet})$, the requirement is that the value of $rind_{\hat{\mathsf{H}}}^{\hat{\mathsf{G}}} (\chi)$ at $(g, \omega_2, \dots, \omega_n) \in \hat{\mathsf{G}}^{(n)}$ is
\begin{multline*}
\frac{1}{2 \vert \mathsf{H} \vert} \sum_{\substack{\sigma \in \hat{\mathsf{G}} \\ \sigma(g, \omega_2, \dots, \omega_n) \sigma^{-1} \in \hat{\mathsf{H}}^n}} \uptau^{n-1} \uptau_{\pi}^{\refl}(\hat{\beta})([\sigma](g, \omega_2, \dots, \omega_n))^{-1}
\\
\times \chi(\sigma g^{\pi(\sigma)} \sigma^{-1}, \sigma \omega_2 \sigma^{-1}, \dots, \sigma \omega_n \sigma^{-1}).
\end{multline*}

A first guess is that $R_n^{\bullet}$ is the theory $E\mathbb{R}^{\bullet}_n$ constructed by Hu and Kriz \cite{hu2001}. This theory satisfies properties (i), (ii), (iv) and likely (v). However, as the following example shows, when $n \geq 2$ the group $E\mathbb{R}_n^{\bullet}(\mathbf{B} \mathsf{G})$ is too small to satisfy (iii).

\begin{Ex}
Let $\mathsf{G} = \mathbb{Z}_2$ with Real structure $\hat{\mathsf{G}} = \mathbb{Z}_2 \times \mathbb{Z}_2$. Then $\hat{\mathsf{G}}$ acts trivially on $\hat{\mathsf{G}}^{(2)} = \mathsf{G} \times \hat{\mathsf{G}}$ and $\textnormal{Cl}\mathbb{R}_{2,p}(\mathsf{G},\widetilde{C}_{\bullet})$ is a free $\widetilde{C}_{\bullet}$-module of rank $2 \vert \mathsf{G} \vert^2 = 8$. On the other hand, \cite[Theorem 3.2]{kitchloo2008} gives an isomorphism
\[
E\mathbb{R}_2^{\bullet}(\mathbf{B} \mathbb{Z}_2) \simeq E\mathbb{R}_2^{\bullet} \pser{u} \slash ([2]_{\hat{F}}(u))
\]
where $[2]_{\hat{F}}(u)$ a modified $2$-series of the formal group law associated to $E_2^{\bullet}$,
\[
[2]_{\hat{F}}(u) = 2 u +_F \alpha u^2 +_F u^4.
\]
Here $u \in E\mathbb{R}_2^{-16}(\mathbf{B} \mathbb{Z}_2)$ and $\alpha \in E\mathbb{R}_2^{16}(\mathbf{B} \mathbb{Z}_2)$. It follows that $E\mathbb{R}_2^{\bullet}(\mathbf{B} \mathbb{Z}_2)$ is a free $E\mathbb{R}_2^{\bullet}$-module of rank $4$. 
\end{Ex}

When $\mathsf{G} = \mathbb{Z}_{2^q}$ with trivial Real structure, the map $E\mathbb{R}^{\bullet}_2(\mathbf{B} \mathsf{G}) \rightarrow E^{\bullet}_2(\mathbf{B} \mathsf{G})$ has certain injectivity properties. See, for example, \cite[Theorem 2.1]{kitchloo2008b} and \cite[Theorem 1.5(iii)]{kitchloo2018}. In contrast, the restriction $\textnormal{Cl}\mathbb{R}_2(\mathsf{G}, S) \rightarrow \textnormal{Cl}_2(\mathsf{G},S)$ always has a non-trivial kernel, namely those Real $2$-class functions which are supported on $\mathsf{G} \times (\hat{\mathsf{G}} \backslash \mathsf{G})$. In particular, $E \mathbb{R}_2^{\bullet}(\mathbf{B} \mathsf{G})$ seems to miss the `Klein bottle sector' of $\textnormal{Cl}\mathbb{R}_2(\mathsf{G}, S)$.

\setcounter{secnumdepth}{0}

\footnotesize

\bibliographystyle{plain}
\bibliography{mybib}


\end{document}